\documentclass[oneside,11pt]{amsart}
\usepackage{amsmath, amsfonts, amsthm, amssymb}
\usepackage{mathrsfs}
\usepackage{setspace}
\usepackage{enumitem}
\usepackage{color}
\usepackage{tikz}
\usetikzlibrary{cd,arrows,patterns}
\usepackage{hyperref}

\usepackage[all]{xy}
\xyoption{matrix}
\xyoption{arrow}

\usepackage[margin=1in]{geometry}

\newtheorem{thm}{Theorem}[subsection]
\newtheorem{lem}[thm]{Lemma}
\newtheorem{cor}[thm]{Corollary}
\newtheorem{prop}[thm]{Proposition}

\newtheorem{innercustomthm}{Theorem}
\newenvironment{customthm}[1]
{\renewcommand\theinnercustomthm{#1}\innercustomthm}{\endinnercustomthm}
 
\theoremstyle{definition}
\newtheorem{defn}[thm]{Definition}
\newtheorem{eg}[thm]{Example}
\newtheorem{nota}[thm]{Notation}
\newtheorem{rem}[thm]{Remark}
 
\numberwithin{equation}{section}
 
\DeclareMathOperator{\im}{im}
\DeclareMathOperator{\coker}{coker}
\DeclareMathOperator{\Hom}{Hom}
\DeclareMathOperator{\Ext}{Ext}
\DeclareMathOperator{\End}{End}
\DeclareMathOperator{\ind}{ind}
\DeclareMathOperator{\add}{add}
\DeclareMathOperator{\tp}{tp}
\DeclareMathOperator{\pc}{pc}

\DeclareMathOperator{\Reg}{reg}

\DeclareMathOperator{\Aut}{Aut}
\newcommand\undim{\underline\dim\,}

\newcommand{\field}[1]{\mathbb{#1}}
\newcommand{\ZZ}{\ensuremath{{\field{Z}}}}

\newcommand{\RR}{\ensuremath{{\field{R}}}}

\newcommand{\NN}{\ensuremath{{\field{N}}}}

\newcommand{\cC}{\ensuremath{{\mathcal{C}}}}

\newcommand{\cF}{\ensuremath{{\mathcal{F}}}}

\newcommand{\cM}{\ensuremath{{\mathcal{M}}}}
\newcommand{\cP}{\ensuremath{{\mathcal{P}}}}
\newcommand{\cQ}{\ensuremath{{\mathcal{Q}}}}

\newcommand{\cS}{\ensuremath{{\mathcal{S}}}}
\newcommand{\cT}{\ensuremath{{\mathcal{T}}}}

\newcommand{\cX}{\ensuremath{{\mathcal{X}}}}
\newcommand{\cY}{\ensuremath{{\mathcal{Y}}}}

\newcommand{\fX}{\ensuremath{{\mathfrak{X}}}}
\newcommand{\fY}{\ensuremath{{\mathfrak{Y}}}}

\newcommand{\fD}{\ensuremath{{\mathfrak{D}}}}
\newcommand{\fL}{\ensuremath{{\mathfrak{L}}}}
\newcommand{\fC}{\ensuremath{{\mathfrak{C}}}}

\renewcommand{\Im}{\ensuremath{{\mathsf{Im}}}}

\newcommand{\mods}{\mathsf{mod}}
\newcommand{\proj}{\mathsf{proj}}
\newcommand{\rad}{\mathsf{rad}}

\title{Infinitesimal Semi-invariant Pictures and Co-amalgamation}

\author{Eric J. Hanson}
\address{LACIM, Universit\'e du Qu\'ebec \`a Montr\'eal, Montreal, QC H2L 2C4, CANADA\newline \indent D\'epartement de math\'ematiques, Universit\'e de Sherbrooke, Sherbrooke QC J1K 2R1, CANADA}
\email{ejhanso3@ncsu.edu}

\author{Kiyoshi Igusa}
\address{Department of Mathematics, Brandeis University, Waltham MA 02453, USA}\email{igusa@brandeis.edu}

\author{Moses Kim}
 \address{Department of Mathematics, Brandeis University, Waltham MA 02454, USA}\email{moseskim@brandeis.edu}

\author{Gordana Todorov}
\address{Department of Mathematics, Northeastern University, Boston MA 02115, USA}\email{g.todorov@neu.edu}

\date{July 7, 2023}

\subjclass[2020]{
16G20.\\
\
\\
\indent This is the peer reviewed version of the following article: [E. J. Hanson, K. Igusa, M. Kim, and G. Todorov, {\it Infinitesimal semi-invariant pictures and co-amalgamation}, J. Lon. Math. Soc. (2) {\bf 109} (2024), no. 1], which has been published in final form at \url{https://doi.org/10.1112/jlms.12786}. This article is available under a Creative Commons CC BY-NC 4.0 license. \copyright \ the author(s).\\}

\begin{document}
\maketitle

\begin{abstract}
	The purpose of this paper is to study the local structure of the semi-invariant picture of a tame hereditary algebra near the null root. Using a construction that we call co-amalgamation, we show that this local structure is completely described by the semi-invariant pictures of a collection of self-injective Nakayama algebras. We then describe the cones of this local structure using cluster-like structures that we call support regular clusters. Finally, we show that the local structure is (piecewise linearly) invariant under cluster tilting.
\end{abstract}

\tableofcontents

%%%%%%%%%%%%%%%%%%%%%%%%%%%%%%%%%%%%%%%%

\section{Introduction}

By ``infinitesimal semi-invariant picture'' we mean the germ at the $g$-vector of the null root of the semi-invariant picture for a tame hereditary algebra and also for cluster-tilted algebras of tame type. Semistability and the study of semi-invariants of quivers began with the work of Schofield \cite{schofield_semi} and King \cite{king_moduli}. This was then extended to quivers with relations in \cite{DW_semi}. In the hereditary case, semi-invariant pictures were first defined in \cite{IOTW_modulated}, building off of the theory developed in \cite{IOTW_cluster}. Similar ``pictures'' are used by Gross-Hacking-Keel-Kontsevich in \cite{GHKK_canonical} under the name \emph{cluster scattering diagrams}. Bridgeland algebraically defines such a scattering diagram for arbitrary finite dimensional algebras in \cite{bridgeland_scattering}, which is isomorphic to the cluster scattering diagram in the hereditary case\footnote{The recent paper \cite{mou_scattering} generalizes this result to any cluster-tilted algebra which admits a \emph{green-to-red sequence.}}. 

Semistability conditions define the walls of what is known as the \emph{wall-and-chamber structure} of an algebra. In \cite{BST_wall,asai}, Br\"ustle, Smith, and Treffinger and Asai study the connection between the {wall-and-chamber structure} defined in \cite[Section 6]{bridgeland_scattering} and the \emph{polyhedral fan of $\tau$-tilting pairs} of \cite{DIJ_tilting} (see also \cite{AIR_tilting}) for arbitrary finite dimensional algebras. 

Recall that every finite dimensional hereditary algebra has a cluster category, introduced in \cite{BMRRT_tilting}. In this category, we have cluster tilting objects which come equipped with a \emph{mutation} that mirrors the combinatorics of the mutation of clusters for a cluster algebra \cite{FZ_IV}. In order to generalize these combinatorics to the non-hereditary case, $\tau$-tilting theory is introduced in \cite{AIR_tilting}. A basic object $X = M \oplus P[1]$ with $M \in \mods\Lambda$ and $P \in \proj \Lambda$ is called \emph{support $\tau$-rigid} if $M$ is $\tau$-rigid, i.e., $\Hom(M,\tau M)=0$, and $\Hom_\Lambda(P,M) = 0$. A support $\tau$-rigid object $X$ is called \emph{support $\tau$-tilting} if it is not a proper direct summand of any other (basic) support $\tau$-rigid object.
An algebra is called \emph{$\tau$-tilting finite} if there are only finitely many isomorphism classes of support $\tau$-tilting objects for $\Lambda$ or, equivalently (see \cite{DIJ_tilting}), if $\mods\Lambda$ contains only finitely many bricks (up to isomorphism).
 The main result of \cite{BST_wall} associates to every $\tau$-tilting pair a unique chamber in this wall-and-chamber structure. This association is shown to be a bijection in \cite{asai}.

When an algebra admits infinitely many $\tau$-tilting pairs, the walls of the semi-invariant picture can accumulate (see e.g. \cite{paquette_accumulation}). In the case of a tame hereditary algebra or tame cluster-tilted algebra, there is a unique ``limiting wall'' corresponding to the null root. In the present paper, we study the infinitesimal structure of the semi-invariant picture near this limiting wall. One of the surprising features is that the chambers in the infinitesimal structure are not always simplicial. See Figure \ref{fig:regularPics2}.

%%%%%%%%%%%%%%%%%%%%%%%%%%%%%%%%%%%%%%%%%%%%%%%%%%%%%%

\subsection{Organization and main results}
The organization of this paper is as follows. 

In Section \ref{sec:geometryBackground}, we discuss polyhedral fans and a general notion of a wall-and-chamber structure (Definition~\ref{def:wallchamber}) from a geometric viewpoint. In particular, we define co-amalgamated products (Definition~\ref{def:co-amalgamation}), which play a central role in Section~\ref{sec:co-amalgamation}.

In the remainder of Section \ref{sec:background}, we provide background material and references on $g$-vectors, cluster- and $\tau$-tilting theory, the (standard) wall-and-chamber structure of an algebra and semi-invariant pictures, tame hereditary algebras, and cluster-tilted algebras.

In Section \ref{sec:stability}, we define the \emph{infinitesimal} (at some $v \in \RR^n$) and \emph{$v^\perp$-semi-invariant domains} (Definition \ref{def:infinitesimal stability}) for arbitrary finite dimensional algebras. As a special case of $v^\perp$-semi-invariant domains, we define \emph{regular semi-invariant domains} (Definition \ref{def:regular stability}) for tame hereditary and tame cluster-tilted algebras, taking $v$ to be the $g$-vector of the null root. The exceptional case of the Kronecker quiver requires special attention. Each of these notions defines a new wall-and-chamber structure. We then prove our first main theorem:

\begin{customthm}{A}[Theorem \ref{thmA} and Proposition~\ref{prop:nullPics}]\label{thmintro:mainA}
	Let $\Lambda$ be an arbitrary finite dimensional algebra and let $n$ be the number of (isoclasses of) simple $\Lambda$-modules. Then for $0 \neq v \in \RR^n$, the $v^\perp$ wall-and-chamber structure is equal to the infinitesimal wall-and-chamber structure at $v$. In particular, for $\Lambda$ tame hereditary or tame cluster-tilted, the regular semi-invariant picture is equal to the infinitesimal semi-invariant picture at the $g$-vector of the null root.
\end{customthm}

 To conclude Section~\ref{sec:stability}, we discuss the connection between Theorem~\ref{thmintro:mainA} and Asai's reduction of wall-and-chamber structures \cite[Section~4.1]{asai}.

In Section \ref{sec:co-amalgamation}, we associate a self-injective Nakayama algebra $\Lambda_{r_i}$ to each exceptional tube of a tame hereditary algebra. We then compute the standard wall-and-chamber structures of these self-injective Nakayama algebras and use \emph{co-amalgamation} (Definition~\ref{def:co-amalgamation}) to relate these to the regular semi-invariant picture of the hereditary algebra as follows:

\begin{customthm}{B}[Theorem \ref{thmB}]\label{thmintro:mainB}
	Let $\Lambda$ be a tame hereditary algebra with exceptional tubes of ranks $r_1,\ldots,r_m$. Then the regular wall-and-chamber structure of $\Lambda$ is linearly isomorphic to the co-amalgamated product of the standard wall-and-chamber structures of self-injective Nakayama algebras $\Lambda_{r_i}$ (associated to the exceptional tubes) along the hyperplanes perpendicular to $\overline{1} = (1,\ldots,1)$.
\end{customthm}

A special argument is needed for the Kronecker which has no exceptional tubes.
In Section \ref{sec:regular rigid}, we define \emph{support regular rigid objects} and \emph{support regular clusters} (Definition \ref{def:regular rigid}) for a tame hereditary algebra and compare them to sets of support $\tau$-rigid objects for the self-injective Nakayama algebras $\Lambda_{r_i}$ (Section \ref{sec:combining chambers}). We show that the support regular rigid objects which are \emph{projectively closed} (Definition \ref{def:projectively closed}) define a polyhedral fan (Proposition \ref{prop:proj closed fan}). This is used to prove our third main theorem:

\begin{customthm}{C}[Theorem \ref{thmC} and Corollary \ref{cor:walls}]\label{thmintro:mainC}
	Let $H$ be a tame hereditary algebra. Then the chambers of the regular wall-and-chamber structure of $H$ are in bijection with the support regular clusters for $H$. Moreover, the labels of the walls bounding each chamber can be deduced from the corresponding support regular cluster.
\end{customthm}

We conclude Section \ref{sec:regular rigid} by relating our support regular rigid objects to other similar definitions in the literature (Corollary \ref{cor:imaginary clusters}).

Section \ref{sec:cluster-tilted} studies the regular semi-invariant pictures of tame cluster-tilted algebras. We first use the mutation formulas of Reading (Theorem \ref{readingFormula}), Mou (Theorem \ref{thm:mouFormula}) and Derksen-Weyman-Zelevinsky (Theorem \ref{thm:DWZFormula}) to describe how the standard semi-invariant picture mutates in the tame case. We then use this to prove our fourth main theorem:

\begin{customthm}{D}[Theorem \ref{thmD}, simplified form]\label{thmintro:mainD}
	Let $\Lambda = J(Q,W)$ be a tame cluster-tilted algebra over an algebraically closed field $K$, and let $\Gamma$ be a Euclidean quiver which is mutation equivalent to $Q$. Then the regular wall-and-chamber structure of $\Lambda$ is piecewise linearly isomorphic to the regular wall-and-chamber structure of $K\Gamma$.
\end{customthm}

%%%%%%%%%%%%%%%%%%%%%%%%%%%%%%%%%%%%%%%%%%%%%%%%%%%%%%

\subsection{Acknowledgements}
KI is supported by Simons Foundation Grant \#686616. EH was partially supported by NSERC Discovery Grants and the Canada Research Chairs program.  Portions of this work were completed while EH was employed by Brandeis University, by the Norwegian University of Science and Technology (NTNU), and by l'Universit\'e de Sherbrooke. EH is thankful to these institutions for their support. The authors would also like to thank an anonymous referee for their careful reading of this manuscript and numerous suggestions for improvement.

%%%%%%%%%%%%%%%%%%%%%%%%%%%%%%%%%%%%%%%%%%%%%%%%%%%%%%

\section{Background}\label{sec:background}

Let $\Lambda$ be a finite dimensional algebra over an algebraically closed field $K$. We assume throughout that all algebras are basic and connected. We denote by $\mods\Lambda$ the category of finitely generated left $\Lambda$-modules and by $\proj\Lambda$ the full subcategory of finitely generated projective $\Lambda$-modules. We denote by $\ind\Lambda$ the subcategory of indecomposable $\Lambda$-modules. Readers are referred to \cite{ASS_elements} for background material on associative algebras and their representations/modules.

We denote by $\{S(i)\}_{i = 1}^n$ a set of representatives of the isomorphism classes of the simple modules in $\mods\Lambda$ and by $P(i)$ the projective cover of $S(i)$. We assume that $\Lambda$ is basic; that is, $\Lambda \cong \bigoplus_{i = 1}^n P(i)$. Given $M \in \mods\Lambda$, the \emph{dimension vector} of $M$ is the vector $\undim M \in \NN^n$ given by $\undim M_i = \dim_K\Hom_\Lambda(P(i),M)$. We call $M$ a \emph{brick} if $\End_\Lambda(X)$ is a division algebra over $K$. We denote by $\tau$ and $\nu$ the Auslander-Reiten translate and Nakayama functor. We say a module $M \in \mods\Lambda$ is \emph{homogeneous} if $M$ is indecomposable and $\tau M \cong M$.

%%%%%%%%%%%%%%%%%%%%%%%%%%%%%%%%%%%%%%%%%%%%%%%%%%%%%%
\subsection{Wall-and-chamber structures, polyhedral fans and co-amalgamated products}\label{sec:geometryBackground}

In this section, we first give an overview of the definitions of polyhedral cones, general wall-and-chamber structures, and polyhedral fans. Readers are referred to \cite[Chapters 1-2]{brasselet_introduction}, \cite[Chapter 1]{rockafellar_convex}, and \cite[Section 3.1]{BST_wall} for additional details. We then define the \emph{co-amalgamated product} of a pair of wall-and-chamber structures. In this section, we consider $n$ to be an arbitrary positive integer, but in all other sections we consider $n$ as the number of (isoclasses of) simple modules over some algebra.

Let $v_1,\ldots,v_k \in \RR^n$. We denote by $C(v_1,\ldots,v_k)$ the set of linear combinations $\lambda_1 v_1+\cdots + \lambda_k v_k$ with each $\lambda_i \geq 0$; $C(v_1,\ldots,v_k)$ is called a \emph{polyhedral cone.} The \emph{dimension} of a polyhedral cone is the minimum of the dimensions of all linear subspaces of $\RR^n$ containing it.

The polyhedral cone $C(v_1,\ldots,v_k)$ is called \emph{rational} if there exists a scalar multiple of each $v_i$ which is rational, and is called \emph{strictly convex} if it does not contain a 1-dimensional linear subspace. The \emph{relative interior} of a cone is its interior in the linear subspace that it spans and the \emph{relative boundary} of a cone is the complement (in its closure) of its interior.

For $v \in \RR^n$, there is an associated hyperplane $H(v) = \{w \in \RR^n:v\cdot w = 0\}$ and two half spaces $H(v)^+ = \{w \in \RR^n:v\cdot w \geq 0\}$ and $H(v)^- = \{w\in \RR^n:v\cdot w \leq 0\}$. Every polyhedral cone can be written as the intersection of finitely many half-spaces. We call the hyperplanes  corresponding to these half-spaces the \emph{defining hyperplanes} of the cone. A \emph{face} of a polyhedral cone is its intersection with a subset of its defining hyperplanes. We note that for any (strictly convex, rational) polyhedral cone, the intersection of two faces is a face and each face is a (strictly convex, rational) polyhedral cone. We call a cone a \emph{wall} if it has dimension $(n-1)$. 

\begin{defn}\label{def:wallchamber}
A \emph{wall-and-chamber structure} on $\RR^n$ is a pair $(\mathfrak{X},\mathfrak U)$, where $\mathfrak{X}$ is a set of rational polyhedral cones in $\RR^n$ and 
$\mathfrak U$ is the set of connected components of $\RR^n\setminus \overline{\bigcup_{X\in \mathfrak X} X}$,
which satisfies the following:
	\begin{enumerate}
		\item If $C, C' \in \mathfrak{X}$, then $C\cap C' \in \mathfrak{X}$.
		\item If $C \in \mathfrak{X}$, then there exists a wall $W \in \mathfrak{X}$ so that $C \subseteq W$.
		\item Every $U\in\mathfrak U$ is convex.
	\end{enumerate}
The elements of $\mathfrak U$ are called the \emph{chambers} of $\mathfrak X$. Since $\mathfrak X$ determines $\mathfrak U$, we say that $\mathfrak X$ \emph{gives} a wall-and-chamber structure if the above conditions are satisfied.
\end{defn}

 Wall-and-chamber structures are sometimes examples of \emph{polyhedral fans}, which will also appear in this paper. We recall their definition now.

\begin{defn}\cite[Definition 3.1]{brasselet_introduction}\label{def:polyhedral fan}
	A \emph{polyhedral fan} (also sometimes called a cone complex) is a countable\footnote{We note that the original definition of \cite{brasselet_introduction} requires this to be a finite set, but we allow for a countable set so as to include the $g$-vector fan of a finite dimensional algebra (see Definition~\ref{def:g fan}).} set $\mathfrak{X}$ of strictly convex rational polyhedral cones in $\mathbb{R}^n$ such that:
	\begin{enumerate}
		\item If $C \in \mathfrak{X}$ then every face of $C$ is in $\mathfrak{X}$.
		\item If $C, C' \in \mathfrak{X}$, then $C\cap C'$ is a face of both $C$ and $C'$.
	\end{enumerate}
\end{defn}

\begin{eg}
    Let $\fX$ be a polyhedral fan in $\RR^n$ and let $\mathfrak U$ be the set of chambers of $\fX$. Then $(\fX,\mathfrak U)$ is a wall-and-chamber structure if and only if the following hold.
    \begin{enumerate}
        \item For every $C \in \fX$, there exists $W \in \fX$ so that $C \subseteq W$ and $\dim W = n-1$. In this case, the polyhedral fan $\fX$ is said to be \emph{pure of dimension $(n-1)$.}
        \item Every element of $\mathfrak U$ is convex.
    \end{enumerate}
On the other hand, there exist many wall-and-chamber structures $(\fX,\mathfrak U)$ where $\fX$ is not a polyhedral fan. For example, let $\Delta = \{(x,y) \in \RR^2 \mid x = y\}$ and let $\fY = \{\Delta\}$. Then $\fY$ gives a wall-and-chamber structure with two chambers; however, $\fY$ is not a polyhedral fan because the cone $\Delta$ is not strictly convex.
\end{eg}

\begin{rem}
	The wall-and-chamber structure studied in \cite{BST_wall,asai} satisfies Definition \ref{def:wallchamber}, and also has the structure of a \emph{simplicial} fan. Our fans are not simplicial in general, however. See, e.g., the quadrilateral in the center of Figure \ref{fig:regularPics2}. We will associate additional wall-and-chamber structures to finite dimensional algebras, namely the \emph{infinitesimal}, $v^\perp$, and \emph{regular} wall-and-chamber structures, in Section \ref{sec:stability}.
\end{rem}

We will use the following definition for piecewise-linear isomorphisms in this paper.

\begin{defn}\label{def:fan_isom}
    Let $\mathfrak{X}$ and $\mathfrak{Y}$ be sets of rational polyhedral cones in $\RR^n$. Let $\phi:\RR^n \rightarrow \RR^n$ be a piecewise-linear homeomorphism. Suppose that
    \begin{enumerate}
        \item For all $C \in \mathfrak{X}$ and for all $S \subseteq \RR^n$ such that $\phi|_S$ is linear there exists a finite set $\mathcal{V}_{C,S}$ of $(\dim C)$-dimensional cones in $\mathfrak{Y}$ so that $\phi(C\cap S) \subseteq \bigcup_{D \in \mathcal{V}_{C,S}} D$, and
        \item For all $C' \in \mathfrak{Y}$ and for all $S' \subseteq \RR^n$ such that $(\phi)^{-1}|_{S'}$ is linear there exists a finite set $\mathcal{W}_{C',S'}$ of $(\dim C')$-dimensional cones in $\mathfrak{X}$ so that $\phi^{-1}(C'\cap S') \subseteq \bigcup_{D \in \mathcal{W}_{C',S'}} D$.
    \end{enumerate}Then we call $\phi$ a \emph{piecewise-linear isomorphism} from $\fX$ to $\fY$. Moreover, if such a $\phi$ exists then we say $\fX$ and $\fY$ are \emph{piecewise-linearly isomorphic} and write $\fX \cong \fY$.
\end{defn}

\begin{eg}
Denote
        \begin{eqnarray*}
            \Delta &=& \{(x,y) \in \RR^2 \mid x = y\}\\
            C_1 &=& \{(x,y) \in \RR^2 \mid x = -y, \ y \geq 0\}\\
            C_2 &=& \{(x,y) \in \RR^2 \mid x = -y, \ y \leq 0\}.
        \end{eqnarray*}
        Then the map $\phi:\RR^2 \rightarrow \RR^2$ given by $\phi(x,y) = (-x,y)$ is a (piecewise-)linear isomorphism from $\{\Delta\}$ to $\{C_1,C_2\}$.
\end{eg}

We now define co-amalgamated products and tabulate some basic results. This construction will be used to understand the regular wall-and-chamber structure of a tame hereditary algebra in terms of the wall-and-chamber structures of self-injective Nakayama algebras in Section \ref{sec:thmA}.

\begin{defn}\label{def:co-amalgamation}
    Let $\mathfrak{X}$ be a set of rational polyhedral cones in $\RR^n$ and let $\mathfrak{Y}$ be a set of rational polyhedral cones in $\RR^m$. Let $\phi: \RR^n \rightarrow \RR$ and $\psi: \RR^m\rightarrow\RR$ be linear functionals. Denote
    $$\Delta^{\phi,\psi} := \{(v,w) \in \RR^n \times \RR^m \mid \phi(v) = \psi(w)\}.$$
    For readability, we will write elements in $\Delta^{\phi,\psi}$ using their coordinates in $\RR^n \times \RR^m$. For each $V \in \mathfrak{X}$, we denote $\widetilde{V} := \{(v,w) \in \Delta^{\phi,\psi} \subseteq \RR^n \times \RR^m \mid v \in V\}$. For each $W \in \mathfrak{Y}$, we denote $\widetilde{W}$ analogously. We then define the \emph{co-amalgamated product} of $\mathfrak{X}$ and $\mathfrak{Y}$ along $\phi$ and $\psi$, denoted $\mathfrak{X}^\phi \ominus \mathfrak{Y}^\psi$, to be
    $$\mathfrak{X}^\phi \ominus \mathfrak{Y}^\psi := \{\widetilde{V} \mid V \in \mathfrak{X}\} \cup \{\widetilde{W} \mid W \in \mathfrak{Y}\} \cup \{\widetilde{V} \cap \widetilde{W} \mid (V,W) \in \mathfrak{X} \times \mathfrak{Y}\}.$$
\end{defn}

\begin{rem}
 We note that, by identifying $\Delta^{\phi,\psi}$ with $\RR^{n + m - 1}$, the co-amalgamated product $\mathfrak{X}^\phi \ominus\mathfrak{Y}^\psi$ can be seen as a set of rational polyhedral cones in $\RR^{n + m - 1}$.
\end{rem}

 We defer details of co-amalgamated products to the proof of Theorem~\ref{thmintro:mainB} in Section~\ref{sec:co-amalgamation}. As part of the proof, we will explicitly compute the co-amalgamated product of wall-and-chamber structures associated to certain self-injective Nakayama algebras.

\begin{rem}\label{rem:not simplicial}
We note that even if $\mathfrak{X}$ and $\mathfrak{Y}$ are simplicial, it may be the case that $\mathfrak{X}^\phi \ominus\mathfrak{Y}^\psi$ is not simplicial. For example, the wall-and-chamber structure in Figure \ref{fig:regularPics2} is not simplicial; however, it will follow from Theorem~\ref{thmintro:mainB} that this wall-and-chamber structure is the co-amalgamated product of two simplicial wall-and-chamber structures.
\end{rem}

We will be particularly interested in taking the co-amalgamated product of two wall-and-chamber structures. In this special case, we have the following.

\begin{prop}\label{prop:coamalg_wallchamber}
 Let $\mathfrak{X}$ and $\mathfrak{Y}$ give wall-and-chamber structures in $\RR^n$ and $\RR^m$, respectively. Then for any linear functionals $\phi:\RR^n \rightarrow \RR$ and $\psi: \RR^m \rightarrow \RR$, the co-amalgamated product $\mathfrak{X}^\phi \ominus\mathfrak{Y}^\psi$ gives a wall-and-chamber structure (in $\RR^{n + m - 1}$).
\end{prop}

\begin{proof}
    The fact that $\mathfrak{X}^\phi \ominus\mathfrak{Y}^\psi$ is closed under intersections is clear. To see (2) of Definition~\ref{def:wallchamber},
    let $U \in \mathfrak{X}^\phi \ominus\mathfrak{Y}^\psi$. Then we can assume without loss of generality that $U \subseteq \widetilde{V}$ for some $V \in \mathfrak{X}$. There then exists $V' \in \mathfrak{X}$ of dimension $(n-1)$ so that $V \subseteq V'$. It follows that $\widetilde{V} \subseteq \widetilde{V'}$, and that $\widetilde{V'}$ has dimension $n+m-2$. So, $U$ is contained in the wall $\widetilde{V'}$.
     The chambers of $\mathfrak{X}^\phi \ominus\mathfrak{Y}^\psi$ are the intersections $\widetilde C\cap \widetilde D$, where $\widetilde C$, $\widetilde D$ are the inverse images in $\Delta^{\phi,\psi}$ of the chambers $C\subseteq \RR^n$ and $D \subseteq \RR^m$ of $\fX$ and $\fY$, respectively.
\end{proof}

We conclude this section with the following straightforward observation.

\begin{prop}\label{prop:co-amalg facts}
 	Let $\mathfrak{X}_1, \mathfrak{X}_2,$ and  $\mathfrak{X}_3$ be sets of rational polyhedral fans in $\RR^{n_1}$, $\RR^{n_2}$, and $\RR^{n_3}$. Let $\phi_1:\RR^{n_1}\rightarrow \RR$, $\phi_2:\RR^{n_2}\rightarrow \RR$, and $\phi_3:\RR^{n_3} \rightarrow \RR$ be linear functionals. Let $\phi_{12}: \RR^{n_1+n_2 - 1}\rightarrow \RR$ be the linear functional which, when identifying $\RR^{n_1+n_2 - 1}$ with $\Delta^{\phi_1,\phi_2}$ as in Definition~\ref{def:co-amalgamation}, gives $\phi_{12}(v,w) = \phi_1(v) = \phi_2(w)$. Define $\phi_{23}$ likewise. Then
	\begin{enumerate}
		\item $\mathfrak{X}_1^{\phi_1}\ominus \mathfrak{X}_2^{\phi_2} \cong \mathfrak{X}_2^{\phi_2}\ominus \mathfrak{X}_1^{\phi_1}$.
		\item $(\mathfrak{X}_1^{\phi_1}\ominus\mathfrak{X}_2^{\phi_2})^{\phi_{12}}\ominus\mathfrak{X}_3^{\phi_3} \cong \mathfrak{X}_1^{\phi_1}\ominus(\mathfrak{X}_2^{\phi_2}\ominus\mathfrak{X}_3^{\phi_3})^{\phi_{23}}$.
	\end{enumerate}
\end{prop}

For any positive integer $m$, we use
Proposition~\ref{prop:co-amalg facts}, to inductively define $\ominus_{i = 1}^m \mathfrak{X}_i^{\phi_i}$  in the natural way.

%%%%%%%%%%%%%%%%%%%%%%%%%%%%%%%%%%%%%%%%%%%%%%%%%%%%%%

\subsection{$g$-vectors, clusters, and $\tau$-tilting theory}

Let $P_1\xrightarrow{f} P_0$ be a morphism of projective $\Lambda$-modules. Recall that if $f$ is indecomposable (as an object in the bounded derived category) then there are two possibilities:
 \begin{enumerate}
	\item $f$ is a minimal projective presentation of an indecomposable module $M \in \mods\Lambda$. We then identify the morphism $P_1\xrightarrow{f} P_0$ with $M$.
	\item $P_0 = 0$ and $P_1$ is an indecomposable projective module. We then denote the object $P_1\rightarrow 0$ as $P_1[1]$.
\end{enumerate}

The \emph{$g$-vector} of $P(j)$ is defined to be $e_j$ and the $g$-vector of $P(j)[1]$ is defined to be $-e_j$.
We then additively extend these presentations and vectors as follows: given any object $M \oplus P[1]$ with $M\in\mods\Lambda$ and $P\in \proj \Lambda$, let $P_1\oplus P \xrightarrow{[f,0]} P_0 \rightarrow M$ be the direct sum of the minimal projective presentation of $M$ and the object $P\rightarrow 0$. Write $P_0 = \bigoplus_{i = 1}^m P(j_i)$ and $P_1\oplus P = \bigoplus_{i = 1}^{m'} P(j'_i)$ as direct sums of indecomposable projectives. Then the $g$-vector of $M\oplus P[1]$ is $g(M\oplus P[1]) = \sum_{i = 1}^{m} e_{j_i} - \sum_{i = 1}^{m'}e_{j'_i}$.

The name $g$-vector comes with the relationship between these vectors and the $g$-vectors studied in the context of cluster algebras (see \cite{FZ_IV}). We state here only the facts we will need about $g$-vectors.

Recall that a module $X \in \mods\Lambda$ is called \emph{rigid} if $\Ext^1_\Lambda(X,X) = 0$ and is called \emph{$\tau$-rigid} if $\Hom_\Lambda(X,\tau X) = 0$. For $\Lambda$ hereditary, the \emph{cluster category} of $\Lambda$, denoted $\cC(\Lambda)$, was introduced in \cite{BMRRT_tilting,CCS_quivers}. Up to isomorphism, the objects of $\cC(\Lambda)$ are the objects of $\mods\Lambda$ together with the \emph{negative projectives}, objects $P[1]$ for $P$ projective. A basic object $X = M\oplus P[1] \in \cC(\Lambda)$ is called \emph{cluster tilting} or a \emph{cluster} if $M$ is rigid, $\Hom_\Lambda(P,M) = 0$, and $X$ is not a direct summand of any other (basic) object with these properties.

Cluster tilting objects come equipped with a \emph{mutation} that mirrors the combinatorics of the mutation of clusters for a cluster algebra. In order to generalize these combinatorics to the non-hereditary case, $\tau$-tilting theory is introduced in \cite{AIR_tilting}. A basic object $X = M \oplus P[1]$ with $M \in \mods\Lambda$ and $P \in \proj \Lambda$ is called \emph{support $\tau$-rigid} if $M$ is $\tau$-rigid and $\Hom_\Lambda(P,M) = 0$. A support $\tau$-rigid object $X$ is called \emph{support $\tau$-tilting} if it is not a proper direct summand of any other (basic) support $\tau$-rigid object. We denote by $|M\oplus P[1]|$ the number of indecomposable direct summands (up to isomorphism) of the $\tau$-rigid object $M\oplus P[1]$. We denote by $\mathsf{str}(\Lambda)$ and $\mathsf{stt}(\Lambda)$ the sets of (isoclasses of) support $\tau$-rigid and support $\tau$-tilting objects for $\Lambda$. An algebra is called \emph{$\tau$-tilting finite} if $|\mathsf{stt}(\Lambda)| < \infty$ or, equivalently (see \cite{DIJ_tilting}), if $\mods\Lambda$ contains only finitely many bricks (up to isomorphism).

Given two support $\tau$-rigid objects $M\oplus P[1]$ and $N\oplus Q[1]$, we say $M\oplus P[1]$ is \emph{contained} in $N\oplus Q[1]$ if $M$ is a direct summand of $N$ and $P$ is a direct summand of $Q$. One of the key results of \cite{AIR_tilting} is that every support $\tau$-rigid object is contained in at least one support $\tau$-tilting object and that $M\oplus P[1]$ is support $\tau$-tilting if and only if $|M\oplus P[1]| = n$. Moreover, if $|M \oplus P[1]| = n-1$, then $M\oplus P[1]$ is contained in precisely two support $\tau$-tilting objects.

Support $\tau$-rigid objects can be studied geometrically using the \emph{$g$-vector fan}, defined as follows.

\begin{defn}\label{def:g fan} Let $\Lambda$ be a finite dimensional algebra.
	\begin{enumerate}
		\item Let $M\oplus P[1]$ be support $\tau$-rigid for $\Lambda$. We denote by $C(M\oplus P[1])$ the cone in $\RR^n$ spanned by the $g$-vectors of the indecomposable direct summands of $M\oplus P[1]$.
		\item The \emph{$g$-vector fan} of $\Lambda$ is the collection of cones $C(M\oplus P[1])$ for $M\oplus P[1]$ support $\tau$-rigid.
	\end{enumerate}
\end{defn}

This fan was first studied from the viewpoint of representation theory in \cite{DIJ_tilting}. We will use the following facts.

\begin{prop}\label{prop:g fan}
	Let $\Lambda$ be a finite dimensional algebra. Then
	\begin{enumerate}
		\item \cite{DIJ_tilting} The $g$-vector fan is a polyhedral fan which is simplicial.
		\item \cite{AIR_tilting} For any $M\oplus P[1]$ which is support $\tau$-rigid, the vectors spanning the cone $C(M\oplus P[1])$ are linearly independent.
	\end{enumerate}
\end{prop}

We include the following useful proposition which substitutes for the six term Hom-Ext sequence for hereditary algebras.

\begin{prop}\label{useful proposition}
Given $P\xrightarrow pQ\to T\to0$ a minimal projective presentation of any module $T$ and $X$ any other module, there is a natural exact sequence 
\[
	0\to \Hom(T,X)\to \Hom(Q,X)\xrightarrow{p^\ast} \Hom(P,X)\to D\Hom(X,\tau T)\to 0.
\]
\end{prop}

\begin{proof}
$\ker p^\ast=\Hom(T,X)$ by left exactness of $\Hom$. But $D\Hom(P,X)=\Hom(X,\nu P)$ giving the exact sequence
\[
	0\to \Hom(X,\tau T)\to \Hom(X,\nu P)\xrightarrow{Dp^\ast} \Hom(X,\nu Q).
\]
So, $\coker p^\ast=D\ker Dp^\ast=D\Hom(X,\tau T)$.
\end{proof}

\begin{cor}\label{cor: six term exact sequence}
Given any short exact sequence $0\to X\to Y\to Z\to 0$ we get a long exact sequence
\[
	0\to \Hom(T,X)\to \Hom(T,Y)\to \Hom(T,Z)\to \] \[\to D\Hom(X,\tau T)\to D\Hom(Y,\tau T)\to D\Hom(Z,\tau T)\to0
\]
\end{cor}

\begin{proof}
Take $\Hom$ of the projective presentation $p:P\to Q$ of $T$ to the exact sequence to get the following.
$$
%\xymatrixrowsep{10pt}\xymatrixcolsep{10pt}
\xymatrix{%begin xy matrix
0\ar[r] & \Hom(Q,X)\ar[d]\ar[r] &
	\Hom(Q,Y)\ar[d]\ar[r] &
	\Hom(Q,Z)\ar[d]\ar[r] & 0\\
0\ar[r]& \Hom(P,X) \ar[r]& 
	\Hom(P,Y) \ar[r]&
	\Hom(P,Z)\ar[r]&0
	}%end xy matrix
$$
The snake lemma, together with Proposition \ref{useful proposition}, gives the corollary.
\end{proof}

%%%%%%%%%%%%%%%%%%%%%%%%%%%%%%%%%%%%%%%%%%%%%%%%%%%%%%

\subsection{Semi-invariant pictures and the standard wall-and-chamber structure}\label{sec:pictures}

In this section, we recall the notion of semistability and the construction of the (standard) wall-and-chamber structure and semi-invariant picture associated to a finite dimensional algebra (Definition~\ref{def:picture}). 
We follow much of the exposition of \cite{BST_wall}.

	Let $\Lambda$ be an arbitrary finite dimensional algebra. Let $P_1 \xrightarrow{f} P_0$ be the minimal projective presentation of an object $X = M \oplus P[1]$ with $M\in\mods\Lambda$ and $P \in \proj\Lambda$. Then for any $Y \in \mods\Lambda$, we have a 4 term exact sequence
	$$0 \rightarrow \Hom_\Lambda(M,Y) \rightarrow \Hom_\Lambda(P_0,Y) \xrightarrow{f^*} \Hom_\Lambda(P_1,Y) \rightarrow D\Hom_\Lambda(Y,\tau M\oplus\nu P) \rightarrow 0,$$
	where $\tau$ is the Auslander-Reiten translate and $\nu$ is the Nakayama functor.
Observe that if $\coker f = 0$ (i.e. $X = P[1]$), the first two terms are 0 and this restricts to the standard duality $\Hom_\Lambda(P,Y) \cong D\Hom_\Lambda(Y,\nu P)$.
	
We recall that $D\Hom_\Lambda(Y,\tau X) \cong \Ext_\Lambda^1(X,Y)$ when $\Lambda$ is hereditary, but in general $\Ext_\Lambda^1(X,Y)$ can be a proper submodule of $D\Hom_\Lambda(Y,\tau X)$. Moreover, in general, we have the following:
\begin{eqnarray*}
	g(X) \cdot \undim Y &=& \dim_K\Hom_\Lambda(P_0,Y) - \dim_K\Hom_\Lambda(P_1,Y)\\
		&=& \dim_K\Hom_\Lambda(X,Y) - \dim_K\Hom_\Lambda(Y,\tau X).
\end{eqnarray*}

This is known as the \emph{Euler-Ringel pairing} and leads to the following definitions of stability and semistability due to King \cite{king_moduli}. Note that, in particular, $g(M)\cdot \undim M=0$ when $M$ is homogeneous (and thus $\tau M \cong M$).

\begin{defn}\label{def:semistable} Let $\Lambda$ be an arbitrary finite dimensional algebra and let $Y \in \mods\Lambda$.
	\begin{enumerate}
	\item Let $v \in \RR^n$. Then $Y$ is called \emph{$v$-semistable} if $v\cdot\undim Y = 0$ and $v\cdot\undim Y' \leq 0$ for all submodules $Y' \subseteq Y$. If in addition $Y \neq 0$ and $v\cdot\undim Y' < 0$ for all proper submodules $0\neq Y'\subsetneq Y$, then $Y$ is called \emph{$v$-stable}.
	\item The \emph{semi-invariant domain} of $Y$ is the set $D(Y) = \{v \in \RR^n \mid Y\textnormal{ is $v$-semistable}\}$; i.e.,
	\[
	    D(Y)=\{v\in \RR^n\mid v\cdot\underline\dim Y=0,\ v\cdot\underline\dim Y'\le0 \text{ for all submodules } Y'\subseteq Y\}.
	\]
	\end{enumerate}
\end{defn}

The following will be important in the sequel.

\begin{lem}\label{lem:homOrthogonalToEta}
	Let $\Lambda$ be an arbitrary finite dimensional algebra and let $X, Y \in \mods\Lambda$ with $Y$ indecomposable and $X$ a homogeneous brick (so in particular $\tau X \cong X$). Then $Y$ is $g(X)$-semistable (i.e. $g(X)\in D(Y)$) if and only if either\begin{enumerate}
	    \item $\Hom_\Lambda(X,Y) = 0 = \Hom_\Lambda(Y,X)$, or 
	    \item $Y$ is an iterated self-extension of $X$. (This includes the case $Y\cong X$.)
	\end{enumerate} 
\end{lem}

\begin{proof}
    It is straightforward to show that if $Y$ satisfies either (1) or (2), then $Y$ is $g(X)$-semistable. For the converse, suppose for a contradiction that $Y$ is $g(X)$-semistable but does not satisfy either (1) or (2). Furthermore, suppose that $\dim Y$ is minimal with respect to this property. Since $X=\tau X$, we have: $\dim\Hom_\Lambda(X,Y)=\dim\Hom_\Lambda(Y,X)$ and, by Corollary~\ref{cor: six term exact sequence}, $\dim\Hom_\Lambda(X,Y/Z)\ge \dim\Hom_\Lambda(Y/Z,X)$ for any $Z\subsetneq Y$. Thus since $Y$ does not satisfy (1), there is a nonzero map $f: Y\to X$. Let $Z$ be the kernel of this map. We get an induced monomorphism $\overline f:Y/Z\hookrightarrow X$. Since $\Hom_\Lambda(Y/Z,X)\neq0$ there is a nonzero morphism $g: X\to Y/Z$. Then $\overline f\circ g$ is a nonzero endomorphism of $X$, and therefore an automorphism since $X$ is a brick. So, $Y/Z\cong X$. Since $X$ and $Y$ are $g(X)$-semistable, $Z$ is also. By the minimality assumption on $Y$, this means $\Hom_\Lambda(X,Z) = 0 = \Hom_\Lambda(Z,X)$. (If $Z$ is an iterated self-extension of $X$ then so is $Y$, contradicting that $Y$ is a counterexample.) But then the fact that $\Hom_\Lambda(Z,\tau X) = \Hom_\Lambda(Z,X) = 0$ implies that $\Ext^1_\Lambda(X,Z) = 0$, and so the short exact sequence $Z \hookrightarrow Y \twoheadrightarrow X$ is split. Since $Y$ is indecomposable, this implies that $Z = 0$ and $Y \cong X$, a contradiction.
\end{proof}

We now prepare to define what we call the (standard) wall-and-chamber structure associated to a finite dimensional algebra.

\begin{lem}\label{lem: D(Y) is n-1 dim}
    Let $M$ be $v$-stable for some $v\in \RR^n$. Then $D(M)$ is $n-1$ dimensional and $M$ is a brick.
\end{lem}

\begin{proof}
    By definition, $v\cdot \undim M=0$; i.e., $v$ lies in the hyperplane $(\undim M)^\perp$. Stability means $v\cdot \undim M'<0$ for all $0\neq M'\subsetneq M$. This is a finite set of open conditions which will be satisfied for all elements of $(\undim M)^\perp$ close to $v$. So, $D(M)$ is $n-1$ dimensional. To see that $M$ is a brick, suppose not and let $0 \neq f \in \End(M) \setminus \Aut(M)$. Then $0 \neq f(M)\subsetneq M$ is both a submodule and quotient module of $M$ making $v\cdot\undim f(M)=0$, contradicting the assumption that $M$ is $v$-stable.
\end{proof}

We include a proof of the following for the convenience of the reader.

\begin{lem}\label{lem: Asai}\cite[Proposition~2.7]{asai}
For every nonzero module $M$, there is a submodule $M'\subseteq M$ so that $D(M)\subseteq D(M')$, $\dim D(M')=n-1$, and $M'$ is a brick.
\end{lem}

\begin{proof}
If $D(M)= \{0\}$, one can take $M'$ to be any simple submodule of $M$. Otherwise, let $\theta\neq0$ be a point in the relative interior of $D(M)$. Let $M'$ be a minimal submodule of $M$ so that $\theta\cdot \undim M'=0$. Then $M'$ is $\theta$-stable. So $\dim D(M')=n-1$ and $M'$ is a brick. Since $\theta$ is in the interior of $D(M)$ we must have $\theta'\cdot \undim M'=0$ for all $\theta'$ close to $\theta$ in $D(M)$. And this implies $D(M)\subseteq D(M')$.
\end{proof}

The following lemma is adapted from \cite{IT_pictureMGS}.

\begin{lem}\label{lem: IT13}
For any finite nonempty order ideal $\cS \subseteq \NN^n\setminus \{0\}$ let $L_\cS$ denote the union of all $D(M)$ where $\undim M\in \cS$ and $\dim D(M)=n-1$. Then $L_\cS$ is a closed subset of $\RR^n$ whose complement is a disjoint union of at most $2^{|\cS|}$ convex open sets.
\end{lem}

\begin{proof}
When $\cS=\{e_i\}$, $L_\cS=e_i^\perp$ is a hyperplane whose complement has $2^1$ components and the lemma holds. So, suppose $|\cS|\ge2$. Let $\beta\in \cS$ be maximal and $\cS_0=\cS\setminus\{\beta\}$. By induction on the size of $\cS$, $L_{\cS_0}$ is closed and its complement is a disjoint union of at most $2^{|\cS_0|}$ convex open sets.

Let $\{M_j\}$ be the set of all modules $M_j$ with $\undim M_j=\beta$ and $\dim D(M_j)=n-1$. Each $\partial D(M_j)$ is contained in a union of $D(M')$ where $\undim M'<\beta$ and we may assume $\dim D(M')=n-1$ by the previous lemma. So, $\partial D(M_j)\subseteq L_{\cS_0}$ for all $j$. For each component $U$ of the complement of $L_{\cS_0}$ there are two possibilities.
\begin{enumerate}
\item $U$ is disjoint from $D(M_j)$ for all $j$. In this case, $U$ is a component of the complement of $L_\cS$.
\item $U\cap D(M_j)=U\cap \beta^\perp$ for some $j$. Then $U\cap L_\cS=U\cap \beta^\perp$ cuts $U$ into exactly two disjoint convex subsets.
\end{enumerate}
Thus the complement of $L_\cS$ is a disjoint union of at most $2^{|\cS|}$ convex open sets and $L_\cS$ is closed.
\end{proof}

\begin{prop}\label{prop:picture}
    Let $\Lambda$ be a finite dimensional algebra and let 
    $$\fD(\Lambda) = \{D(X) \mid 0 \neq X \in \mods\Lambda\}.$$
    Then $\fD(\Lambda)$ gives a wall-and-chamber structure in $\RR^n$.
\end{prop}

\begin{proof}We show that $\cX=\fD(\Lambda)$ satisfies Definition \ref{def:wallchamber}.
Let $0 \neq X \in \mods\Lambda$. It is well known that for $Y \in \mods\Lambda$, we have $D(X) \cap D(Y) = D(X\oplus Y)$. Moreover, by Asai's Lemma \ref{lem: Asai} above,  there exists a brick $X' \in \mods\Lambda$ for which $D(X) \subseteq D(X')$ and $\dim D(X') = n-1$. To show the last step, that the chambers of $\cX$ are convex, suppose not. Then there exist $v,w$ in the same chamber but the line segment $vw$ passes through $D(X)$ for some $X$. We can take $D(X)$ to be a wall. So, $v$ and $w$ lie in different chambers of $L_\cS$, where $\cS$ is the set of all nonzero $\beta\in\NN^n$ which are $\le \undim X$. Since $L_\cS\subseteq \overline{\bigcup _{X\in\cX}X}$, $v$ and $w$ lie in different chambers of $\cX$. This contradiction proves the Proposition. 
\end{proof}

\begin{defn}\label{def:picture} Let $\Lambda$ be a finite dimensional algebra.
\begin{enumerate}
    \item We refer to the wall-and-chamber structure given by $\fD(\Lambda)$, as in Proposition~\ref{prop:picture}, as the \emph{standard wall-and-chamber structure} of $\Lambda$.
    \item Let $S^{n-1} \subseteq \RR^n$ be the unit sphere and denote $\fL(\Lambda):= \{D(X) \cap S^{n-1} \mid D(X) \in \fD(\Lambda)\}$. We refer to $\fL(\Lambda)$ as the
\emph{semi-invariant picture} of $\Lambda$.
\end{enumerate} 
\end{defn}

\begin{rem}\label{rem:picture}\
	\begin{enumerate}
		\item Since either one determines the other, the term \emph{semi-invariant picture} is also sometimes used to describe standard wall-and-chamber structure in the literature. We distinguish between the two to emphasize the relationship between semi-invariant pictures and the pictures of \cite{igusa_generalized,igusa_obstruction,IO}.
		\item The definition of the semi-invariant picture taken here is also more general than what has appeared in the literature in that it includes non-hereditary algebras. Semi-invariant pictures of finite hereditary type are studied in detail in \cite{ITW_picture,IT_signed,IT_pictureMGS}. The generalization to tame hereditary type (sometimes called pro-pictures) appears in \cite{BHIT_semi-invariant,ITW_periodic,IPT_semi-stable}.
	\end{enumerate}
\end{rem}

\begin{rem}\label{rem:g-fan walls}
We recall from \cite{BST_wall} that the $g$-vector fan embeds into the standard wall-and-chamber structure of an algebra. Moreover, when $\Lambda$ is $\tau$-tilting finite and/or tame, it is known that the $g$-vector fan is dense in $\RR^n$ (see \cite{DIJ_tilting,PY_tame,hille_volume}).
\end{rem}

\begin{rem}\label{rem:c-vector}
	When the boundary of a chamber non-trivially intersects a wall $D(X)$, the vector $\undim X$ is sometimes referred to as a $c$-vector. In general, representation-theoretic $c$-vectors were first defined by Fu \cite{fu_vectors} and studied in connection with the standard wall-and-chamber structure in \cite{treffinger_sign}. These vectors will be relevant to some of the proofs in Section \ref{sec:cluster-tilted} of this paper. In the  hereditary or cluster-tilted case, this agrees with the notion of $c$-vectors used for cluster algebras.
\end{rem}

We conclude this section with the following.

\begin{thm}\label{thm:g fan = wall chamber}
	Let $\Lambda$ be a finite dimensional algebra. Then
	\begin{enumerate}
		\item The chambers of the standard wall-and-chamber structure of $\Lambda$ are precisely the interiors of the cones $C(M\oplus P[1])$ for $M\oplus P[1]$ support $\tau$-tilting. Moreover, different support $\tau$-tilting objects correspond to different chambers.
		\item If $\Lambda$ is $\tau$-tilting finite, then the union of the walls of the standard wall-and-chamber structure of $\Lambda$ is equal to the union of the cones $C(M\oplus P[1])$ for $M\oplus P[1]$ support $\tau$-rigid with $(n-1)$ indecomposable direct summands. 
	\end{enumerate}
\end{thm}

\begin{proof}
    Item (1) is \cite[Theorem~1.4]{asai}. To prove item (2), first let $M \oplus P[1]$ be support $\tau$-rigid with $(n-1)$ indecomposable direct summands. Then $C(M\oplus P[1])$ is contained in a wall of the standard wall-and-chamber structure by \cite[Corollary~3.16]{BST_wall}. For the converse, we first note that since $\Lambda$ is $\tau$-tilting finite, there are only finitely many cones of the form $C(M\oplus P[1])$ with $M\oplus P[1]$ support $\tau$-rigid. Moreover, the union of these cones is all of $\mathbb{R}^n$ by \cite[Theorem~5.4 and Corollary~6.7]{DIJ_tilting} (see also \cite[Proposition~4.8]{asai}). Item (1), together with the fact that a cone $C(M\oplus P[1])$ has codimension 1 if and only if $M\oplus P[1]$ has $n-1$ indecomposable direct summands, then implies the result.
\end{proof}

%%%%%%%%%%%%%%%%%%%%%%%%%%%%%%%%%%%%%%%%%%%%%%%%%%%%%%
\subsection{Tame hereditary algebras}\label{sec:tame hereditary}

In this paper, we use the term ``tame algebra'' to mean a ``strictly tame algebra''; that is, we do not include algebras of finite type. Tame hereditary algebras over an algebraically closed field are thus path algebras over Euclidean quivers. Readers are referred to \cite{DR_indecomposable} and \cite{crawley_lectures} for background about these algebras. We recall here only the details we need for the present paper.

Let $H$ be a tame hereditary algebra. Then there exist infinitely many indecomposable modules $M \in \mods H$ which are homogeneous (so in particular $\tau M \cong M$), and thus satisfy $g(M)\cdot\undim M = 0$. Moreover, there is a unique vector $\eta \in \RR^n$ so that $\undim M \in \ZZ^+\eta$ for all modules satisfying $g(M)\cdot\undim M = 0$. The vector $\eta$ is referred to as the \emph{null root} of the algebra (or quiver/species). It is important to note that $\eta$ is a sincere vector.

Although $\eta$ is a vector (as opposed to a module), we still wish to assign it a $g$-vector. We do this using the following lemma.

\begin{lem}\label{lem:g_dim}
    Let $H$ be a hereditary algebra, and let $M, N \in \mods H$ such that $\undim M = \undim N$. Then $g(M) = g(N)$.
\end{lem}

\begin{proof}
   First note that, since $H$ is hereditary, the dimension vectors of the indecomposable projective modules form a basis of $\RR^n$. Now let $P_1 \xrightarrow{f} P_0$ be a minimal projective presentation of $M$. Again since $H$ is hereditary, we have that $f$ is injective and thus $\undim M = \undim P_0 - \undim P_1$. We conclude that $g(M) =: (g_1,\ldots,g_n)$ is the unique solution of the equation
   $\sum_{i = 1}^n g_i\cdot \undim P(i) = \undim M$. This proves the result.
\end{proof}

\begin{nota}\label{def:g(eta)}
    As a special case of Lemma~\ref{lem:g_dim}, we have that $g(M)$ is the same for all $M \in \mods H$ which satisfy $\undim M = \eta$. We denote this common value by $g(\eta)$, which we refer to as the \emph{$g$-vector of the null root}.
\end{nota}

An indecomposable module $M$ is called \emph{preprojective} if $g(\eta)\cdot \undim M < 0$, is called \emph{regular} if $g(\eta)\cdot \undim M = 0$ and is called \emph{preinjective} if $g(\eta) \cdot \undim M > 0$. An arbitrary module is called preprojective/regular/preinjective if each of its indecomposable direct summands are. We denote by $\Reg H$ the full subcategory of regular modules in $\mods H$.

We also recall the following lemma.

\begin{lem}\label{lem: preprojective maps to homogeneous}
    Let $M \in \Reg H$ and $X \in \mods H$ be indecomposable. Then:
    \begin{enumerate}
        \item If $M$ is homogeneous and $X$ is preprojective, then $\Hom_H(X,M) \neq 0$.
        \item If $M$ is homogeneous and $X$ is preinjective, then $\Hom_H(M,X) \neq 0$.
        \item If $\Hom_H(X,M) \neq 0$, then $X$ is either regular or preprojective.
        \item If $\Hom_H(M,X) \neq 0$, then $X$ is either regular or preinjective.
    \end{enumerate}
\end{lem}

In particular, every submodule of a regular module is either regular or preprojective.

A regular module with no proper regular submodules is called \emph{quasi simple}. The set of isoclasses of quasi simple modules can be partitioned according to their orbits under the Auslander-Reiten translate $\tau$. The extension closure of such a $\tau$-orbit forms a \emph{stable tube}. The \emph{rank} of a tube is the size of the $\tau$-orbit.

Suppose for this paragraph that $n > 2$; i.e., that $H$ is not the Kronecker path algebra. Then stable tubes of rank 1 are called \emph{homogeneous} and the other tubes are called \emph{exceptional}. There are at most 3 exceptional tubes, which we denote $\cT_1,\ldots,\cT_m$. We denote their ranks $r_1,\ldots,r_m$. It is well-known that $\sum_{i = 1}^m (r_i-1) = n-2$. 

We now consider the case where $H = \begin{tikzcd}K(1&2)\arrow[l,shift left]\arrow[l,shift right]\end{tikzcd}$ is the Kronecker path algebra. The definitions above are also standard over this algebra, but it will simplify the statements of many of our results if we choose one of the tubes to be ``exceptional''. Thus, even though every tube in $\mods H$ has rank 1, we will say that the tube containing the module $\begin{tikzcd}K&K\arrow[l,shift left,"1"]\arrow[l,shift right,"0" above]\end{tikzcd}$ is \emph{exceptional} and that all of the tubes are \emph{homogeneous}. If we denote this tube by $\cT_1$ and its rank by $r_1 = 1$, then the equation $\sum_{i = 1}^1 (r_i-1) = n-2$ still holds, with both sides evaluating to 0.

 Returning to the case where $H$ is arbitrary, we note that every indecomposable regular module is contained in a single tube. Each regular module also has a \emph{quasi length} (or regular length), \emph{quasi top} (or regular top), and \emph{quasi socle} (or regular socle), which are the length, top, and socle computed in the category of regular modules. We recall the following facts about regular modules.
\begin{prop}\label{prop:tube info}
	Let $M \in \Reg H$ be indecomposable. Then:
	\begin{enumerate}
	\item $M$ is quasi uniserial; that is, $M$ has a unique (up to isomorphism) quasi composition series (in the category $\Reg H$).
	\item Suppose $M$ lies in a tube of rank $r$. Then $M$ is a brick if and only if its quasi length is at most $r$.
	\item Suppose $M$ lies in a tube of rank $r$. Then $M$ is $\tau$-rigid if and only if its quasi length is less than $r$.
	\item $M$ is uniquely determined by its quasi socle and quasi length.
	\end{enumerate}
\end{prop}

In particular, we fix the following notation for the modules in the exceptional tubes.

\begin{nota}\label{notation:indexing}
	For $H$ a tame hereditary algebra, we denote the exceptional tubes in $\mods H$ by $\cT_1,\ldots,\cT_m$ and by $r_1,\ldots,r_m$ their ranks. When working in the tube $\cT_i$, we identify indices in the same equivalence class mod $r_i$. We denote by $X_{1,1}^i,\ldots,X_{r_i,1}^i$ the quasi simple modules in $\cT_i$, indexed so that $\tau X_{j,1}^i = X_{j -1,1}^i$. For $1 \leq \ell \leq r_i$, we denote by $X_{j,\ell}^i$ the unique regular module (in the tube $\cT_i$) with quasi length $\ell$ and quasi socle $X_{j,1}^i$.
\end{nota}

We will also need the following.

\begin{prop}\label{prop:tube info 2}
	Let $M\in \Reg H$ be indecomposable.
	\begin{enumerate}
		\item If $M$ lies in a homogeneous tube, then $M \cong \tau M$; that is, $M$ is a homogeneous module.
		\item If $N \in \Reg H$ is indecomposable and lies in a different tube than $M$, then $\Hom_H(M,N) = 0 = \Hom_H(N,M)$ and $\Ext_H^1(M,N) = 0 = \Ext_H^1(N,M)$.
	\end{enumerate}
\end{prop}

We also note that homogeneous modules over an hereditary algebra are regular.

\begin{lem}\label{lem:characterization of D(eta)}
    Let $M$ be homogeneous and let $v\in \RR^n$. Then $v\in D(M)$ if and only if the following hold.
    \begin{enumerate}
        \item $v\cdot \eta=0$
        \item $v\cdot \undim X\le 0$ for all preprojective $X \in \mods H$.
    \end{enumerate}
    In particular, $D(M)$ is the same set for all homogeneous $M$.
\end{lem}

\begin{proof}
    First recall that $\undim M$ is a scalar multiple of $\eta$ since $M$ is homogeneous.
    
    Suppose that $v\in \RR^n$ satisfies (1) and (2). Then $v\cdot\undim M=0$ by (1). Moreover, any submodule $M'\subset M$ is either preprojective or another module in the same homogeneous tube by Lemma~\ref{lem: preprojective maps to homogeneous} and Proposition~\ref{prop:tube info 2}. Therefore $v\cdot\undim M'\le 0$ by (2). We conclude that $v\in D(M)$.
    
    Conversely, let $v\in D(M)$. Then (1) clearly holds. We show (2) by contradiction. Let $X$ be a preprojective module of minimal length so that $v\cdot\undim X>0$. By Lemma \ref{lem: preprojective maps to homogeneous}, $\Hom(X,M)\neq0$, so let $f:X\to M$ be a nonzero morphism. Then $\Im f\subset M$, so $v\cdot \undim (\Im f)\le 0$. Since $\ker f\subsetneq X$, $\ker f$ is preprojective and its length is less than that of $X$. So, $v\cdot \undim (\ker f)\le0$ by the minimality assumption. Therefore, $$v\cdot\undim X=v\cdot \undim \Im f+v\cdot \undim \ker f\le0,$$ a contradiction. This proves statement (2).
\end{proof}

\begin{nota}\label{notation 1.4.8}
    We denote by $D(\eta)$ the common subset $D(M) \subseteq \RR^n$ for all homogeneous $M$. This is well-defined by Lemma~\ref{lem:characterization of D(eta)}.
\end{nota}

The standard wall-and-chamber structures for tame hereditary algebras are studied in detail in \cite{IPT_semi-stable}. We use the following.

\begin{prop}\label{prop:regular linear algebra}\cite[Lemma~6.1]{IPT_semi-stable} Recall that $n$ is the number of vertices of $Q$. This is denoted $N$ in \cite{IPT_semi-stable}.
	\begin{enumerate}
		\item The $(n-1)$-dimensional subspace $g(\eta)^\perp$ is spanned by the dimension vectors of the quasi simple modules in the exceptional tubes\footnote{The subspace $g(\eta)^\perp$ is denoted $H^{ss}_\delta$ in \cite{IPT_semi-stable}}.
		\item The only linear dependencies amongst the dimension vectors of the quasi simple modules in the exceptional tubes are, for $1 \leq i \leq m$, $$\sum_{j = 1}^{r_i} \undim X_{j,1}^i = \eta.$$
	\end{enumerate}
\end{prop}

\begin{rem}
    We note that Proposition~\ref{prop:regular linear algebra} holds for the Kronecker due to our convention that one of the tubes is considered to be ``exceptional'' even though it has rank 1. One could also modify the statement of Proposition~\ref{prop:regular linear algebra}(1) to say that $g(\eta)^{\perp}$ is spanned by the dimension vectors of the quasi simple modules in the exceptional tubes together with the null root $\eta$.
\end{rem}

%%%%%%%%%%%%%%%%%%%%%%%%%%%%%%%%%%%%%%%%%%%%%%%%%%%%%%
\subsection{Cluster-tilted algebras}\label{sec:tame cluster-tilted}

Let $H = KQ$ be a hereditary algebra with cluster category $\cC(H)$, and let $M\sqcup P[1]$ be a cluster. Then $\Lambda = \End_{\cC(H)}(M\sqcup P[1])^{op}$ is called a \emph{cluster-tilted algebra} (of type $Q$). Such algebras were first considered in \cite{BMR}. We consider tame cluster-tilted algebras, which are precisely the endomorphism rings of clusters for tame hereditary algebras (see \cite[Section 3.4]{GL-FS} or the introduction of \cite{MR_rigid}). It is shown in e.g. \cite{assem_course} that the Auslander-Reiten quiver of $\Lambda$ can be obtained by deleting the direct summands of $M\sqcup P[1]$ from the Auslander-Reiten quiver of $\cC(H)$. In particular, this implies that there will be a 1-parameter family (parameterized by $K^*$) of homogeneous tubes in $\mods\Lambda$. 

\begin{defn}\label{def: regular module}
Let $\Lambda$ be a tame cluster-tilted algebra. As in the hereditary case, the generators of the homogeneous tubes of $\mods\Lambda$ share a dimension vector, which is the unique \emph{null root} $\eta$ of $\Lambda$. As before, we denote an arbitrary element of the 1-parameter family of homogeneous modules of dimension $\eta$ by $M_\lambda$ and denote $g(\eta) := g(M_\lambda)$. A module $M\in \mods\Lambda$ is defined to be \emph{regular} if $g(\eta) \in D_{\Lambda}(M)$.
\end{defn}

There has been recent work describing the representation theory of tame cluster-tilted algebras in detail. For example, it is shown in \cite{FG_indecomposable} that indecomposable $\tau$-rigid modules are uniquely determined by their dimension vectors. Rigid modules and bricks are further described in \cite{MR_rigid}.

Cluster-tilted algebras are closely related to the mutation of quivers and \emph{quivers with potential}. We briefly describe this relationship here.

Let $Q$ be a quiver with $n$ vertices and with no loops or 2-cycles. Associated to $Q$ is an $n\times n$ skew symmetric matrix $B$ with entries given by
$$b_{ij} = |\{\text{arrows } i \rightarrow j\}| - |\{\text{arrows }j \rightarrow i\}|.$$
We note that the matrix $B$ determines $Q$ uniquely.

Given $k \in \{1,\ldots,n\}$, there exists a new skew symmetric matrix $\mu_k(B)$ given by
$$b'_{ij} = \begin{cases} -b_{ij} & i = k \text{ or } j = k\\
	b_{ij} + \frac{|b_{ik}|b_{kj} + b_{ik}|b_{kj}|}{2} & \text{otherwise}\end{cases}$$
The matrix $\mu_k(B)$ is called the \emph{mutation} of $B$ at $k$ and the formulas for the coefficients $b'_{ij}$ are called the \emph{Fomin-Zelevinsky rules}. The quiver corresponding to the matrix $\mu_k(B)$ is denoted $\mu_k(Q)$ and is called the mutation of the quiver $Q$ at $k$.

Suppose $\Gamma$ is an acyclic quiver and let $Q = \mu_{k_t}\circ \cdots \circ \mu_{k_1}(\Gamma)$ for some arbitrary sequence $(k_1,\ldots,k_t)$. Then, by \cite{BMR_mutation} (see also \cite[Section~2.2]{assem_course}), there exists an ideal $I$ so that the algebra $KQ/I$ is cluster-tilted of type $\Gamma$, and conversely every cluster-tilted algebra of type $\Gamma$ is realized in this way. The ideal $I$ is obtained by taking cyclic derivatives of a \emph{potential}, which is a sum of cycles in the quiver $Q$, see e.g. \cite[Proposition~2]{assem_course}. The generalization of mutation to \emph{quivers with potential} is established in \cite{DWZ_quivers}. Thus the quiver $KQ/I$ is typically written as $J(Q,W)$, where $W$ is a potential. The letter $J$ is used because $KQ/I$ is the \emph{Jacobian algebra} of the quiver with potential $(Q,W)$. We denote $\mu_k(Q,W):= (\mu_k Q, W')$, where $W'$ is the potential making $J(\mu_k(Q),W')$ into a cluster-tilted algebra. Since $Q$ is mutation equivalent to an acyclic quiver, it is shown in \cite{DWZ_quivers} that the potential $W$ is unique up to ``right-equivalence'', and therefore $J(Q,W)$ is uniquely determined up to isomorphism by the quiver $Q$.

For brevity, we do not expand further on the notion of quivers with potential here. Information about how representations mutate along with their quivers will be developed in Section \ref{sec:cluster-tilted} as needed.

%%%%%%%%%%%%%%%%%%%%%%%%%%%%%%%%%%%%%%%%%%%%%%%%%%%%%%

\section{Infinitesimal stability and regular stability}\label{sec:stability}
In this section, we introduce \emph{infinitesimal stability} and \emph{regular stability} and show that these notions coincide. We begin with the following definitions.

\subsection{Infinitesimal stability}

\begin{defn}\label{def:infinitesimal stability}
	Let $\Lambda$ be an arbitrary finite dimensional algebra and let $0 \neq v \in \RR^n$. Let $X \in \mods\Lambda$ be a nonzero module with $v \in D(X)$. Then
	\begin{enumerate}
	    \item The \emph{infinitesimal semi-invariant domain of $X$ at $v$} is defined as
				$$D_{0,v}(X) := \{w \in v^\perp \mid \exists \varepsilon > 0: v + \varepsilon w \in D(X)\}.$$
		\item The \emph{$v^\perp$ semi-invariant domain of $X$} is defined as
				$$D_{v^{\perp}}(X) := \{w \in v^{\perp} \mid w\cdot \undim X = 0, \forall X' \subseteq X: v\cdot \undim X' = 0 \Rightarrow w\cdot \undim X' \leq 0\}.$$
	\end{enumerate}
\end{defn}

\begin{rem}\label{rem:g in D}
	We emphasize the importance of the fact that $v \in D(X)$ in the definition of the infinitesimal semi-invariant domain. Indeed, the existence of $\varepsilon > 0$ so that $v + \varepsilon w \in D(X)$ alone says nothing about the behavior of $D(X)$ near $v$ (since $\varepsilon$ can be taken to be large). However, when $v \in D(X)$, the convexity of $D(X)$ implies that $v + \varepsilon' w \in D(X)$ for all $\varepsilon' < \varepsilon$.
\end{rem}

We are now ready to state and prove the first part of our first main result, which shows that the notions of infinitesimal stability and $v^\perp$ stability coincide.

\begin{thm}[Theorem~\ref{thmintro:mainA}, part 1]\label{thmA}
	Let $\Lambda$ be an arbitrary finite dimensional algebra. Let $0 \neq v \in \RR^n$ and let $X \in \mods\Lambda$ with $v \in D(X)$. Then $D_{0,v}(X) = D_{v^\perp}(X)$.
\end{thm}

\begin{proof}
	Let $w \in D_{0,v}(X)$ and $\varepsilon > 0$ so that $v + \varepsilon w \in D(X)$. Thus we have
		$$0 = (v + \varepsilon w)\cdot \undim X = \varepsilon w \cdot \undim X.$$
	Moreover, if $X' \subseteq X$ with $v\cdot \undim X' = 0$, then
		$$0 \geq (v + \varepsilon w)\cdot \undim X' = \varepsilon w \cdot \undim X'.$$
	As $\varepsilon > 0$, we conclude that $w \in D_{v^\perp}(X)$.
	
	Now let $w \in D_{v^\perp}(X)$ and let $X' \subseteq X$. There are then two possibilities. If $v \cdot \undim X' = 0$, then by assumption we have
	$$(v + \varepsilon w)\cdot \undim X' = \varepsilon w \cdot \undim X' \leq 0$$
	for all $\varepsilon > 0$. Otherwise, we have that $v \cdot \undim X' < 0$ since $v \in D(X)$. In this case, we can choose a sufficiently small $\varepsilon_{X'} > 0$ so that
	$$\varepsilon_{X'}w \cdot \undim X' < -v\cdot \undim X'.$$
	Taking $\varepsilon$ to be the minimum of the $\varepsilon_{X'}$ over all $X' \subseteq X$ which satisfy $v \cdot \undim X' < 0$ (a finite set), we see that $v + \varepsilon w \in D(X)$.
\end{proof}

As an example, we give a brief explanation of how our Theorem~\ref{thmA} is related to Asai's results about the reduction of wall-and-chamber structures \cite{asai}. Let $\Lambda$ be a finite dimensional algebra, and let $M, P \in \mods\Lambda$ so that $M\oplus P[1]$ is support $\tau$-rigid. A celebrated result of Jasso \cite{jasso_reduction} shows that the subcategory $M^\perp \cap ^\perp\tau M \cap P^\perp$ of $\mods\Lambda$ is equivalent to $\mods\Lambda'$ for some finite dimensional algebra $\Lambda'$. Asai then proves the following.

\begin{thm}\cite[Theorem~4.5]{asai}
    Let $\Lambda$ be a finite dimensional algebra and let $M\oplus P[1]$ be support $\tau$-rigid. Consider the cone $C(M\oplus P[1])\subseteq \mathbb{R}^n$ in the wall-and-chamber structure given by $\mathfrak{D}(\Lambda)$. (That is, $C(M\oplus P[1])$ is either the closure of a chamber or the intersection of a set of walls.) Then there exists an open subset $U \subseteq \RR^n$ with $C(M\oplus P[1]) \subseteq U$ and a surjective linear map $\pi:U \rightarrow \RR^{n-|M\oplus P[1]|}$ such that $\pi(\fD(\Lambda) \cap U) = \fD(\Lambda')$. That is, if $C \subseteq \mathbb{R}^n$ is a cone in the wall-and-chamber structure given by $\fD(\Lambda)$, then $\pi(C \cap U) \subseteq \mathbb{R}^{n-|M\oplus P[1]}$ is a cone in the wall-and-chamber structure given by $\fD(\Lambda')$.
\end{thm}

 In the special case where $M\oplus P[1]$ is indecomposable, Asai's theorem is equivalent to our Theorem~\ref{thmA} applied at $v=g(M\oplus P[1])$.
 
 The connection to Asai's reduction result highlights that if $v$ is the $g$-vector of an indecomposable $\tau$-rigid module, then $\{D_{0,v}(X) \mid X \in \mod \Lambda\}$ gives a wall-and-chamber structure in $v^\perp$. The following result shows that this is still the case when $v$ is arbitrary.

\begin{prop}[Theorem~\ref{thmintro:mainA}, part 2]\label{prop:nullPics}
    Let $\Lambda$ be a finite dimensional algbera and let $0 \neq v \in \RR^n$. Denote
    $$
        \fD_{0,v}(\Lambda):= \{D_{0,v}(X) \mid 0 \neq X \in \mods\Lambda, v \in D(X)\}.
    $$
    Then $\fD_{0,v}(\Lambda)$ gives a wall-and-chamber structures in $v^{\perp}$.
\end{prop}

\begin{proof}
   Let $0 \neq X \in \mods\Lambda$ with $0 \neq v \in D(X)$. As for the standard wall-and-chamber structure, we have that if $0 \neq Y \in \mods\Lambda$ with $0 \neq v \in D(Y)$, then $D_{0,v}(X)\cap D_{0,v}(Y) = D_{0,v}(X\oplus Y)$. It remains to show that $D_{0,v}(X)$ is contained in a wall. To see this, recall from the proof of Proposition~\ref{prop:picture} that there exists $X' \in \mods\Lambda$ with $D(X) \subseteq D(X')$ and $\dim D(X') = n-1$. It is clear that $D_{0,v}(X) \subseteq D_{0,v}(X')$, so we need only show that $D_{0,v}(X')$ is a wall in $v^\perp$; i.e., that $\dim D_{0,v}(X') = n-2$. This follows from the observation that
    $D(X') \cap v^{\perp} \subseteq D_{0,v}(X') \subseteq (\undim X')^\perp \cap v^\perp$.
    
   To see that the chambers of $\fD_{0,v}(\Lambda)$ are convex, suppose not. Then $\exists w,w'\in v^\perp$ in the same chamber of $\fD_{0,v}(\Lambda)$ so that $v+w+w'\in D(X)$ for some $X$ so that $v\in D(X)$. We can assume that $\dim D(X)=n-1$. Let $\cS$ be the set of all nonzero vectors in $\NN^n$ which are $\le \undim X$, and let $L_\cS$ be as in Lemma \ref{lem: IT13}. Recall that $L_\cS$ has finitely many chambers. So, we can ignore all chambers of $L_\cS$ which do not contain $v$ in their closures. Then, for sufficiently small $\varepsilon>0$, $v+\varepsilon w$ and $v+\varepsilon w'$ lie in distinct chambers of $L_\cS$ (since $D(X)$ separates them). Taking the walls of these chambers which contains $v$ and intersecting with the affine hyperplane $v + v^\perp$ we see that $w$ and $w'$ lie in different chambers of $\fD_{0,v}(\Lambda)$. Thus, these chambers must be convex.
\end{proof}

\begin{defn}\label{def:nullPics}
    We call $\fD_{0,v}(\Lambda)$ as in Proposition~\ref{prop:nullPics} (together with its chambers) the \emph{infinitesimal wall-and-chamber structure} of $\Lambda$ at $v$. We also define the \emph{infinitesimal semi-invariant picture} of $\Lambda$ at $v$ by intersecting this wall-and-chamber structures with the unit sphere $S^{n-2}\subseteq v^\perp$ as in Definition~\ref{def:picture}.
\end{defn}

\begin{rem}\label{rem:sphere intersection}
	The infinitesimal semi-invariant picture of $\Lambda$ at $v$ is isomorphic to the intersection of the semi-invariant picture of $\Lambda$ with an infinitesimal sphere $S^{n-2}$ perpendicular to $v$ at $v$.
	\end{rem}
	
\begin{eg}\label{eg: A2-tilde example}
An example for the quiver\footnote{Using our numbering convention from Section~\ref{sec:background}, we emphasize that $K\widetilde{A}_{n-1}$ contains $n$ isoclasses of simple modules} $\tilde A_2 = 1\leftarrow 2 \leftarrow 3 \rightarrow 1$ with path algebra $H$ is shown in Figure~\ref{fig:my_label}. Here, the affine hyperplane $g(\eta)+g(\eta)^\perp$ intersected with walls $D(\eta)$, $D(S_2)$, $D(\tau S_2)$ is indicated. These are the only walls which contain $g(\eta)$. Some other walls not containing $g(\eta)$ are drawn, but they do not contribute to $\fD_{0,g(\eta)}(H)$. The tangent circle $S^{n-2}$ ($=S^1$) of $g(\eta)+g(\eta)^\perp$ at $g(\eta)$ is also indicated with the six vectors $\pm w_1,\pm w_2,\pm w_3$ giving the infinitesimal semi-invariant picture $\fD_{0,g(\eta)}(H)\cap S^1$. Higher dimensional examples can be found in \cite{IKTW_periodic} and Section \ref{sec:regular rigid} of the present paper.
\end{eg}

\begin{figure}
    \centering
    \begin{tikzpicture}
\begin{scope}
    \draw[thick] (-2,-2)--(2,2) (-2,2)--(2,-2);
\draw(-1.8,1.7) node[left]{\tiny$D(S_2)\cap(g(\eta)+g(\eta)^\perp)$};
\draw(1.8,1.7) node[right]{\tiny$D(\tau S_2)\cap(g(\eta)+g(\eta)^\perp)$};
\draw[thick, red] (-3,0)--(3,0);
   \draw[very thick,blue,->] (0,0)--(.5,.5);
  \draw[very thick,blue,->] (0,0)--(-.5,-.5);
   \draw[very thick,blue,->] (0,0)--(-.5,.5);
   \draw[very thick,blue,->] (0,0)--(.5,-.5);
  \draw[very thick,red,<->] (-.7,0)--(.7,0);
\draw (1.3,1.3)--(-3,.6);
 \draw (-1.7,1.7)--(3,1.1);
 \draw (-1.7,-1.7)--(3,-1);
 \draw (-3,-1.18)--(-1.7,-1.7)--(-1.2,-1.9);
\draw (-.97,-.97)--(3,-.6);
\draw (1.27,-1.27)--(-3,-.7);
\draw[red,fill] (0,0) circle[radius=2pt];
\draw[red] (-.1,.3) node{$g(\eta)$} (2.6,0) node[above]{\tiny$D(\eta)\cap(g(\eta)+g(\eta)^\perp)$};
\end{scope}
\begin{scope}
\draw (0,0) circle[radius=.6cm];
\draw[thick,blue,fill] (.42,.42) circle[radius=2pt] (.4,.43) node[above]{$w_2$}
(.42,-.42) circle[radius=2pt] (.4,-.43) node[below]{$-w_1$}
(-.42,.42) circle[radius=2pt] (-.4,.43) node[above]{$w_1$}
(-.42,.-.42) circle[radius=2pt] (-.6,-.43) node[below]{$-w_2$}; 
\draw[thick,red,fill] (.6,0) circle[radius=2pt] (.6,-.1)node[right]{$w_3$}
 (-.6,0) circle[radius=2pt] (-.6,-.1)node[left]{$-w_3$};
\end{scope}
\end{tikzpicture}
    \caption{The affine hyperplane $g(\eta)+g(\eta)^\perp$ with the infinitesimal sphere $S^{n-2}=S^1$ centered at $g(\eta)$ is shown. This structure $\fD_{0,g(\eta)}(H)$ has 6 walls and 6 chambers.}
    \label{fig:my_label}
\end{figure}

\subsection{Regular stability}

In this paper, we are primarily interested in the wall-and-chamber structure $\fD_{0,g(\eta)}(\Lambda)$ for $\Lambda$ a tame hereditary or cluster-tilted algebra. The following gives a precise description of the bricks $X$ with $g(\eta) \in D(X)$.
	
\begin{prop}\label{prop:orthogonal to null}
	Let $H$ be a tame hereditary or cluster-tilted algebra and let $X \in \mods H.$ Then $g(\eta) \in D(X)$ if and only if $X$ is regular. Moreover, if $H$ is tame hereditary then $g(X) \in D(\eta)$ if and only if $g(\eta) \in D(X)$.
\end{prop}

\begin{proof}
	Let $X \in \mods H$ be arbitrary, and let $M_\lambda \in \mods H$ be homogeneous with $\undim M_\lambda = \eta$. By Lemma \ref{lem:homOrthogonalToEta}, $g(\eta) \in D(X)$ if and only if $\Hom_H(M_\lambda,X) = 0 = \Hom_H(X,M_\lambda)$, which is equivalent to $X$ being regular.
	
	Moreover, if $H$ is hereditary then $\Hom_H(M_\lambda,\tau X) \cong \Hom_H(\tau^{-1}M_\lambda,X) \cong \Hom_H(M_\lambda,X)$. This means $g(X) \in D(\eta)$ if and only if $\Hom_H(M_\lambda,X) = 0 = \Hom_H(X,M_\lambda)$.
\end{proof}

\begin{rem}\label{rem:thmA caess}
	In the case that $\Lambda$ is tame hereditary and $v = g(\eta)$, Proposition~\ref{prop:orthogonal to null} shows the cases $v \cdot \undim X' = 0$ and $v\cdot \undim X' < 0$ in the proof of Theorem~\ref{thmA} correspond to the regular and preprojective submodules of $X$, respectively.
\end{rem}

\begin{rem}\label{rem:non-symmetric}
In general, it is not true for a tame cluster-tilted algebra that $g(\eta) \in D(X)$ if and only if $g(X) \in D(\eta)$. For example, consider the cluster-tilted algebra $KQ/I$ of type $\widetilde{A}_3$ with $Q = 1\xrightarrow{\alpha} 3 \xrightarrow{\beta} 4 \xrightarrow{\gamma} 1 \rightarrow 2 \rightarrow 3$ and $I = (\alpha\beta + \beta \gamma + \gamma\alpha)$. Then $\eta = (1,1,1,0)$, $g(\eta) = (1,0,-1,0)$, and $g(S(4)) = (-1,0,0,1)$. Thus $g(S(4)) \notin D(\eta)$ but $g(\eta) \in D(S(4))$.
\end{rem}

In light of Proposition \ref{prop:orthogonal to null}, we restate Definitions \ref{def:infinitesimal stability} and \ref{def:nullPics} in our case of interest.

\begin{defn}\label{def:regular stability}
 Let $H$ be a tame hereditary or cluster-tilted algebra.
 	\begin{enumerate}
		\item Let $X \in \Reg H$. (See Definition \ref{def: regular module}.) Then we refer to the infinitesimal semi-invariant domain of $X$ at $g(\eta)$ as the \emph{regular semi-invariant domain} of $X$, which is given by
			$$D_{reg}(X) := D_{0,g(\eta)}(X) =  \{w \in g(\eta)^\perp \mid w \cdot \undim X = 0, \forall \textnormal{ regular }X' \subseteq X: w \cdot \undim X' \leq 0\}.$$
		\item For $w \in g(\eta)^\perp$, we say $X$ is \emph{$w$-regular semistable} if $w \in D_{reg}(X)$. If $w\cdot \undim X' < 0$ for all regular $0 \neq X' \subsetneq X$, we say $X$ is \emph{$w$-regular stable}.
		\item We refer to the infinitesimal wall-and-chamber structure at $g(\eta)$ as the \emph{regular wall-and-chamber structure} of $H$, which we denote $\fD_{reg}(H):=\fD_{0,g(\eta)}(H)$.
		\item We refer to the infinitesimal semi-invariant picture at $g(\eta)$ as the \emph{regular semi-invariant picture} of $H$. 
	\end{enumerate}
\end{defn}

\begin{rem}
    Define $D_{0,g(\eta)}(\eta):=D_{0,g(\eta)}(M)$ for any homogeneous $M$. Similarly, define $D_{reg}(\eta)\linebreak :=D_{reg}(M)$ for any homogeneous $M$. In particular, Lemma \ref{lem:characterization of D(eta)} implies that $D_{0,g(\eta)}(\eta) = \eta^\perp\cap g(\eta)^\perp = D_{reg}(\eta)$.
\end{rem}

%%%%%%%%%%%%%%%%%%%%%%%%%%%%%%%%%%%%%%%%%%%%%%%%%%%%%%

\section{Regular wall-and-chamber structures from self-injective Nakayama algebras}\label{sec:co-amalgamation}
Let $H$ be a tame hereditary algebra. In this section, we associate to each exceptional tube in $\mods H$ of rank $r$ a self-injective Nakayama algebra $\Lambda_r$. We then show that the regular wall-and-chamber structure of $H$ is a co-amalgamated product of the standard wall-and-chamber structures of these self-injective Nakayama algebras.

%%%%%%%%%%%%%%%%%%%%%%%%%%%%%%%%%%%%%%%%%%%%%%%%%%%%%%

\subsection{The self-injective Nakayama algebras and their standard wall-and-chamber structures}\label{sec:nakayamaEqns}
We now turn our attention to certain self-injective Nakayama algebras and describe their standard wall-and-chamber structures.

\begin{defn}\label{def:nakayamas}
	 Let $Z_r$ denote the cyclic quiver with $r$ vertices $1,2,\ldots,r$ and $r$ arrows arranged in one oriented cycle:
 	$$Z_r: r \rightarrow r-1 \rightarrow r-2 \rightarrow \cdots \rightarrow 1 \rightarrow r.$$
Let $\Lambda_r$ be $KZ_r$ modulo $rad^{r+1}$; i.e. the composition of any $r+1$ arrows is zero. 
\end{defn}

The AR quiver for the algebra $\Lambda_4$ is shown in Figure \ref{fig:NakayamaAR}. Note also that $\Lambda_1$, which will be relevant only for the Kronecker, is isomorphic to $K[\alpha]/(\alpha^2)$. In the following proposition, we collect several facts about the algebra~$\Lambda_r$.

\begin{figure}
 \begin{center}
 \begin{tikzpicture}[scale=.9]
\draw (-3,2.9) node{...};
\draw (14,2.9) node{...};
\foreach \x/\xtext in {0/1,4/2,8/3,12/4} % simple modules
\draw (\x,.5) node{\xtext} ;
\foreach \x/\xtext in {0/$Y_{11}=$,4/$Y_{21}=$,8/$Y_{31}=$,12/$Y_{41}=$} % simple modules
\draw (\x,.5) node[left]{\xtext} ;
\foreach \x in {0,4,8} % dashed lines
\draw[dashed] (\x,.5) +(.4,0) --+(2.6,0);
\foreach \x in {0,4,8} % dashed lines
\draw[dashed] (\x,2.9) + (.4,0) --+(2.6,0);
\foreach \x in {0,4,8,12} % dashed lines
\draw[dashed] (\x,4.15) + (-1.65,0) --+(0.6,0);
\foreach \x in {0,4,8,12} % dashed lines
\draw[dashed] (\x,1.7) + (-1.65,0) --+(0.6,0);
\foreach \x in {0.2,4.2,8.2,12.2} % arrows ne
\draw[->,thick] (\x,.7)  -- +(.6,.8);
\foreach \x in {0.2,4.2,8.2,12.2} % arrows se
\draw[->,thick] (\x,.7)  +(.6,.8) +(-2,.8)-- +(-1.4,0);
\foreach \x/\xtext in {-2/4,2/1,6/2,10/3,14/4} % length 2 modules
\draw (\x,1.5) node{\xtext};
\foreach \x/\xtext in {-2/1,2/2,6/3,10/4,14/1} % length 2 modules
\draw (\x,1.9) node{\xtext};
\foreach \x/\xtext in {-2/$Y_{42}=$,2/$Y_{12}=$,6/$Y_{22}=$,10/$Y_{32}=$,14/$Y_{42}=$} % length 2 modules
\draw (\x,1.7) node[left]{\xtext};
\begin{scope}[xshift=-1cm,yshift=1.2cm]
\foreach \x in {-.8,3.2,7.2,11.2} % arrows ne
\draw[->,thick] (\x,.7) -- +(.6,.8);
\end{scope}
\begin{scope}[xshift=1cm,yshift=1.2cm]
\foreach \x in {0.2,4.2,8.2,12.2} % arrows se
\draw[->,thick] (\x,.7)  +(.6,.8) +(-1,.8)-- +(-.4,0);
\end{scope}
\foreach \x/\xtext in {0/4,4/1,8/2,12/3} % top of length 3
\draw (\x,2.5) node{\xtext};
\foreach \x/\xtext in {0/1,4/2,8/3,12/4}
\draw (\x,2.9) node{\xtext};
\foreach \x/\xtext in {0/$Y_{43}=$,4/$Y_{13}=$,8/$Y_{23}=$,12/$Y_{33}=$}
\draw (\x,2.9) node[left]{\xtext};
\foreach \x/\xtext in {0/2,4/3,8/4,12/1}
\draw (\x,3.3) node{\xtext};
\foreach \x/\xtext in {-2/3,2/4,6/1,10/2,14/3} % length 4 modules
\draw (\x,3.7) node{\xtext};
\foreach \x/\xtext in {-2/4,2/1,6/2,10/3,14/4} % length 4 modules
\draw(\x,4) node{\xtext};
\foreach \x/\xtext in {-2/$Y_{34}=$,2/$Y_{44}=$,6/$Y_{14}=$,10/$Y_{24}=$,14/$Y_{34}=$} % length 4 modules
\draw(\x,4.15) node[left]{\xtext};
\foreach \x/\xtext in {-2/1,2/2,6/3,10/4,14/1} % length 4 modules
\draw (\x,4.3) node{\xtext};
\foreach \x/\xtext in {-2/2,2/3,6/4,10/1,14/2} % length 4 modules
\draw (\x,4.6) node{\xtext};
\begin{scope}[yshift=2.4cm]
\foreach \x in {0.2,4.2,8.2,12.2} 
\draw[->,thick] (\x,.7) -- +(.6,.8);
\end{scope}
\begin{scope}[yshift=2.4cm]
\foreach \x in {-.8,3.2,7.2,11.2} 
\draw[->,thick] (\x,.7)  +(.6,.8) +(-1,.8)-- +(-.4,0);
\end{scope}
\begin{scope}[yshift=2mm]
\begin{scope}[xshift=-1cm,yshift=3.5cm]
\foreach \x in {-.8,3.2,7.2,11.2} 
\draw[->,thick] (\x,.7) -- +(.6,.8);
\end{scope}
\begin{scope}[xshift=1cm,yshift=3.5cm]
\foreach \x in {0.2,4.2,8.2,12.2} 
\draw[->,thick] (\x,.7)  +(.6,.8) +(-1,.8)-- +(-.4,0);
\end{scope}

\foreach \x/\xtext in {0/$Y_{35}=P(3)$,4/$Y_{45}=P(4)$,8/$Y_{15}=P(1)$,12/$Y_{25}=P(2)$} 
\draw (\x,5.3)+(.5,0) node[left]{\xtext};
\end{scope}
\end{tikzpicture}
\caption{The AR quiver for the algebra $\Lambda_4=KZ_4/rad^5$. $Y_{j\ell}$ is the module with socle $S(j)$ and length $\ell$.}\label{fig:NakayamaAR}
\end{center}
\end{figure}

\begin{prop}\label{prop:Nakayama}
 For all $r\ge1$, $\Lambda_r$ has the following properties.
\begin{enumerate} 
	\item \cite[Prop. 3.8]{ASS_elements} $\Lambda_r$ is self-injective and uniserial.
	\item \cite[Thm. 3.2]{ASS_elements} All indecomposable $\Lambda_r$-modules are uniserial.
	\item \cite[Thm. 3.5]{ASS_elements} There are exactly $r(r+1)$ isomorphism classes of indecomposable $\Lambda_r$-modules, given by $P(i)/rad^kP(i)$ for $i \in \{1,\ldots,r\}$ and $k \in \{1,\ldots,r+1\}$.
	\item An indecomposable $\Lambda_r$-module is a brick if and only if its length is at most $r$.
	\item Every brick in $\mods\Lambda_r$ is determined by any two of its top, socle, and length.
\end{enumerate}
\end{prop}

(4) is an easy exercise. (5) is an immediate consequence of (4).

\begin{nota}\label{notation:nakayama indexing}
	We denote by $Y_{j,\ell}$ the unique indecomposable $\Lambda_r$-module with socle $S(j)$ and length $\ell$. As when working in the exceptional tubes of $\mods H$, we identify indices mod $r$ when working in $\mods\Lambda_r$. We denote by $\fD(\Lambda_r)$ with its chambers the standard wall-and-chamber structure of $\Lambda_r$ (Definition \ref{def:picture}) and by $D(Y_{j,\ell})$ the semi-invariant domain of $Y_{j,\ell}$. We denote $\overline{1}:= (1,\ldots,1)$.
\end{nota}

It is straightforward to show that $\tau Y_{j,\ell} \cong Y_{j-1,\ell}$ for any brick $Y_{j,\ell} \in \mods\Lambda_r$. 

\begin{prop}\label{prop:NakayamaEqns}
	Let $Y_{j,\ell} \in \mods\Lambda_{r}$ be a brick. Then the semi-invariant domain of $Y_{j,\ell}$ is given by the equation and inequalities
	$$
	D(Y_{j,\ell}) = \left\{v \in \RR^r \ \middle| \ \sum_{k = j}^{j+\ell-1} v\cdot \undim Y_{k,1} = 0, \ \forall \  \ell' < \ell: \sum_{k = j}^{j+\ell' - 1} v\cdot \undim Y_{k,1} \leq 0\right\}.
	$$
\end{prop}

\begin{proof}
	By construction, the unique composition series of $Y_{j,\ell}$ is
	$$0 \subsetneq Y_{j,1} \subsetneq \cdots \subsetneq Y_{j,\ell}$$
	and these are all of the submodules of $Y_{j,\ell}$. In addition, 
	the composition factors of $Y_{j,\ell}$ are
	$Y_{j,1}, \ldots, Y_{j + \ell - 1, 1}.$
	The result then follows from $
	\undim Y_{j,\ell} = \sum_{k = j}^{j+\ell-1} \undim Y_{k,1}.$
\end{proof}

\begin{rem}\label{rem:functionals}
	We recall that $\{\undim Y_{j,1}\}_{j = 1}^r$ is the standard basis of $\RR^r$. Thus the coordinate functional $y_j:\RR^r \rightarrow \RR$ can be identified with the functional $(-)\cdot \undim Y_{j,1}$. In particular, this means
	$$
	D(Y_{j,\ell}) = \left\{v \in \RR^r \ \middle| \ \sum_{k = j}^{j+\ell-1} y_k(v) = 0, \ \forall \ \ell' < \ell: \sum_{k = j}^{j+\ell'-1} y_k(v) \leq 0\right\}.
	$$
\end{rem}

We conclude this section with the following observation which will be critical in our proof of Theorem~\ref{thmintro:mainB}.

 Given a sum $\sum_{k = i}^j a_k$, we call a sum of the form $\sum_{k = i}^{j'} a_k$ for some $i \leq j'\leq j$ a \emph{left subsum}. We say that a left subsum is \emph{proper} if $j' < j$.
 
\begin{lem}\label{lem:wholeHyperplane}
	Let $(a_1,\ldots,a_r)$ be an ordered list of real numbers so that $\sum_{k = 1}^r a_k = 0$. Then there exists some index $i$ so that every left subsum of $\sum_{k = i}^{i + r - 1} a_k$ is non-positive (where indices are identified mod $r$).
\end{lem}

\begin{proof}
	If the $a_k$ are identically zero, we are done. Thus assume this is not the case and cyclically re-index the list so that $a_1 < 0$.
	
	Assume the claim does not hold for $i = 1$ and let $p_1$ be the smallest index so that the sum $\sum_{k = 1}^{p_1} a_k > 0$. By construction, every proper left subsum of this sum is non-positive and satisfies $\sum_{k = 1 + p_1}^{r} a_k <  0$. We denote by $n_1$ the smallest index with $n_1 > p_1$ and $a_{n_1} < 0$.
	
	We now consider the sum $\sum_{k = n_1}^{r} a_k$. If this sum has a positive left subsum, we define $p_2$ and $n_2$ as before. Iterating this construction, we eventually have some $n_m$ for which $\sum_{k = n_m}^{r} a_k$ has no positive left subsum.
	
	We can then partition our sum as follows:
	$$0 = \sum_{k = 1}^r a_k = \sum_{k = 1}^{p_1} a_k + \sum_{k = 1+p_1}^{-1+n_1}a_k + \sum_{k = n_1}^{p_2} a_k + \cdots + \sum_{k = 1+p_m}^{-1 + n_m} a_k + \sum_{k = n_m}^r a_k.$$
	By construction, all of the sums on the right hand side of this equation are positive except the last one. Moreover, sums of the form $\sum_{k = 1 + p_i}^{-1 + n_i} a_k$ have all positive summands and all other sums on the right hand side have no positive proper left subsums. It follows that $\sum_{k = n_m}^{n_m + r - 1} a_k$ has no positive left subsum (where indices are identified mod $r$).
\end{proof}

\begin{prop}\label{prop:wholeHyperplane}
	The union of the semi-invariant domains of the bricks (in $\mods\Lambda_r$) of dimension vector $\overline{1} = (1,\ldots,1)$ is a hyperplane. More precisely, $\bigcup_{j = 1}^r D(Y_{j,r}) = \ker \sum_{k = 1}^r y_{k}.$ 
\end{prop}

We deduce this proposition from the above lemma.

\begin{proof}
	By Remark \ref{rem:functionals}, we have that $D(Y_{j,r}) \subseteq \ker\sum_{k = 1}^r y_{k}$ for all $j$. Conversely, let $v \in \ker\sum_{k = 1}^r y_{k}$. By Lemma \ref{lem:wholeHyperplane}, there exists an index $i$ so that $\sum_{k = i}^{i'} v_k \leq 0$ for all $i \leq i' \leq i + r$. Remark \ref{rem:functionals} then implies $v \in D(Y_{i,r})$.
\end{proof}
%%%%%%%%%%%%%%%%%%%%%%%%%%%%%%%%%%%%%%%%%%%%%%%%%%%%%%

\subsection{Regular wall-and-chamber structures as co-amalgamations}\label{sec:thmA}
We now give explicit equations and inequalities defining the cones in the wall-and-chamber structure given by $\fD_{reg}(H)$ (Def. \ref{def:regular stability}) using $X^i_{j,\ell}$ (Notation \ref{notation:indexing}).

\begin{prop}\label{prop:tameEqns}
	Let $H$ be a tame hereditary algebra. Let $X_{j,\ell}^i\in \Reg H$ be an indecomposable module in the tube $\cT_i$. Then the regular semi-invariant domain of $X_{j,\ell}^i$ is given by the equation and inequalities
	$$D_{reg}(X_{j,\ell}^i) = \left\{v \in g(\eta)^\perp \ \middle| \ \sum_{k = j}^{j+\ell-1} v\cdot \undim X_{k,1}^i = 0, \ \forall \ \ell' < \ell: \sum_{k = j}^{j+\ell'-1} v\cdot \undim X_{k,1}^i \leq 0\right\},$$
	where indices are considered mod $r_i$.
\end{prop}

\begin{proof}
	By construction, the unique quasi composition series of $X_{j,\ell}^i$ is
	$$0 \subsetneq X_{j,1}^i \subsetneq \cdots \subsetneq X_{j,\ell}^i$$
	and these are all of the regular submodules of $X_{j,\ell}^i$. In addition, 
	the regular composition factors of $X_{j,\ell}^i$ are
	$X_{j,1}^i, \ldots, X_{j + \ell - 1, 1}^i.$
	This means $$\undim X_{j,\ell}^i = \sum_{k = j}^{j+\ell-1} \undim X_{k,1}^i.$$
	The result then follows immediately.
\end{proof}

Our next result will allow us to recover the regular semi-invariant domain $D_{reg}(\eta)$ from those considered in Proposition~\ref{prop:tameEqns}, i.e., $D_{reg}(X_{j,\ell}^i)$.

\begin{lem}\label{lem:decompose D(eta)}
	Let $H$ be a tame hereditary algebra. Then for each exceptional tube $\cT_i$ with rank $r_i$ we have a decomposition
	$$D_{reg}(\eta) = \bigcup_{j = 1}^{r_i} D_{reg}(X_{j,r_i}^i).$$
\end{lem}

\begin{proof}
	This follows using the same argument as for the proof of Proposition \ref{prop:wholeHyperplane} by replacing each $v_j = y_j(v) = \undim Y_{j,1}^i\cdot v$ with $\undim X_{j,1}^i \cdot v$ and replacing $\ker \sum_{k = 1}^{r_i}y_k$ with $\ker\sum_{k = 1}^{r_i}(-)\cdot\undim X_{k,1}^i$.
\end{proof}

We now strengthen this result to the standard semi-invariant domain of $\eta$.

\begin{cor}\label{cor:decompose D(eta)}
	Let $H$ be a tame hereditary algebra and let $\cT_i$ be an exceptional tube with rank $r_i$. Then there is a decomposition
	$$D(\eta) = \bigcup_{j = 1}^{r_i} D(X_{j,r_i}^i).$$
\end{cor}

\begin{proof}
	Let $\cT_i$ be an exceptional tube and let $X_{j,r_i}^i$ be of dimension $\eta$. This means that every submodule of $X_{j,r_i}^i$ is either regular or preprojective. Moreover, we recall from Lemma~\ref{lem:characterization of D(eta)} and Notation \ref{notation 1.4.8} that $v \in D(\eta)$ if and only if (a) $v\cdot \eta = 0$ and (b) for all preprojective $X \in \mods H$ we have $v\cdot \undim X \leq 0$. This means
		$$D(\eta)\supseteq \bigcup_{j = 1}^{r_i}D(X_{j,r_i}^i) \supseteq \bigcup_{j = 1}^{r_i}D_{reg}(X_{j,r_i}^i)\bigcap D(\eta)
		= D_{reg}(\eta)\bigcap D(\eta) = D(\eta).$$
\end{proof}

Corollary~\ref{cor:decompose D(eta)} is closely related to a result of \cite{IKTW_periodic} in type $\widetilde{A}_{n-1}$, where $D(\eta)$ is shown to be the union of certain ``$D_{ab}(\eta)$''.

\begin{rem}
    Note that both Lemma~\ref{lem:decompose D(eta)} and Corollary~\ref{cor:decompose D(eta)} hold for the Kronecker by our convention to consider one of the homogeneous tubes as ``exceptional''.
\end{rem}

The next result follows immediately from comparing Notations~\ref{notation:indexing} and~\ref{notation:nakayama indexing}.

\begin{prop}\label{prop:same homs}
	Let $H$ be a tame hereditary algebra and let $\cT_i$ be an exceptional tube and let $\Lambda_{r_i}$ be the Nakayama algebra associated to $\cT_i$. Then for any $j,j',\ell,\ell'$ with $\ell,\ell' \leq r_i$, we have	
	\begin{eqnarray*}
		\Hom_H(X_{j,\ell}^i,X_{j',\ell'}^{i}) &\cong& \Hom_{\Lambda_{r_i}}(Y_{j,\ell},Y_{j',\ell'})\\
		\Hom_H(X_{j,\ell}^i,\tau X_{j',\ell'}^{i}) &\cong& \Hom_{\Lambda_{r_i}}(Y_{j,\ell},\tau Y_{j',\ell'}).
	\end{eqnarray*}
	Moreover, if $X_{j'',\ell''}^i$ is the kernel (resp. cokernel) of a morphism $X_{j,\ell}^i \rightarrow X_{j',\ell'}^i$ then $Y_{j'',\ell''}$ is the kernel (resp. cokernel) of the corresponding morphism $Y_{j,\ell} \rightarrow Y_{j',\ell'}$.
\end{prop}

We are now ready to prove Theorem~\ref{thmintro:mainB}, which uses co-amalgamation (see Definition~\ref{def:co-amalgamation}) to describe $\fD_{reg}(H)$ in terms of the $\fD(\Lambda_{r_i})$.

\begin{thm}[Theorem~\ref{thmintro:mainB}]\label{thmB}
	Let $H$ be a tame hereditary algebra. For each exceptional tube $\cT_i$, let $\fD(\Lambda_{r_i})$ (with its chambers) be the standard wall-and-chamber structure of the self-injective Nakayama algebra $\Lambda_{r_i}$, and let $\phi_i:\RR^{r_i}\rightarrow \RR$ be given by $(-)\cdot\overline{1}$. Then the co-amalgamated product
	$\ominus_{i = 1}^m \fD(\Lambda_{r_i})^{\phi_i}$
	is (piecewise-) linearly isomorphic to $\fD_{reg}(H) = \fD_{0,g(\eta)}(H)$.
\end{thm}

\begin{proof}
    Throughout the proof, we identify $\RR^{n-1}$ with $g(\eta)^\perp \subseteq \RR^n$. Moreover, we denote by
    $$\{e_1^1,\ldots,e_{r_1}^1,\ldots,e_1^m,\ldots,e_{r_m}^m\}$$
    the standard basis of $\RR^{r_1 + \cdots + r_m}$. Likewise, we denote by
    $$\Delta^{\phi_1,\ldots,\phi_m} = \{(v_1,\ldots,v_m) \in \RR^{r_1 + \cdots + r_m} \mid \phi_1(v_1) = \cdots = \phi_m(v_m)\}.$$
    
    Now let $\Psi:\RR^{r_1+\cdots+r_m}\rightarrow g(\eta)^\perp$ be the linear map which sends $e_j^i$ to $\undim X_{j,1}^i$ for each $i \in \{1,\ldots,m\}$ and $j \in \{1,\ldots,r_i\}$. We claim that $\Psi|_{\Delta^{\phi_1,\ldots,\phi_m}}$ is an isomorphism from $\ominus_{i = 1}^m \fD(\Lambda_{r_i})^{\phi_i}$ to $\fD_{reg}(H)$.
    
    We first note that $\Psi|_{\Delta^{\phi_1,\ldots,\phi_m}}$ is a linear isomorphism of vector spaces. Indeed, by Proposition~\ref{prop:regular linear algebra}, we have that $\Psi$ is surjective and $$\dim \Delta^{\phi_1,\ldots,\phi_m} = 1 + (r_1-1) + \cdots + (r_m - 1) = n-1 = \dim g(\eta)^\perp.$$
    
    We now show that $\Phi|_{\Delta^{\phi_1,\ldots,\phi_m}}$ sends cones to (unions of) cones. For each $i \in \{1,\ldots,m\}$ and each indecomposable $Y_{j,\ell}^i \in \mods \Lambda_{r_i}$, denote
     $$
     V_{Y_{j,\ell}^i} := \{(v_1,\ldots,v_m) \in \Delta^{\phi_1,\ldots,\phi_m} \subseteq \RR^{r_1+\cdots+r_m} \mid v_i \in D(Y_{j,\ell}^i)\}.
     $$
     It then follows immediately from Propositions~\ref{prop:NakayamaEqns} and~\ref{prop:tameEqns} that, for each such $Y_{j,\ell}^i$, we have $\Psi|_{\Delta^{\phi_1,\ldots,\phi_m}}(V_{Y_{j,\ell}^i}) = D_{reg}(X_{j,\ell}^i)$. Since every cone in $\ominus_{i = 1}^m \fD(\Lambda_{r_i})^{\phi_i}$ can be written as the intersection of sets of the form $V_{Y_{j,\ell}^i}$, we conclude that the image of any cone in $\ominus_{i = 1}^m \fD(\Lambda_{r_i})^{\phi_i}$ is a cone in $\fD_{reg}(H)$ of the same dimension.
     
     It remains to show that $(\Psi|_{\Delta^{\phi_1,\ldots,\phi_m}})^{-1}$ sends cones to unions of cones. To see this, we note that Lemma~\ref{lem:decompose D(eta)} can be used to restrict to cones of the form $D_{reg}(X_{j,r_i}^i)$. The result then follows analogously as for $\Psi|_{\Delta^{\phi_1,\ldots,\phi_m}}$.
\end{proof}

%%%%%%%%%%%%%%%%%%%%%%%%%%%%%%%%%%%%%%%%%%%%%%%%%%%%%%

\section{Support regular rigid objects}\label{sec:regular rigid}
Let $H$ be a tame hereditary algebra. In this section, we define \emph{support regular rigid objects} and \emph{support regular clusters}, meant to be regular analogues of support rigid objects and clusters. We use these notions to describe a polyhedral fan (Definition \ref{def:polyhedral fan}) which is closely related to the regular wall-and-chamber structure (Definition \ref{def:regular stability}).

We begin by defining \emph{projective vectors} which take the place of projective modules in the category of regular modules.

%%%%%%%%%%%%%%%%%%%%%%%%%%%%%%%%%%%%%%%%%%%%%%%%%%%%%%
\subsection{Definitions and basic properties}

\begin{defn}\label{def:projective vector}
	Let $H$ be a tame hereditary algebra. Let $\cX = \{X_{j_1,1}^1,\ldots,X_{j_m,1}^m\}$ be a choice of one quasi simple module from each exceptional tube in $\Reg H$. (Recall that our convention is to consider one of the tubes over the Kronecker to be exceptional.) Consider the intersection of the infinitesimal semi-invariant domains of the other quasi simple modules:
	$$
	\bigcap_{X_{j,1}^i \notin \cX} D_{0,g(\eta)}(X_{j,1}^i)=g(\eta)^\perp\cap \bigcap_{X_{j,1}^i \notin \cX} (\undim X_{j,1}^i)^\perp.
	$$
	By Proposition \ref{prop:regular linear algebra}, this intersection is a 1-dimensional (linear) subspace of $g(\eta)^\perp$. Moreover, by the same proposition, $\eta$ together with $\{\undim X_{j,1}^i \mid X_{j,1}^i \notin \cX\}$ is a basis of $g(\eta)^\perp$. We conclude that there exists a unique vector $p(\cX)$ in this intersection satisfying $p(\cX)\cdot \eta = 1$. We refer to $p(\cX)$ as the \emph{projective vector} of $\cX$.
\end{defn}

\begin{rem}
    A construction similar to our projective vectors also appears in \cite[Proposition~6.2]{IPT_semi-stable}. Indeed, they consider the convex cone in $g(\eta)^\perp$ spanned by the dimension vectors of the quasi simple modules and $\eta$. They then show that the sets $\mathcal{X}$ as in Definition~\ref{def:projective vector} correspond bijectively with the boundary facets of this cone. A key difference in this approach compared to ours is that \cite{IPT_semi-stable} considers the linear span of (the dimension vectors of the) quasi simple modules which are not contained in $\mathcal{X}$, while we consider the intersection of their orthogonal spaces.
\end{rem}

\begin{prop}\label{prop:multiplicity}
	Let $\cX = \{X_{j_1,1}^1,\ldots,X_{j_m,1}^m\}$ be as in Definition \ref{def:projective vector}. Let $X_{j,\ell}^i$ be an indecomposable regular module in the tube $\cT_i$. Then $p(\cX)\cdot\undim X_{j,\ell}^i$ is the multiplicity of $X_{j_i,1}^i$ in the regular composition series of $X_{j,\ell}^i$.
\end{prop}

\begin{proof}
	By definition, $p(\cX)\cdot \undim X_{k,1}^i = 0$ if $k \neq j_i$. Thus since $p(\cX)\cdot \eta =1$, we have $p(\cX)\cdot \undim X_{j_i,1}^i = 1$. The result then follows immediately from Proposition \ref{prop:tameEqns}.
\end{proof}

As an immediate corollary, we have the following.

\begin{cor}\label{cor:multiplicity}
	Let $\cX = \{X_{j_1,1}^1,\ldots,X_{j_m,1}^m\}$ be as in Definition \ref{def:projective vector} and let $M \in \Reg H$ have no homogeneous direct summand. Then $p(\cX)\cdot\undim M$ is the sum of the multiplicities of the $X_{j_i,1}^i$ in the regular composition series of $M$.
\end{cor}

\begin{rem}\label{rem:always nonnegative}
	If $M$ is homogeneous then $\undim M = k\eta$ for some $k$. Then by definition, $p(\cX)\cdot\undim M = k$. In particular, this and Proposition \ref{prop:multiplicity} mean $p(\cX)\cdot \undim N \geq 0$ for all regular $N$.
\end{rem}

\begin{rem}\label{rem:projective vectors}
	Recall that for an arbitrary $M \in \mods H$, we have that $g(P(j))\cdot \undim M$ is equal to the multiplicity of $S(j)$ in the composition series of $M$. Thus Proposition \ref{prop:multiplicity} and Corollary \ref{cor:multiplicity} give our rationale for calling the vectors $p(\cX)$ projective vectors.
\end{rem}

\begin{eg}\label{eg:projective vectors}
	Consider the quiver $1 \rightarrow 2 \rightarrow 3 \leftarrow 4 \leftarrow 1$ of type $\widetilde{A}_3$. There are two tubes of rank 2, with quasi simples $X_{1,1}^1 = 4, X_{2,1}^1 = 123, X_{1,1}^2 = 2$, and $X_{2,1}^2 = 143$. This quiver has $\eta = (1,1,1,1)$ and $g(\eta) = (1,0,-1,0)$. The four projective vectors are:
	\begin{eqnarray*}
		p(2,4) &=& (-1/2,1,-1/2,1) \in D_{reg}(143)\cap D_{reg}(123) \cap g(\eta)^\perp\\
		p(2,123) &=& (0,1,0,0) \in D_{reg}(143)\cap D_{reg}(4) \cap g(\eta)^\perp\\
		p(143,4) &=& (0,0,0,1) \in D_{reg}(2)\cap D_{reg}(123) \cap g(\eta)^\perp\\
		p(143,123) &=& (1/2,0,1/2,0) \in D_{reg}(2)\cap D_{reg}(4) \cap g(\eta)^\perp
	\end{eqnarray*}
\end{eg}

Even though there are no projectives in the category of regular modules, we can still relate $p(X)\cdot \undim M$ to the dimension of a hom-space.

\begin{prop}\label{prop:long homs}
	Let $\cX = \{X_{j_1,1}^1\ldots,X_{j_m,1}^m\}$ be as in Definition \ref{def:projective vector} and let $\ell > 0$ be a positive integer.  Let $L_{\cX,\ell}$ be the direct sum (over $i$) of the modules $X_{j_i-\ell+1,\ell}^i$. (Note that $X_{j_i-\ell+1,\ell}^i$ has quasi length $\ell$ and quasi top $X_{j_i,1}^i$). Now let $M \in \Reg H$ have no homogeneous direct summand and suppose every direct summand has quasi-length at most $\ell$.
	Then $p(\cX)\cdot\undim M = \dim_K\Hom_H(L_{\cX,\ell},M)$.
\end{prop}

\begin{proof}
	Let $X_{j,\ell'}^i$ be an indecomposable direct summand of $M$. Since the only direct summand of $L_{\cX,\ell}$ in $\cT_i$ is $X_{j_i-\ell+1,\ell}^i$, we have
	$$\dim_K\Hom_H(L_{\cX,\ell},X_{j,\ell'}^i) = \dim_K\Hom_H(X_{j_i-\ell+1,\ell}^i,X_{j,\ell'}^i).$$
	Since $\ell \geq \ell'$ and the quasi top of $X_{j_i-\ell+1,\ell}^i$ is $X_{j_i,1}^i$, this is equal to the multiplicity of $X_{j_i,1}^i$ in the regular composition series of $X_{j,\ell'}^i$. The result then follows from Corollary \ref{cor:multiplicity}.
\end{proof}

We are now ready to define support regular rigid objects.

\begin{defn}\label{def:regular rigid}
	Let $H$ be a tame hereditary algebra. Let $M \in \Reg H$ be basic and let $\cP^+,\cP^-$ be sets of projective vectors. We call the triple $(M,\cP^+,\cP^-)$ \emph{support regular rigid} if
	\begin{enumerate}
		\item $\Hom_H(M,\tau M) = 0$.
		\item For all $p \in \cP^+$, $p\cdot\undim \tau M = 0$.
		\item For all $p \in \cP^-$, $p\cdot\undim M = 0$.
		\item At least one of $\cP^+$ and $\cP^-$ is empty.
	\end{enumerate}
	If in addition $(M,\cP^+,\cP^-)$ is maximal (see Remark \ref{rem:contains} below), we call it a \emph{support regular cluster}. We say $(M,\cP^+,\cP^-)$ is \emph{null-nonnegative} if $\cP^- = \emptyset$ and \emph{null-nonpositive} if $\cP^+ = \emptyset$.
\end{defn}

\begin{rem}\label{rem:contains}
	We say the support regular rigid object $(N,\cQ^+,\cQ^-)$ contains $(M,\cP^+,\cP^-)$ if $M$ is a direct summand of $N$, $\cP^+\subseteq \cQ^+$, and $\cP^-\subseteq \cQ^-$. The condition that $(M,\cP^+,\cP^-)$ is maximal means that if $(N,\cQ^+,\cQ^-)$ contains $(M,\cP^+,\cP^-)$ then $(N,\cQ^+,\cQ^-) = (M,\cP^+,\cP^-)$.
\end{rem}

\begin{rem}\label{rem:regular rigid}
	Definition \ref{def:regular rigid} is meant to serve as a regular analogue of the definition of a support $\tau$-rigid object. Since projective vectors are taking the place of projective modules, conditions (1) and (2) together are analogous to specifying that a module which potentially includes projective direct summands is $\tau$-rigid. Likewise, conditions (3) and (4) take the place of the condition that $\Hom_H(P,M) = 0$ when $M\oplus P[1]$ is support $\tau$-rigid. Further justification for why condition (4) is necessary is given in Proposition \ref{prop:one projective}.
\end{rem}

\begin{rem}
    Note that there are precisely three support regular rigid objects in the case of the Kronecker: $(0,\emptyset,\emptyset)$, $(0,\{p(X_{1,1}^1)\},\emptyset)$, and $(0,\emptyset,\{p(X_{1,1}^1)\})$.
\end{rem}

\begin{nota}\label{nota:proj sets}\
    \begin{enumerate}
        \item Let $\cP$ be a set of projective vectors. We denote by
    $$\tp(\cP) := \bigcup_{p(\cX) \in \cP}\cX.$$
         \item Let $\cY$ be a set of quasi-simple objects in $\Reg H$. We denote by
            $$\pc(\cY) := \left\{p(\cX) \mid \cX \subseteq \cY \text{ and } \forall \ i \in \{1,\ldots,m\}: |\cX \cap \cT_i| = 1\right\}.$$
            That is, given $\cX$ as in Definition~\ref{def:projective vector}, we have $p(\cX) \in \pc(\cY)$ if and only if $\cX \subseteq \cY$.
    \end{enumerate}
\end{nota}

\begin{rem}\label{rem:pctp}
    We emphasize that, in Notation~\ref{nota:proj sets}, both $\cP$ and $\pc(\cY)$ are sets of (projective) vectors while both $\tp(\cP)$ and $\cY$ are sets of quasi-simple modules. In light of Corollary~\ref{cor:multiplicity}, the set $\tp(\cP)$ can be thought of as the set of quasi-simple modules which are included in the ``top'' of some $p(\cX) \in \cP$. It is also possible to think of $\pc(\cY)$ as the set of ``projective covers'' of the modules in $\cY$; however, some caution is warranted. Indeed, $\tp(\pc(\cY)) = \cY$, if and only if either (a) $\cY \cap \cT_i \neq \emptyset$ for all $i$, or (b) $\cY = \emptyset$.  Otherwise, one has $\tp(\pc(\cY)) = \emptyset \neq \cY$. These conditions (a) and (b) will correspond to the conditions (a) and (b) in Proposition~\ref{prop:composition}(1,2), where they will be used to relate support regular rigid objects over $\Lambda$ to support $\tau$-rigid objects over the self-injective Nakayama algebras $\Lambda_{r_i}$. Similarly, it may be the case that $\cP \subsetneq \pc(\tp(\cP))$, as demonstrated by the following example.
\end{rem}

\begin{eg}
    Suppose $H$ has two exceptional tubes of rank 2. Denote
    \begin{eqnarray*}
        \cP_1 &=& \{p(Y_{1,1}^1,Y_{1,1}^2), p(Y_{2,1}^1,Y_{2,1}^2)\}\\
        \cP_2 &=& \{p(Y_{1,1}^1,Y_{2,1}^2), p(Y_{2,1}^1,Y_{1,1}^2)\}.
    \end{eqnarray*}
    Then $\tp(\cP_1) = \tp(\cP_2) = \{Y_{1,1}^1,Y_{2,1}^1,Y_{1,1}^2,Y_{2,1}^2\}$ and $\pc(\tp(\cP_1)) = \pc(\tp(\cP_2)) = \cP_1 \cup \cP_2$.
\end{eg}

Based on this example, we give the following definition.

\begin{defn}\label{def:projectively closed}
	Let $(M,\cP^+,\cP^-)$ be support regular rigid. We say that $(M,\cP^+,\cP^-)$ is \emph{projectively closed} if $\cP^+ = \pc(\tp(\cP^+))$ and $\cP^- = \pc(\tp(\cP^-))$.
\end{defn}

For $(M,\cP^+,\cP^-)$, denote $\cQ^+ = \pc(\tp(\cP^+))$ and $\cQ^- = \pc(\tp(\cP^-))$. It is straightforward to show that $(M,\cQ^+,\cQ^-)$ is a projectively closed support regular rigid object. Moreover, any projectively closed regular rigid object which contains $(M,\cP^+,\cP^-)$ will also contain $(M,\cQ^+,\cQ^-)$. This motivates the following definition.

\begin{defn}\label{def:projective closure}
    Let $(M,\cP^+,\cP^-)$ be support regular rigid, and denote $\cQ^+ = \pc(\tp(\cP^+))$ and $\cQ^- = \pc(\tp(\cP^-))$, Then we call $(M,\cQ^+,\cQ^-)$ the \emph{projective closure} of $(M,\cP^+,\cP^-)$.
\end{defn}

\begin{rem}\label{rem:preserve null side}
	We observe that the projective closure of $(M,\cP^+,\cP^-)$ is null-nonnegative (resp. null-nonpositive) if and only if $(M,\cP^+,\cP^-)$ is null-nonnegative (resp. null-nonpositive).
\end{rem}

\begin{nota}\label{notation:regular rigid}
	We denote by $\mathsf{srr}(H)$ the set of (isoclasses of) support regular rigid objects for $H$. We denote by $\mathsf{srr}_c$, $\mathsf{srr}^+$ and $\mathsf{srr}^-$ the sets of projectively closed, null-nonnegative, and null-nonpositive support regular rigid objects, respectively. $\mathsf{srr}_c^+$ and $\mathsf{srr}_c^-$ are defined in the natural way.
\end{nota}

%%%%%%%%%%%%%%%%%%%%%%%%%%%%%%%%%%%%%%%%%%%%%%%%%%%%%%

\subsection{Relationship to support $\tau$-rigid objects for self-injective Nakayama algebras}\label{sec:combining chambers}
In this section, we relate support regular rigid objects to certain collections of support $\tau$-rigid objects for the self-injective Nakayama algebras $\Lambda_r$. We begin by proving several results about the support $\tau$-rigid objects for $\Lambda_r$.

\begin{prop}\label{prop:one projective}
	Let $M\oplus P[1] \in \mathsf{str}(\Lambda_r)$. Then $P = 0$ or $M$ has no projective direct summand.
\end{prop}
\begin{proof}
	Suppose $P \neq 0$ and there exists an indecomposable projective direct summand $Q$ of $M$. It follows that the length of $Q$ is $r+1$ and $Q$ is supported at every vertex of $Z_r$. Thus there exists a nonzero morphism $P \rightarrow Q$, a contradiction.
\end{proof}

\begin{rem}\label{rem:one projective}
	Proposition \ref{prop:one projective} is our main justification for condition (4) in Definition~\ref{def:regular rigid}.
\end{rem}

\begin{defn}\label{def:null-nonnegative}
	Let $M\oplus P[1]\in\mathsf{str}(\Lambda_r)$. We say $M\oplus P[1]$ is \emph{null-nonnegative} if $(g(M)-g(P))\cdot \overline{1} \geq 0$. Likewise, we say $M\oplus P[1]$ is \emph{null-nonpositive} if $(g(M)-g(P))\cdot \overline{1} \leq 0$. We denote by $\mathsf{str}^+\Lambda_r$ and $\mathsf{str}^-\Lambda_r$ the sets of null-nonnegative and null-nonpositive support $\tau$-rigid objects for $\Lambda_r$, respectively, and similarly for $\mathsf{stt}^+(\Lambda_r)$ and $\mathsf{stt}^-(\Lambda_r)$.
\end{defn}

\begin{prop}\label{prop:null-nonnegative}
	Let $M\oplus P[1]\in\mathsf{str}(\Lambda_r)$. Then
	\begin{enumerate}
		\item $M\oplus P[1]$ is null-nonnegative if and only if $P = 0$.
		\item $M\oplus P[1]$ is null-nonpositive if and only if $M$ contains no projective direct summand.
	\end{enumerate}
\end{prop}

\begin{proof}
	Let $N \in \mods\Lambda_r$ be indecomposable $\tau$-rigid. Let $Y_{j,1}$ be the top of $N$, and let $Y_{j',1}$ be the socle of $N$. Direct computation then shows 
	$$g(M) = \begin{cases} e_{j} & M \text{ is projective}\\e_{j}-e_{j'-1} & \text{otherwise.}\end{cases}$$
	Thus $g(N)\cdot\overline{1} = 1$ if $N$ is projective and $g(N)\cdot\overline{1} = 0$ otherwise. The result then follows from Proposition \ref{prop:one projective}.
\end{proof}

\begin{lem}\label{lem:complete all positive}
	Let $M \in \mathsf{str}(\Lambda_r)$. Then there exists a vertex $v \in Z_r$ on which $\tau M$ is not supported. By symmetry, if $M$ contains no projective direct summand, then there exists a vertex $v' \in Z_r$ on which $M$ is not supported.
\end{lem}

\begin{proof}
	Let $Y_{j,\ell}$ be an indecomposable non-projective direct summand of $M$ of maximal length. We note that since $Y_{j,\ell}$ is not projective, $\ell \leq r-1$ by Proposition \ref{prop:Nakayama}. Recall that $S(j+\ell-1)$ is the top of $Y_{j,\ell}$. Now let $Y_{j',\ell'}$ be an indecomposable direct summand of $M$. Then either $\tau Y_{j',\ell} = 0$ or $\ell' \leq \ell$. In either case, the fact that $M$ is $\tau$-rigid means that $\tau Y_{j',\ell'}$ cannot be supported at vertex $j+\ell-1$. Indeed, otherwise we would have a nonzero morphism $Y_{j,\ell} \rightarrow \tau Y_{j',\ell'}$.
\end{proof}

As an immediate consequence of Lemma \ref{lem:complete all positive} we have the following.

\begin{prop}\label{prop:complete all positive}\
\begin{enumerate}
    \item Let $M\in\mathsf{str}(\Lambda_r)$, and suppose that $M$ contains no projective direct summand. Then there exist nonzero projectives $P,Q$ so that both $M \oplus P$ and $M \oplus Q[1]$ are support $\tau$-rigid. In particular, $M$ is contained in both a null-nonnegative support $\tau$-tilting object and a null-nonpositive support $\tau$-tilting object.
    \item Let $M \oplus P[1] \in \mathsf{stt}(\Lambda_r)$. Then $P = 0$ if and only if $M$ has a projective direct summand. That is, $\mathsf{stt}^+(\Lambda_r) \cap \mathsf{stt}^-(\Lambda_r) = \emptyset$.
    \item For $M \in \mathsf{str}(\Lambda_r)$ containing no projective direct summands, denote by $P_M^+$ and $P_M^-$ the direct sums of all indecomposable projectives $Q$ which satisfy $\Hom_{\Lambda_r}(Q,\tau M)$ and $\Hom_{\Lambda_r}(Q,M)$, respectfully. Then there is a bijection $\mathsf{stt}^+(\Lambda_r) \rightarrow \mathsf{stt}^-(\Lambda_r)$ given by $M \oplus P_M^+ \mapsto M \oplus P_M^-[1]$.
\end{enumerate}
\end{prop}

\begin{proof}
    (1) Let $P(v)$ and $P(v')$ be the projective covers of the simples at the vertices $v$ and $v'$ in the notation of Lemma~\ref{lem:complete all positive}. Then both $M \oplus P(v)$ and $M \oplus P(v')[1]$ are support $\tau$-rigid. Now any support $\tau$-tilting object containing $M\oplus P(v)$ will be null-nonpositive and any support $\tau$-tilting object containing $M \oplus P(v')[1]$ will be null-nonnegative by Propositions~\ref{prop:one projective} and~\ref{prop:null-nonnegative}.
    
    (2) Suppose $M \oplus P[1] \in \mathsf{stt}^+(\Lambda_r) \cap \mathsf{stt}^-(\Lambda_r)$. Then $P = 0$ and $M$ has no projective direct summand by Proposition~\ref{prop:null-nonnegative}. But then $M$ cannot be $\tau$-tilting by (1).
    
    (3) Let $M \in \mathsf{str}(\Lambda_r)$, and suppose that $M$ contains no projective direct summand. Denote $P_M^+$ and $P_M^-$ as in the statement, and note that each of these is uniquely determined by $M$. Now by the remark preceeding Proposition~\ref{prop:NakayamaEqns}, we have that $M$ and $\tau M$ are supported on the same number of vertices. Thus $P_M^+$ and $P_N^-$ have the same number of direct summands, and so $M \oplus P_M^+ \in \mathrm{stt}^+(\Lambda_r)$ if and only if $M \oplus P_M^-[1] \in \mathrm{stt}^-(\Lambda_r)$ by Proposition~\ref{prop:null-nonnegative}. This implies the result.
\end{proof}

We now define maps which allow us to move between support regular rigid objects for $H$ and support $\tau$-rigid objects for the $\Lambda_r$.

\begin{prop}\label{def:srr to str}
	Let $H$ be a tame hereditary algebra with exceptional tubes $\cT_1,\ldots,\cT_m$. For $1 \leq i \leq m$, there is a map $\rho_i:\mathsf{srr}(H)\rightarrow \mathsf{str}(\Lambda_{r_i})$ given as follows:
	
	Let $(M,\cP^+,\cP^-)$ be a support regular rigid object. Let $\cM_i$ be the set of indecomposable direct summands of $M$ in the tube $\cT_i$. Then $\rho_i(M,\cP^+,\cP^-) = N \sqcup Q[1]$, where
	\begin{eqnarray*}
		N &=& \left(\bigoplus_{X_{j,\ell}^i \in \cM_i}Y_{j,\ell}^i\right)\bigoplus\left(\bigoplus_{X_{j,1}^i\in\cT_i \cap \tp(\cP^+)}Y_{j,r_i+1}^i\right)\\
		Q &=& \bigoplus_{X_{j,1}^i\in\cT_i \cap \tp(\cP^-)}Y_{j,r_i+1}^i
	\end{eqnarray*}
\end{prop}

\begin{proof}
	Let $(M,\cP^+,\cP^-)$ be a support regular rigid object. We first recall that $Y_{j,r_i+1}^i$ is projective for all $j$. Moreover, for all $X_{j,1}^i \in \cT_i \cap \tp(\cP^+)$, we have $\Hom(X_{j,r_i+1}^i, \tau M) = 0$ by Proposition \ref{prop:long homs}. Proposition \ref{prop:same homs} then implies that $\Hom(N,\tau N) = 0$. Finally, if $\cP^- \neq \emptyset$, then $\cP^+ = \emptyset$ and $\Hom(Q,N) = 0$ by Propositions \ref{prop:long homs} and \ref{prop:same homs}. We conclude that the map $\rho_i$ is well-defined.
\end{proof}

\begin{rem}
    The formulas for $N$ and $Q$ in Proposition~\ref{def:srr to str} above say that (a) a non-projective $\tau$-rigid module $Y_{j,\ell}^i \in \mods\Lambda_{r_i}$ is a direct summand of $\rho_i(M,\cP^+,\cP^-)$ if and only if $X_{j,\ell}^i$ is a direct summand of $M$, (b) a projective $Y_{j,r_i+1}^i \in \mods\Lambda_{r_i}$ is a direct summand of $\rho_i(M,\cP^+,\cP^-)$ if and only if $X_{j,1}^i \in \tp(\cP^+)$, and (c) a ``shifted'' projective $Y_{j,r_i+1}^i[1] \in \mods\Lambda_{r_i}[1]$ is a direct summand of $\rho_i(M,\cP^+,\cP^-)$ if and only if $X_{j,1}^i \in \tp(\cP^-)$. A consequence is that 
    $$\sum_{i = 1}^m |\rho_i(M,\cP^+,\cP^-)| = |M| + |\tp(\cP^+)| + |\tp(\cP^-)|.$$

\end{rem}

\begin{rem}\label{rem:srr to str}
	Let $H$ be a tame hereditary algebra and let $\cT_i$ be an exceptional tube in $\mods H$. Then by Proposition \ref{prop:one projective}, $\rho_i$ induces maps $\rho_i^+:\mathsf{srr}^+H \rightarrow \mathsf{str}^+\Lambda_{r_i}$ and $\rho_i^-:\mathsf{srr}^-H \rightarrow \mathsf{str}^-\Lambda_{r_i}$.
\end{rem}

\begin{prop}\label{def:str to srr}
	Let $H$ be a tame hereditary algebra with exceptional tubes $\cT_1,\ldots,\cT_m$. Then there are maps
	\begin{eqnarray*}
		\iota^+&:& \mathsf{str}^+(\Lambda_{r_1})\times\cdots\times\mathsf{str}^+(\Lambda_{r_m}) \rightarrow \mathsf{srr}_c^+H\\
		\iota^-&:& \mathsf{str}^-(\Lambda_{r_1})\times\cdots\times\mathsf{str}^-(\Lambda_{r_m}) \rightarrow \mathsf{srr}_c^-H
	\end{eqnarray*}
	given as follows.
	\begin{enumerate}
		\item Let $(M_1,\ldots,M_m) \in \mathsf{str}^+(\Lambda_{r_1})\times\cdots\times\mathsf{str}^+(\Lambda_{r_m})$. For each $i$, let $\cM_i$ be the set of non-projective indecomposable direct summands of $M_i$. Denote
		    $$\cY = \bigcup_{i = 1}^m \left\{X_{j,1}^i \mid Y_{j,r_i + 1}^i \text{ is a direct summand of }M_i\right\}.$$
    Then $\iota^+(M_1,\ldots,M_m) = (M,\pc(\cY),\emptyset)$, where
		$$M :=\bigoplus_{Y_{j,\ell}^i \in \cM_i} X_{j,\ell}^i.$$
		
		\item Let $(M_1\oplus P_1[1],\ldots,M_m\oplus P_m[1]) \in \mathsf{str}^-(\Lambda_{r_1})\times\cdots\times\mathsf{str}^-(\Lambda_{r_m})$. Let $\cM_i$ be the set of (non-projective) direct summands of $M_i$. Denote
		    $$\cY = \bigcup_{i = 1}^m \left\{X_{j,1}^i \mid Y_{j,r_i + 1}^i \text{ is a direct summand of }P_i\right\}.$$
    Then $\iota^-(M_1\oplus P_1[1],\ldots,M_m\oplus P_m[1]) = (M,\emptyset,\pc(\cY))$, where
    $$M :=\bigoplus_{Y_{j,\ell}^i \in \cM_i} X_{j,\ell}^i.$$
	\end{enumerate}
\end{prop}

\begin{proof}
	We prove only (1) as the proof of (2) is similar. Let $(M_1,\ldots,M_m) \in \mathsf{str}^+(\Lambda_{r_1}) \times\cdots\times \mathsf{str}^+(\Lambda_{r_m})$. Then $\Hom(M,\tau M) = 0$ as a result of Proposition \ref{prop:same homs}. Likewise, we have $p(\cX) \cdot \undim \tau M = 0$ for $\cX \in \pc(\cY)$ by Propositions \ref{prop:long homs} and \ref{prop:same homs}. We conclude that $(M,\pc(\cY), \emptyset)$ is support regular rigid. It is clear that $(M,\pc(\cY),\emptyset)$ is projectively closed.
\end{proof}

The remainder of this section uses the maps $\iota^+,\iota^-$, and $\rho_i$ to determine when a support regular rigid object is a support regular cluster. Note in particular that conditions (a) and (b) Proposition~\ref{prop:composition}(1,2) below correspond to the conditions (a) and (b) in Remark~\ref{rem:pctp}.

\begin{prop}\label{prop:composition} Let $H$ be a tame hereditary algebra.
	\begin{enumerate}
		\item Let $(M_1,\ldots,M_m) \in \mathsf{str}^+(\Lambda_{r_1})\times\cdots\times\mathsf{str}^+(\Lambda_{r_m})$. If either (a) each $M_i$ contains a projective direct summand or (b) no $M_i$ contains a projective direct summand,  then $M_i = \rho^+_i\circ \iota^+(M_1\,\ldots,M_m)$. Otherwise, $\rho^+_i\circ \iota^+(M_1\,\ldots,M_m)$ is the direct summand of $M_i$ obtained by deleting all projective direct summands.
		\item Let $(M_1\oplus P_1[1],\ldots,M_m\oplus P_m[1]) \in \mathsf{str}^-(\Lambda_{r_1})\times\cdots\times\mathsf{str}^-(\Lambda_{r_m})$. If either (a) each $P_i$ is nonzero or (b) no $P_i$ is nonzero, then $M_i\oplus P_i[1] = \rho^-_i\circ \iota^-(M_1\oplus P_1[1],\ldots,M_m\oplus P_m[1])$. Otherwise, $\rho^-_i\circ \iota^-(M_1\oplus P_1[1],\ldots,M_m\oplus P_m[1]) = M_i$.
		\item Let $(M,\cP^+,\emptyset) \in \mathsf{srr}^+H$ (possibly with $\cP^+ = \emptyset$). Then $\iota^+ \circ (\rho_1, \ldots, \rho_m)(M,\cP^+,\emptyset)$ is the projective closure of $(M,\cP^+,\emptyset)$.
		
		\item Let $(M,\emptyset,\cP^-) \in \mathsf{srr}^-H$ (possibly with $\cP^- = \emptyset$). Then $\iota^- \circ (\rho_1, \ldots, \rho_m)(M,\emptyset,\cP^-)$ is the projective closure of $(M,\emptyset,\cP^-)$. 
	\end{enumerate}
\end{prop}

\begin{proof}
	All four claims follow immediately from Proposition \ref{prop:same homs} and the definitions of the maps $\iota^+,\iota^-$, and $\rho_i$.
\end{proof}

\begin{rem}\label{rem:compositions}
    Note that Proposition~\ref{prop:composition}(1,3) together imply that there is a bijection
    $$(\rho_1^+,\ldots,\rho_m^+): \mathsf{srr}_c^+(H) \rightarrow \{(M_1,\ldots,M_m) \in \mathsf{str}^+(\Lambda_{r_1}) \times \cdots \times \mathsf{str}^+(\Lambda_{r_m}) \mid \text{(a) or (b)}\},$$
    where (a) and (b) are the conditions in Proposition~\ref{prop:composition}(1). Similarly, Proposition~\ref{prop:composition}(2,4) together imply that there is a bijection
    $$(\rho_1^-,\ldots,\rho^-_m): \mathsf{srr}_c^-(H) \rightarrow \{(M_1,\ldots,M_m) \in \mathsf{str}^-(\Lambda_{r_1}) \times \cdots \times \mathsf{str}^-(\Lambda_{r_m}) \mid \text{(a) or (b)}\},$$
    where (a) and (b) are the conditions in Proposition~\ref{prop:composition}(2).
\end{rem}

\begin{prop}\label{prop:inclusions}
	Let $H$ be a tame hereditary algebra. 
	\begin{enumerate}
		\item Let $(M,\cP^+,\emptyset)$ be a null-nonnegative support regular rigid object. For each $i$, let $M_i$ be a null-nonnegative support $\tau$-rigid object for $\Lambda_{r_i}$. Then $\iota^+(M_1,\ldots,M_m)$ contains $(M,\cP^+,\emptyset)$ if and only if $\rho_i(M,\cP^+,\emptyset)$ is contained in $M_i$ for all $i$.
		\item Let $(M,\emptyset,\cP^-)$ be a null-nonpositive support regular rigid object. For each $i$, let $M_i\oplus P_i[1]$ be a null-nonpositive support $\tau$-rigid object for $\Lambda_{r_i}$. Then $\iota^-(M_1\oplus P_1[1],\ldots,M_m\oplus P_m[1])$ contains $(M,\emptyset,\cP^-)$ if and only if $\rho_i(M,\emptyset,\cP^-)$ is contained in $M_i\oplus P_i[1]$ for all $i$.
		\item Let $(M_1,\ldots,M_m) \in \mathsf{str}^+ \Lambda_{r_1}\times\cdots\times \mathsf{str}^+\Lambda_{r_m}$ and let $(M,\cP^+,\emptyset)$ be a null-nonnegative support regular object containing $\iota^+(M_1,\ldots,M_m)$. If either (a) each $M_i$ contains a projective direct summand or (b) no $M_i$ contains a projective direct summand, then $\rho_i(M,\cP^+,\emptyset)$ contains $M_i$ for all $i$.
		\item Let $(M_1\oplus P_1[1],\ldots,M_m\oplus P_m[1]) \in \mathsf{str}^-\Lambda_{r_1}\times\cdots\times \mathsf{str}^-\Lambda_{r_m}$ and let $(M,\emptyset,\cP^-)$ be a null-nonpositive support regular object containing $\iota^-(M_1\oplus P_1[1],\ldots,M_m\oplus P_m[1])$. If either (a) each $P_i$ is nonzero or (b) no $P_i$ is nonzero, then $\rho_i(M,\emptyset,\cP^-)$ contains $M_i\oplus P_i[1]$ for all $i$.
	\end{enumerate}
\end{prop}

\begin{proof}
    We prove only (2) and (4), as the proofs of (1) and (3) are similar.

    (2) Recall from Remark~\ref{rem:srr to str} that $\rho_i(M,\emptyset,\mathcal{P}^-)\in \mathsf{str}^-(\Lambda_{r_i})$ for all $i$. By the definitions of the maps $\iota^-$ and $\rho_i$, one therefore obtains equivalences $(2a) \iff (2b) \iff (2c)$ for the following conditions.
    \begin{itemize}
        \item[(2a)] $\iota^-(M_1\oplus P_1[1],\ldots,M_m \oplus P_m[1])$ contains $(M,\emptyset,\mathcal{P}^-)$.
        \item[(2b)]For all $i, j, \ell$ both (a) if $X_{j,\ell}^i$ is a direct summand of $M$ then $Y_{j,\ell}^i$ is a direct summand of $M_i$, and (b) if $\ell = 1$ and $X_{j,1}^i \in \tp(\mathcal{P}^-)$, then $Y_{j,r_i+1}^i$ is a direct summand of $P_i$.
        \item[(2c)] $\rho_i(M,\emptyset,\mathcal{P}^-)$ is contained in $M_i \oplus P_i[1]$ for all $i$.
    \end{itemize}

    (4) Let $(M,\emptyset,\mathcal{P}^-)$ be a null-nonpositive support regular rigid object which contains $\iota^-(M_1\oplus P_1[1],\ldots,M_m\oplus P_m[1])$. From the definition of $\iota^-$, it follows that, for all $i, j, \ell$, one has that if $Y_{j,\ell}^i$ is a direct summand of $M_i$ then $X_{j,\ell}^i$ is a direct summand of $M$. (Note that we have implicitly used the fact that $Y_{j,r_i + 1}^i$ cannot be a direct summand of $M_i$ by Proposition~\ref{prop:null-nonnegative}.) By the definition of $\rho_i$, this implies that $\rho_i(M,\emptyset,\mathcal{P}^-)$ contains $M_i$ for all $i$.
    
    It remains to show that $\rho_i(M,\emptyset,\mathcal{P}^-)$ contains $P_i[1]$ for all $i$. If we are in case (b) of the statement there is nothing to show, so suppose we are in case (a). From the definition of $\iota^-$, it follows that, for all $i, j$, if $Y_{j,r_i+1}^i$ is a direct summand of $P_i$, then $X_{j,1}^i \in \tp(\mathcal{P}^-)$. By the definition of $\rho_i$, this implies the result.
\end{proof}

We now characterize support regular clusters in terms of the maps $\rho_i$.

\begin{thm}\label{thm:maximal} Let $H$ be a tame hereditary algebra.
	\begin{enumerate}
		\item A null-nonnegative regular rigid object $(M,\cP^+,\emptyset)$ is a support regular cluster if and only if the following hold.
		\begin{enumerate}
			\item $\rho_i^+(M,\cP^+,\emptyset)$ is support $\tau$-tilting for all $1 \leq i \leq m$.
			\item $(M,\cP^+,\emptyset)$ is projectively closed.
		\end{enumerate}
		\item A null-nonpositive regular rigid object $(M,\emptyset,\cP^-)$ is a support regular cluster if and only if the following hold.
		\begin{enumerate}
			\item $\rho_i^-(M,\emptyset,\cP^-)$ is support $\tau$-tilting for all $1 \leq i \leq m$.
			\item $(M,\emptyset,\cP^-)$ is projectively closed.
		\end{enumerate}
	\end{enumerate}
\end{thm}

\begin{proof}
	We prove (1) as the proof of (2) is similar. First suppose $(M,\cP^+,\emptyset)$ is a support regular cluster. Since it is contained in its projective closure, it must be projectively closed. Now for each $i$, by Proposition \ref{prop:complete all positive}(1), there exists $M_i$ a null-nonnegative support $\tau$-tilting object containing $\rho_i(M,\cP^+,\emptyset)$. Then by Proposition \ref{prop:inclusions}(1,3) we have that $\iota^+(M_1,\ldots,M_m) = (M,\cP^+,\emptyset)$ and hence $\rho_i(M,\cP^+,\emptyset) = M_i$. 
	
	Now suppose conditions (a) and (b) hold and let $(N,\cQ^+,\emptyset)$ contain $(M,\cP^+,\emptyset)$. By condition (b) and Proposition \ref{prop:inclusions}(3), we then have that $\rho_i(N,\cQ^+,\emptyset) = \rho_i(M,\cP^+,\emptyset)$. By Proposition \ref{prop:inclusions}(1), this means $(M,\cP^+,\emptyset)$ contains $(N,\cQ^+,\emptyset)$; so, they are equal and $(M,\cP^+,\emptyset)$ is a support regular cluster.
\end{proof}

As an immediate corollary, we have the following analog of Proposition~\ref{prop:complete all positive}.

\begin{cor}\label{cor:complete} Let $H$ be a tame hereditary algebra.
    \begin{enumerate}
        \item Let $(M,\emptyset,\emptyset)$ be support regular rigid. Then there exist nonempty sets of projective vectors $\mathcal{P}^+$, $\mathcal{P}^-$ such that both $(M,\mathcal{P}^+,\emptyset)$ and $(M,\emptyset,\mathcal{P}^-)$ are support regular rigid. In particular, $(M,\emptyset, \emptyset)$ is contained in both a null-nonnegative support regular cluster and a null-nonpositive support regular cluster.
        \item Let $(M,\mathcal{P}^+,\mathcal{P}^-)$ be a support regular cluster. Then (exactly) one of $\mathcal{P}^+ \neq \emptyset$ and $\mathcal{P}^- \neq \emptyset$. In particular, every support regular cluster is either null-nonnegative or null-nonpositive, but not both.
        \item Let $(M,\emptyset,\emptyset)$ be support regular rigid. Let $\mathcal{P}_M^+$ and $\mathcal{P}_M^-$ be the sets of all projective vectors $p$ which satisfy $p\cdot \undim \tau M = 0$, and $p\cdot \undim M = 0$, respectively. Then there is a bijection from the set of null-nonnegative support regular cluster to the set of null-nonpositive support regular clusters given by $(M,\mathcal{P}^+,\emptyset) \mapsto (M,\emptyset,\mathcal{P}^-)$.
    \end{enumerate}

	Let $(M,\cP^+,\cP^-)$ be support regular rigid. Then there exists a support regular cluster which contains $(M,\cP^+,\cP^-)$.
\end{cor}

\begin{proof}
    (1) We prove the statement in the null-nonpositive case, as the proof in the null-nonnegative case is similar. For each $i$, denote $M_i = \rho_i(M,\emptyset,\emptyset)$. Then Proposition~\ref{prop:complete all positive}(1) implies that there exist a null-nonpositive support $\tau$-tilting pair of the form $M_i \oplus P_i[1]$. Moreover, we have $P_i \neq 0$ for all $i$ by Proposition~\ref{prop:complete all positive}(2).
    
    Now denote $(M,\emptyset,\mathcal{P}^-) = \iota^-(M_1\oplus P_1[1],\ldots,M_k \oplus P_k[1])$. (Note that the module part is $M$ by construction.) Then $\rho_i(N,\emptyset,\mathcal{Q}^-) = M_i \oplus P_i$ by Proposition~\ref{prop:composition}(2) and $\mathcal{P}^-$ is projective closed by the definition of $\iota^-$. Theorem~\ref{thm:maximal}(2) therefore implies the result.
    
    (2) This is an immediate consequence of (1). 
    
    (3) Let $(M,\emptyset,\emptyset) \in \mathsf{srr}(H)$. Denote $\mathcal{P}_M^+$ and $\mathcal{P}_M^-$ as in the statement, and note that each of these is uniquely determined by $M$. By their definition, we know that $\mathcal{P}_M^+$ and $\mathcal{P}_M^-$ must both be nonempty and projective closed. In particular, $\mathcal{P}_M^+ \cap \cT_i \neq \emptyset$ for all $i$, and likewise for $\mathcal{P}_M^-$. Now for each $i$ denote $\rho_i(M,\mathcal{P}_M^+,\emptyset) = M_i \oplus P_i$ and $\rho_i( M,\emptyset,\mathcal{P}_M^-) = M_i \oplus Q_i[1]$. Likewise, consider $M_i \oplus P_{M_i}^+$ and $M_i \oplus P_{M_i}^-[1]$ as in Proposition~\ref{prop:complete all positive}(3). Then in particular we have that $P_i$ is a direct summand of $P_{M_i}^+$ and that $Q_i$ is a direct summand of $P_{M_i}^-$ for all $i$. Then $\iota^+(M_1\oplus P_{M_1}^+,\ldots,M_m \oplus P_{M_m}^+) = : (M,\mathcal{Q},\emptyset)$ contains $(M,\mathcal{P}^+_M,\emptyset)$ by Proposition~\ref{prop:inclusions}(1). (Note that the module part of $(M,\mathcal{Q},\emptyset)$ is indeed $M$ by construction.) It follows that $\mathcal{Q} = \mathcal{P}^+_M$, and thus that $P_i = P_{M_i}^+$ by Proposition~\ref{prop:composition}(1). Thus $(M,\mathcal{P}_M^+,\emptyset)$ is a support regular cluster if and only if $M_i \oplus P_{M_i}^+$ is support $\tau$-rigid for all $i$ by Theorem~\ref{thm:maximal}(1). A similar argument then shows that $(M,\emptyset,\mathcal{P}_M^-)$ is a support regular cluster if and only if $M_i \oplus P_{M_i}^-[1]$ is support $\tau$-tilting for all $i$. The result then follows from Proposition~\ref{prop:complete all positive}(3).
\end{proof}

%%%%%%%%%%%%%%%%%%%%%%%%%%%%%%%%%%%%%%%%%%%%%%%%%%%%%%
\subsection{The polyhedral fan of support regular rigid objects}\label{sec:stables}

In this section, we associate to each support regular rigid object a cone in $\RR^{n-1}$ and show that the union of these cones is a polyhedral fan. This can be seen as a regular analogue of the $g$-vector fan (see Definition \ref{def:g fan} and Proposition \ref{prop:g fan}).

We begin by modifying our $g$-vectors.

\begin{defn}\label{def:null g}
For a regular module $M \in \Reg H$ we denote $g_0(M) := g(M) - \frac{g(M)\cdot g(\eta)}{|g(\eta)|^2}g(\eta)$, the projection of $g(M)$ onto $g(\eta)^\perp$.
\end{defn}

Since $g(\eta)\cdot\undim N = 0$ for all regular modules $N$, we immediately obtain the following.

\begin{prop}\label{prop:null g}
	Let $M, N$ be regular modules. Then $g_0(M)\cdot\undim N = g(M)\cdot\undim N$.
\end{prop}

\begin{defn}\label{def:cones}
Let $(M,\cP^+,\cP^-)$ be support regular rigid. We denote by $C(M,\cP^+,\cP^-)$ the polyhedral cone which is the non-negative span of:
	\begin{itemize}
		\item $g_0(M_j)$ for $M_j$ an indecomposable direct summand of $M$.
		\item $\{p(\cX): p(\cX) \in \cP^+\}$.
		\item $\{-p(\cX): p(\cX) \in \cP^-\}$.
	\end{itemize}
So, $C(M,\cP^+,\cP^-)\subseteq g(\eta)^\perp$. We denote by $\fC(M,\cP^+,\cP^-)$ the relative interior of $C(M,\cP^+,\cP^-)$.
\end{defn}

We observe that there are often linear dependencies amongst the vectors in $\cP := \cP^+\cup \cP^-$. In particular, we have the following.

\begin{lem}\label{lem:cone}
	Let $\cP$ be a set of projective vectors. Let 
		\begin{eqnarray*}
			v &=& \sum_{p(\cX) \in \cP}\lambda_{\cX}\cdot p(\cX)\\
			v' &=& \sum_{p(\cX) \in \cP}\lambda'_{\cX}\cdot p(\cX).
		\end{eqnarray*}
		be two vectors where the coefficients $\lambda_{\cX}, \lambda'_{\cX}$ are arbitrary. For all quasi simple $X_{j,1}^i \in \tp(\cP)$, we denote
		\begin{eqnarray}
			\lambda_{i,j} &:=& \sum_{p(\cX) \in \cP:X_{j,1}^i \in \cX}\lambda_{\cX}\label{eqn:tube sum}\\
			\lambda'_{i,j} &:=& \sum_{p(\cX) \in \cP:X_{j,1}^i \in \cX}\lambda'_{\cX}\nonumber
		\end{eqnarray}
Then $v = v'$ if and only if $\lambda_{i,j} = \lambda'_{i,j}$ for all $X_{j,1}^i \in \tp(\cP)$.
\end{lem}

\begin{proof}
	Suppose $v = v'$ and let $X_{j,1}^i \in \tp(\cP)$. Then by Proposition \ref{prop:multiplicity}, we have 
	$$\lambda_{i,j} = v\cdot\undim X_{j,1}^i = v'\cdot\undim X_{j,1}^i = \lambda'_{i,j}.$$
	Likewise, if $\lambda_{i,j} = \lambda'_{i,j}$ for all $i,j$, then we have $v\cdot\undim X_{j,1}^i = v'\cdot\undim X_{j,1}^i$ for all quasi simple $X_{j,1}^i \in \tp(\cP)$. As the dimension vectors of the quasi-simples span $\RR^{n-1}$, we conclude that $v = v'$.
\end{proof}

As a result, we can give an alternative description of $C(M,\cP^+,\cP^-)$.

\begin{prop}\label{prop:cone}
	Let $(M,\cP^+,\cP^-)$ be support regular rigid. Decompose $M \cong \oplus_{k = 1}^t M_k$ into a direct sum of indecomposable modules. Then $C(M,\cP^+,\cP^-)$ is equal to the set of vectors that can be written in the form
	$$v = \sum_{k = 1}^t \lambda_k\cdot g_0(M_k) + \sum_{p(\cX) \in \cP^+}\lambda_{\cX}\cdot p(\cX) - \sum_{p(\cX)\in \cP^-}\lambda_{\cX}\cdot p(\cX)$$
	such that
	\begin{enumerate}
		\item For all $k \in \{1,\ldots,t\}$, $\lambda_k \geq 0$.
		\item $\lambda_{i,j} \geq 0$ (where $\lambda_{i,j}$ is as defined in Equation \ref{eqn:tube sum}) for all $X_{j,1}^i \in \tp(\cP^+ \cup \cP^-)$.
	\end{enumerate}
\end{prop}

\begin{proof}
	Let $C'(M,\cP^+,\cP^-)$ be the set of vectors described in the proposition. We assume $\cP^- = \emptyset$, as the case where $\cP^+ = \emptyset$ is analogous. It is clear that $C(M,\cP^+,\cP^-) \subset C'(M,\cP^+,\cP^-)$. Thus let $v \in C'(M,\cP^+,\cP^-)$ and write
	$$v = \sum_{k = 1}^t \lambda_k\cdot g_0(M_k) + \sum_{p(\cX) \in \cP^+}\lambda_{\cX}\cdot p(\cX).$$
	By Lemma \ref{lem:cone} above, we need only show that there exist coefficients $\lambda'_{\cX}\geq 0$ for each $p(\cX) \in \cP^+$ so that $\lambda_{i,j} = \lambda'_{i,j}$ for all $X_{j,1}^i \in \tp(\cP^+)$.  In general, the set of possible $\{\lambda'_{\cX}\}_{p(\cX) \in \cP^+}$ satisfying this property is an \emph{$m$-index transportation polytope} (see e.g. \cite[Section 1]{KL_transportation}). Since for $i$ fixed we have $\sum_{j} \lambda_{i,j} = \sum_{p(\cX) \in \cP^+}\lambda_{\cX} \geq 0$, which does not depend on $i$, this solution space is nonempty.
	Indeed, if $\sum_j \lambda_{i,j} = 0$, one can take the $\lambda_{\cX}'$ to be identically zero. Otherwise, one explicit solution is
	\[
	\lambda'_\cX=S^{1-m}\prod_{X^i_{j,1}\in\cX} \lambda_{i,j},
	\]
	where $S=\sum_j \lambda_{i,j} $.
\end{proof}

We now turn our attention to determining the subcategory of semi-stable modules for the interiors of each cone $C(M,\cP^+,\cP^-)$. The following can be seen as a strengthening of Lemma \ref{lem:cone}.

\begin{prop}\label{prop:defined projection}
	Let $(M,\cP^+,\cP^-)$ be support regular rigid. Let $v \in C(M,\cP^+,\cP^-)$ and write 
	$$v = \sum_{k = 1}^t \lambda_k\cdot g_0(X_{j_k,\ell_k}^{i_k}) + \sum_{p(\cX) \in \cP^+} \lambda_{\cX}\cdot p(\cX) - \sum_{p(\cX) \in \cP^-} \lambda_{\cX}\cdot p(\cX).$$
	For each $i$, let $\cM_i$ be the set of indecomposable direct summands of $M$ in the tube $\cT_i$. Denote
	$$v(i) := \sum_{X_{j_k,\ell_k}^i \in \cM_i} \lambda_k\cdot g(Y_{j_k,\ell_k}^i) + \sum_{X_{j,1}^i \in \cT_i \cap \tp(\cP^+)}\lambda_{i,j}\cdot g(Y_{j,r_i + 1}^i) - \sum_{X_{j,1}^i \in \cT_i \cap \tp(\cP^-)}\lambda_{i,j}\cdot g(Y_{j,r_i + 1}^i).$$
	Then the following hold.
	\begin{enumerate}
		\item The association $v \mapsto v(i)$ is a well-defined linear map $C(M,\cP^+,\cP^-) \rightarrow C(\rho_i(M,\cP^+,\cP^-))$.
		\item The vector $v$ uniquely determines the coefficients $\lambda_k$ and $\lambda_{i,j}$.
		\item Let $X_{j,\ell}^i \in \Reg H$ be a brick, or equivalently be such that $\ell \leq r_i$. Then $v\cdot\undim X_{j,\ell}^i = v(i)\cdot\undim Y_{j,\ell}^i$. In particular, $X_{j,\ell}^i$ is $v$-regular semistable if and only if $Y_{j,\ell}^i$ is $v(i)$-semistable.
	\end{enumerate}
\end{prop}

\begin{proof}
	(1) Let
	\begin{eqnarray*}
		v &=& \sum_{k = 1}^t \lambda_k\cdot g_0(X_{j_k,\ell_k}^{i_k}) + \sum_{p(\cX) \in \cP^+} \lambda_{\cX}\cdot p(\cX) - \sum_{p(\cX) \in \cP^-} \lambda_{\cX}\cdot p(\cX)\\
		v' &=& \sum_{k = 1}^t \lambda'_k\cdot g_0(X_{j_k,\ell_k}^{i_k}) + \sum_{p(\cX) \in \cP^+} \lambda'_{\cX}\cdot p(\cX) - \sum_{p(\cX) \in \cP^-} \lambda'_{\cX}\cdot p(\cX)
	\end{eqnarray*}
	and suppose $v = v'$. Let $X_{j,1}^i$ be a quasi simple. Then by Propositions \ref{prop:same homs}, \ref{prop:multiplicity}, and \ref{prop:null g}, we have that
	$$v(i)\cdot \undim Y_{j,1}^i = v\cdot \undim X_{j,1}^i = v'\cdot \undim X_{j,1}^i = v'(i)\cdot \undim Y_{j,1}^i.$$
	Since $\{\undim Y_{j,1}^i\}_{j = 1}^{r_i}$ is a basis of $\RR^{r_i}$, this implies that $v(i) = v'(i)$.
	
	(2) Let $v, v'$ be as in (1), and suppose $v = v'$. We observe that the $g$-vectors appearing in the definition of $v(i)$ are precisely those of the direct summands of $\rho_i(M,\cP^+,\cP^-)$. In particular, this means they are linearly independent. This implies that $\lambda_k = \lambda'_k$ for all $k$. It then follows from Lemma~\ref{lem:cone} that also $\lambda_{i,j} = \lambda'_{i,j}$ for all $i$ and $j$.
	
	(3) analogously to (1), this follows directly from Propositions \ref{prop:same homs}, \ref{prop:multiplicity}, and \ref{prop:null g}.
\end{proof}

We have the following consequences when $(M,\cP^+,\cP^-)$ is projectively closed.

\begin{cor}\label{cor:int}
	Let $(M,\cP^+,\cP^-)$ be projectively closed and let
	$$v = \sum_{k = 1}^t \lambda_k\cdot g_0(X_{j_k,\ell_k}^{i_k}) + \sum_{p(\cX) \in \cP^+} \lambda_{\cX}\cdot p(\cX) - \sum_{p(\cX) \in \cP^-} \lambda_{\cX}\cdot p(\cX)$$
	be a vector in $C(M,\cP^+,\cP^-)$. Then $v \in \fC(M,\cP^+,\cP^-)$ if and only if the coefficients $\lambda_k, \lambda_{i,j}$ are all positive.
\end{cor}

\begin{proof}
	Let $\cS$ be the set of $g_0$-vectors and projective vectors defining $C(M,\cP^+,\cP^-)$. Suppose $v$ has all of its coefficients $\lambda_k$ and $\lambda_{i,j}$ nonzero. Then for all $w \in \pm\cS$ and for all $\varepsilon > 0$ sufficiently small, the vector $v + \varepsilon w \in C(M,\cP^+,\cP^-)$. This means $v \in \fC(M,\cP^+,\cP^-)$. If $v$ has a coefficient $\lambda_k = 0$, then for all $\varepsilon > 0$, the vector $v - \varepsilon g_0(X_{j_k,\ell_k}^{i_k}) \notin C(M,\cP^+,\cP^-)$ by Proposition \ref{prop:defined projection}. Likewise, if $v$ has a coefficient $\lambda_{i,j} = 0$, then for all $p(\cX) \in \cP^+$ (resp for all $p(\cX) \in \cP^-$) and for all $\varepsilon > 0$, the vector $v - \varepsilon p(\cX) \notin C(M,\cP^+,\cP^-)$ (resp. the vector $v + \varepsilon p(\cX) \notin C(M,\cP^+,\cP^-)$) by Proposition \ref{prop:defined projection}. Thus in either case, $v \notin \fC(M,\cP^+,\cP^-)$.
\end{proof}

\begin{cor}\label{cor:cone dim}
	Let $(M,\cP^+,\cP^-)$ be projectively closed.
	\begin{enumerate}
		\item If $\cP^+\cup \cP^- = \emptyset$, then the dimension of $C(M,\cP^+,\cP^-)$ is equal to $|M|$.
		\item If $\cP^+\cup \cP^- \neq \emptyset$, then the dimension of $C(M,\cP^+,\cP^-)$ is equal to $$1 - m + \sum_{i = 1}^m |\rho_i(M,\cP^+,\cP^-)|.$$
	\end{enumerate}
\end{cor}

\begin{proof}
	In both cases, we consider the linear map 
	$$F:C(M,\cP^+,\cP^-) \rightarrow C(\rho_1(M,\cP^+,\cP^-))\times\cdots\times C(\rho_m(M,\cP^+,\cP^-)).$$
	This map is injective by Proposition \ref{prop:defined projection}. Thus let
	$$(v_1,\ldots,v_m) \in C(\rho_1(M,\cP^+,\cP^-))\times\cdots\times C(\rho_m(M,\cP^+,\cP^-)).$$
	For each $i$, we can uniquely write
	$$v_i = \sum_{X_{j_k,\ell_k}^i \in \cM_i} \lambda_k\cdot g(Y_{j_k,\ell_k}^i) + \sum_{X_{j,1}^i \in \cT_i\cap \tp(\cP^+)}\lambda_{i,j}\cdot g(Y_{j,r_i + 1}^i) - \sum_{X_{j,+1}^i \in \cT_i \cap \tp(\cP^-)}\lambda_{i,j}\cdot g(Y_{j,r_i + 1}^i).$$
	By Proposition \ref{prop:cone} and its proof, there exists $v \in C(M,\cP^+,\cP^-)$ with $F(v) = (v_1,\ldots,v_m)$ if and only if the sums
	$$s_i := \sum_{j}\lambda_{i,j}$$
	are equal for all tubes $\cT_i$. In case (1), this is satisfied automatically, so
	$$\dim C(M,\cP^+,\cP^-) = \dim\Im F = \sum_{j = 1}^m |\rho_i(M,\cP^+,\cP^-)| = |M|.$$
	In case (2), this means there are $m-1$ linear relations defining $\Im F$, so
	$$\dim C(M,\cP^+,\cP^-) = \dim\Im F = 1-m + \sum_{i = 1}^m |\rho_i(M,\cP^+,\cP^-)|.$$
\end{proof}

We conclude this section by defining the polyhedral fan of support regular rigid objects.

\begin{prop}\label{prop:proj closed fan}
	Let $\mathfrak{X}$ be the set of cones $C(M,\cP^+,\cP^-)$ for $(M,\cP^+,\cP^-)$ support regular rigid and projectively closed. Then $\mathfrak{X}$ is a polyhedral fan.
\end{prop}

\begin{proof}
Let $(M,\cP^+,\cP^-)$ be projectively closed. We must show that every codimension 1 face of $C(M,\cP^+,\cP^-)$ can be written in the form $C(N,\cQ^+,\cQ^-)$ for some projectively closed $(N,\cQ^+,\cQ^-)$. The result then follows for lower dimensional faces as each is contained in a codimension 1 face. We will assume that $\cP^- = \emptyset$ as the result for $\cP^+ = \emptyset$ follows analogously.

We identify each support regular rigid object with the set of vectors defining its cone. In particular, $(M,\cP^+,\cP^-)$ is identified with
$$\cS = \{g_0(X_{j_1,\ell_1}^{i_1}),\ldots,g_0(X_{j_t,\ell_t}^{i_t}),p(\cX_1),\ldots,p(\cX_s)\}.$$
For $\cS' \subseteq \cS$, we denote by $\overline{\cS'}$ the projective closure of $\cS'$.

Claim 1: Suppose $\cS'$ is obtained from $\cS$ by deleting one $g_0$-vector. Then $\cS' = \overline{\cS'}$ corresponds to a codimension 1 face of $C(M,\cP^+,\cP^-)$. Indeed, suppose $g_0(X_{j,\ell}^i)$ was deleted. We see that $\dim C(\cS') = \dim C(\cS) - 1$ by Corollary \ref{cor:cone dim}. Moreover, $C(\cS')$ is contained in the relative boundary of $C(M,\cP^+,\cP^-)$ by Corollary \ref{cor:int}. Finally, there does not exist $\cS' \subsetneq \cS'' \subsetneq \cS$, meaning $C(\cS')$ is a face of $C(M,\cP^+,\cP^-)$.

Claim 2: Suppose $\cS'$ is obtained from $\cS$ be choosing $X_{j,1}^i \in \tp(\cP^+ \cup \cP^-)$ and deleting all projective vectors $p(\cX)$ so that $X_{j,1}^i \in \cX$. If either $|\cP^+\cup\cP^-| = 1$ or there exists $j' \neq j$ with $X_{j',1}^i \in \tp(\cP^+ \cup \cP^-)$, then $\cS' = \overline{\cS'}$ corresponds to a codimension 1 face of $C(M,\cP^+,\cP^-)$. Indeed, we see that $\dim C(\cS') = \dim C(\cS) - 1$ by Corollary \ref{cor:cone dim}. Finally, given $\cS' \subsetneq \cS'' \subsetneq \cS$, there exists $p(\cX) \in \cS''$ with $X_{j,1}^i \in \cX$. By Corollary \ref{cor:int}, the sum of all of the vectors in $\cS''$ thus lies in $\fC(M,\cP^+,\cP^-)$ and $\dim C(\cS'') = \dim C(M,\cP^+,\cP^-)$. This means $C(\cS')$ is a face of $C(M,\cP^+,\cP^-)$ as claimed.

Claim 3: There are no codimension 1 faces of $C(M,\cP^+,\cP^-)$ other than those in Claims (1) and (2). Indeed, let $\cS' \subseteq \cS$. If there exists $g_0(X_{j,\ell}^i)\in\cS \setminus \cS'$, then $C(\cS')$ is contained in $C(\cS\setminus \{g_0(X_{j,\ell}^i)\})$, which is a face by Claim 1. Likewise, if there exists $X_{j,1}^i \in \tp(\cP^+ \cup \cP^-)$ so that $\cS \setminus \cS'$ contains a projective vector $p(\cX)$ with $X_{j,1}^i$, then the fact that $(M,\cP^+,\cP^-)$ is projective closed implies that $\cS \setminus \cS'$ contains all projective vectors $p(\cX) \in \cS$ for which $X_{j,1}^i \in \cX$, and so $C(\cS')$ is contained in $C(\cS\setminus \{p(\cX): X_{j,1}^i \in \cX\})$. If either $|\cP^+\cup\cP^-| = 1$ or there exists $j' \neq j$ with $X_{j',1}^i \in \tp(\cP^+ \cup \cP^-)$, then this is a codimension 1 face by Claim 2. Otherwise, $X_{j,1}^i \in \cX'$ for all $\cX' \in \cP^+ \cup \cP^-$ and there exists some $X_{j'',1}^{i''} \in \tp(\cP^+ \cup \cP^-)$ for which this is not the case. This means $C(\cS')$ is contained in $C(\cS\setminus \{p(\cX): X_{j'',1}^{i''} \in \cX\})$, which is a codimension 1 face by Claim 2. The only case that remains is when $\cS \setminus \cS'$ contains only projective vectors and for all $X_{j,\ell}^i \in \tp(\cP^+\cup\cP^-)$, there exists $p(\cX') \in \cS'$ with $X_{j,\ell}^i \in \cX'$. Corollary \ref{cor:int} then implies that the sum of all of the vectors in $\cS'$ lies in $\fC(M,\cP^+,\cP^-)$, so $C(\cS')$ is not a codimension 1 face of $C(M,\cP^+,\cP^-)$.

Now given $(M_1,\cP_1^+,\cP_1^-)$ and $(M_2,\cP_2^+,\cP_2^-)$ two projectively closed support regular rigid objects, let $N$ be the direct sum of the common indecomposable direct summands of $M_1$ and $M_2$, let $\cQ^+ = \cP_1^+ \cap \cP_2^+$, and let $\cQ^- = \cP_1^- \cap \cP_2^-$. Then $C(M_1,\cP_1^+,\cP_1^-) \cap C(M_2,\cP_2^+,\cP_2^-) = C(N,\cQ^+,\cQ^-)$ and $(N,\cQ^+,\cQ^-)$ is projectively closed.
\end{proof}

\begin{rem}
	The reason we consider only support regular rigid objects which are projectively closed is two-fold: Using all support regular rigid objects does not result in a polyhedral fan and does not accurately describe the regular wall-and-chamber structure. For example, in Figure \ref{fig:regularPics2}, the cone $C(p(2,123), p(143,4))$ is not a face of the cone spanned by all four projective vectors, nor is it included in a wall.
\end{rem}

%%%%%%%%%%%%%%%%%%%%%%%%%%%%%%%%%%%%%%%%%%%%%%%%%%%%%
\subsection{Walls and chambers from support regular rigid objects}\label{sec:chambers}

The goal of this section is to relate the polyhedral fan of support regular rigid objects to the regular wall-and-chamber structures. In particular, we show that chambers in the regular wall-and-chamber structure correspond to support regular clusters.

The following lemma and its proof can be seen as a ``regular analogue'' of \cite[Lemma 3.12]{BST_wall}. 

\begin{lem}\label{lem:trace}
	Let $(M,\cP^+,\cP^-)$ be support regular rigid and let $v \in \fC(M,\cP^+,\cP^-)$. Let $N$ be a regular module. Let $\ell$ be the larger of the regular length of $N$ and the regular length of $M$, and let $L$ be the direct sum of the (isoclass representatives of the) indecomposable regular modules with quasi length $\ell$ and
	quasi top in $\tp(\cP^+)$ (in particular, $L=0$ if $\cP^+=\emptyset$). Define
	$$tN = \sum_{f\in \rad\Hom_H(M,N)}\Im(f) + \sum_{f \in \rad\Hom_H(L,N)}\Im(f).$$
	Then $v \cdot \undim tN \geq 0$, with equality if and only if $tN = 0$.
\end{lem}

\begin{proof}
	Note that $\fC(M,\cP^+,\cP^-)$ is contained in the relative interior of the cone corresponding to the projective closure of $(M,\cP^+,\cP^-)$ by the proof of Proposition \ref{prop:proj closed fan}. Moreover, the set $\tp(\cP^+)$ does not change upon taking projective closure. Thus, we will assume that $(M,\cP^+,\cP^-)$ is projectively closed.
	
	Now observe that $tN$ is regular, since the image of any map between regular modules is necessarily regular. Moreover, there exists a positive integer $k$ and an epimorphism $q: (M\oplus L)^k \twoheadrightarrow tN$.
	
    Claim 1: $\Hom(tN,\tau M)=0$. Suppose for a contradiction there exists a nonzero map $tN \rightarrow \tau M$. Composing with $q$ then gives a nonzero morphism $(M\oplus L)^k \rightarrow \tau M$. Since $\Hom_H(M,\tau M) = 0$, this means there is a nonzero morphism $L\rightarrow \tau M$. By Proposition \ref{prop:long homs}, this means there exists $\cX \in \cP^+$ such that $p(\cX)\cdot \undim \tau M \neq 0$, a contradiction.
	
    Claim 2: $p(\cX)\cdot\undim tN=0$ for all $p(\cX)\in \cP^-$.	Now suppose for a contradiction that there exists $p(\cX) \in \cP^-$ so that $p(\cX)\cdot \undim tN \neq 0$ (and hence $p(\cX)\cdot \undim tN > 0$ by Remark \ref{rem:always nonnegative}). In particular, this means $\cP^+ = \emptyset$ and thus $L=0$.
    Moreover, there exists an indecomposable direct summand $N'$ of $tN$ so that $p(\cX)\cdot \undim N' \neq 0$. Write $\cX = \{X_{j_1,1}^1,\ldots,X_{j_m,1}^m\}$. Now if $N'$ is contained in an exceptional tube $\cT_i$, Proposition \ref{prop:multiplicity} implies that $X_{j_i,1}^i$ appears in the regular composition series of $N'$. Thus there exists an indecomposable direct summand of $M^k$ which contains $X_{j_i,1}^i$ in its regular composition series, contradicting that $p(\cX)\in \cP^-$.
    If $N'$ is not contained in an exceptional tube $\cT_i$, then $M^k$ and hence $M$ contains a homogeneous direct summand, a contradiction.
	
	Now write $v = \sum \lambda_k\cdot g_0(M_k) + \sum \lambda_{\cX}\cdot p(\cX) - \sum \lambda_{\cY}\cdot p(\cY)$ in the form in Definition \ref{def:cones}. Then by the preceding paragraphs, we have
	\begin{eqnarray*}
		v \cdot \undim tN &=& \sum \lambda_k(\dim_K\Hom_H(M_k,tN)-\dim_K\Hom_H(tN,\tau M_k))\\
		&& + \sum \lambda_{\cX}(p(\cX)\cdot \undim tN) - \sum \lambda_{\cY}(p(\cY)\cdot \undim tN)\\
		&=& \sum\lambda_k \dim_K\Hom_H(M_k,tN) + \sum\lambda_{\cX}p(\cX)\cdot \undim tN\\
		&\geq& 0.
	\end{eqnarray*}
	
	Now suppose that $tN \neq 0$. It remains to show that $v\cdot \undim tN \neq 0$ in this case. It suffices to prove this in the case that $tN$ is indecomposable. As in Equation~\ref{eqn:tube sum}, for all quasi simple $X_{j,1}^i \in \tp(\cP^+)$, denote $\lambda_{i,j} = \sum_{p(\cX) \in \cP: X_{j,1}^i \in \cX} \lambda_{\cX}$.
	
	Suppose first that $tN$ is homogeneous. Then $\Hom_H(M,tN) = 0$. Moreover, $\undim tN = c \cdot \eta$ for some $c > 0$, and so $p(\cX) \cdot \undim tN = c$ for all $p(\cX) \in \cP^+$. It follows (for any choice of $i$) that $$v\cdot \undim tN = \sum_{p(\cX) \in \cP^+} c = \sum_{X_{j,1}^i \in \cT_i \cap \tp(\cP^+)} c\lambda_{i,j}.$$ Since each coefficient $\lambda_{i,j}$ is positive (Corollary~\ref{cor:int}), this implies the result.
	
	Since $tN$ is assumed to be indecomposable, it remains to consider the case where $tN$ lies in some exceptional tube, say $\cT_i$. Treating $i$ as fixed, for $p(\cX) \in \cP^+$, consider the unique $X_{j,1}^i \in \cX$. Then Proposition~\ref{prop:multiplicity} implies that $p(\cX) \cdot \undim tN =: m_j$ is the multiplicity of $X_{j,1}^i$ in a regular composition series of $tN$. Since $m_j$ depends only of $X_{j,1}^i$, and not on all of $\cX$, it follows that
	$$v\cdot \undim tN = \sum_k \lambda_k \dim_K \Hom_H(M_k,tN) + \sum_{j} m_j\lambda_{i,j}.$$
	Now if $\Hom_H(M,tN) \neq 0$, then the fact that the coefficients $\lambda_k$ and $\lambda_{i,j}$ are all positive (Corollary~\ref{cor:int}) implies the result. Thus suppose $\Hom_H(M,tN) = 0$. From the definition of $tN$, this means that $\Hom_H(L,tN) \neq 0$. Then there exists $p(\cX) \in \cP^+$ such that $p(\cX) \cdot \undim tN > 0$ by Proposition~\ref{prop:long homs}. But then $m_j\lambda_{i,j} \neq 0$, where $j$ is such that $X_{j,1}^i \in \cX$. Again since the coefficients $\lambda_k$ and $\lambda_{i,j}$ are all positive (Corollary~\ref{cor:int}), this implies the result.
\end{proof}

\begin{lem}\label{lem:when regular semistable}
	Let $(M,\cP^+,\cP^-)$ be support regular rigid and let $v \in \fC(M,\cP^+,\cP^-)$. Let $N$ be a regular module. Then $N$ is $v$-regular semistable if and only if all of the following hold:
	\begin{enumerate}
		\item $\Hom_H(M,N) = 0$.
		\item $\Hom_H(N,\tau M) = 0$.
		\item $p(\cX) \cdot \undim N = 0$ for all $p(\cX) \in \cP^+\cup \cP^-$.
	\end{enumerate}
\end{lem}

\begin{proof}
	Suppose first that $N$ is $v$-regular semistable and let $tN$ be as in Lemma \ref{lem:trace}. Then since $tN$ is a regular submodule of $N$, we have $v\cdot \undim tN \leq 0$. Lemma \ref{lem:trace} then implies that $tN = 0$. This means $\Hom_H(M,N) = 0$ and, by Proposition \ref{prop:long homs}, $p(\cX)\cdot\undim N = 0$ for all $p(\cX) \in \cP^+$. As $v\cdot\undim N = 0$, this means that $\Hom_H(N,\tau M) = 0$ and so, by Remark \ref{rem:always nonnegative}, $p(\cX')\cdot\undim N = 0$ for all $p(\cX') \in \cP^-$.
	
	Now suppose all three conditions hold and let $N'$ be a regular submodule of $N$. Condition (1) then implies that $\Hom_H(M,N') = 0$. Thus suppose $\cP^+$ is nonempty and let $N''$ be an indecomposable direct summand of $N'$. By condition (3), there does not exists $\cX \in \cP^+$ and $X_{j,\ell}^i \in \cX$ which appears in the regular composition series of $N''$. Writing $v$ as in Definition \ref{def:cones}, this implies $v\cdot\undim N' \leq 0$. Moreover, by Proposition~\ref{prop:null g}, for any indecomposable direct summand $M'$ of $M$, we have $$g_0(M') \cdot \undim N = g(M') \cdot \undim N = \dim_k \Hom_H(M',N) - \dim_k\Hom_H(N,\tau M') = 0.$$ Together with condition (3), this implies that $v \cdot \undim N = 0.$ We conclude that $N$ is $v$-regular semistable.
\end{proof}

\begin{prop}\label{prop:exists semistable}
	Let $(M,\cP^+,\cP^-)$ be support regular rigid. Let $v \in \fC(M,\cP^+,\cP^-)$.
	If there exists a nonzero regular module which is $v$-regular semistable, then $(M,\cP,\cP^-)$ is not a support regular cluster.
\end{prop}

\begin{proof}
	Assume without loss of generality that $(M,\cP^+,\cP^-)$ is projectively closed and let $N \neq 0$ be $v$-regular semistable. By Lemma \ref{lem:when regular semistable}, we can assume $N$ is indecomposable and a brick. If $\cP^+\cup \cP^- \neq \emptyset$, we can further assume that $N$ is not homogeneous and lies in some exceptional tube $\cT_i$. By Proposition \ref{prop:defined projection}, this means there exists a $v(i)$-semistable module in $\mods\Lambda_{r_i}$. By \cite[Theorem 3.14]{BST_wall}, this means $\rho_i(M,\cP^+,\cP^-)$ is not support $\tau$-tilting and thus $(M,\cP^+,\cP^-)$ is not a support regular cluster by Theorem \ref{thm:maximal}.
	
	If $\cP^+\cup \cP^- = \emptyset$, then $\rho_i(M,\cP^+,\cP^-)$ is not support $\tau$-tilting by Proposition \ref{prop:complete all positive}(2). Again, this means $(M,\cP^+,\cP^-)$ is not a support regular cluster by Theorem \ref{thm:maximal}.
\end{proof}

\begin{thm}\label{thmCa}
	Let $(M,\cP^+,\cP^-)$ be a support regular cluster. Then $\fC(M,\cP^+,\cP^-)$ is a chamber in the regular wall-and-chamber structure $\fD_{reg}(H)$.
\end{thm}

\begin{proof}
	Let $(M,\cP^+,\cP^-)$ be a support regular cluster. Then by Corollary \ref{cor:cone dim}, $\fC(M,\cP^+,\cP^-)$ is open in $\RR^{n-1}$ and thus is contained in a chamber since it is connected.
	
	Now let $v \in \partial C(M,\cP^+,\cP^-)$ and write $$v = \sum_{k = 1}^t \lambda_j\cdot g_0(X_{j_k,\ell_k}^{i_k}) + \sum_{p(\cX) \in \cP^+} \lambda_{\cX}\cdot p(\cX) -  \sum_{p(\cX') \in \cP^-} \lambda_{\cX'}\cdot p(\cX').$$
	By Corollary \ref{cor:int}, we know that either there is some $\lambda_k = 0$ or there exists a quasi simple $X_{j,1}^i \in \tp(\cP^+ \cup \cP^-)$ with $\lambda_{i,j} = 0$. By \cite[Theorem 3.14]{BST_wall}, in the first case there exists a nonzero $v(i_k)$-semistable module (in $\mods\Lambda_{r_{i_k}}$). Likewise, in the second there exists a nonzero $v(i)$-semistable module (in $\mods\Lambda_{r_i}$) by the same result. Proposition \ref{prop:defined projection} then implies there exists a nonzero $v$-regular semistable module (in $\mods H$), and thus $v$ is contained in a wall. This shows that $\fC(M,\cP^+,\cP^-)$ is a chamber.
\end{proof}

We now focus on proving that the association in Theorem \ref{thmCa} is a bijection.

Let $\cS, \cF \subseteq \mods\Lambda_{r}$. We recall that $(\cS,\cF)$ is called a \emph{torsion pair} if $\cS^\perp = \cF$ and $\cS = \ ^\perp\cF$. In this case, $\cS$ is called a \emph{torsion class}. It is well-known that a full subcategory $\cS \subseteq \mods\Lambda_r$ is a torsion class if and only if it is closed under quotients and extensions. Moreover, since the algebra $\mods\Lambda_r$ is $\tau$-tilting finite, there is a bijection between support $\tau$-tilting objects for $\Lambda_r$ and torsion classes in $\mods\Lambda_r$ given by $M\oplus P[1] \mapsto \mathsf{Fac} M$, where $\mathsf{Fac}M$ is the subcategory of factors of finite direct sums of $M$ (see \cite{AIR_tilting}).

We now associate to each $v \in \RR^{n-1}$ a collection of $m$ torsion classes, motivated by \cite[Section 3.4]{BST_wall} and \cite[Section 6]{bridgeland_scattering}.

\begin{prop}\label{prop:split torsion classes}
	Let $v \in \RR^{n-1}$ and let $\cS(v) = \{M \in \Reg H\mid \forall (M \twoheadrightarrow M')\in \Reg H: v\cdot\undim M' \geq 0\}$. For each exceptional tube $\cT_i$, let $\cS_i(v)$ be the additive closure of $\{Y_{j,\ell}^i\mid X_{j,\ell}^i \in \cS\}$. Then $\cS_i(v)$ is a torsion class. Moreover, either for all $i$ there exists $j_i$ with $Y_{j_i,r_i}^i \in \cS_i(v)$ or no $\cS_i(v)$ contains a brick of dimension vector $\overline{1}$.
\end{prop}

\begin{proof}
	The fact that $\cS_i(v)$ is a torsion class follows immediately from Proposition \ref{prop:same homs}. Moreover, if there exists $M \in \cS(v)$ with $\undim M = \eta$, then for all $i$, there exists $j_i$ so that $X_{j_i,r_i}^i \in \cS(v)$ (and hence $Y_{j_i,r_i}^i \in \cS_i(v)$) by an argument analogous to the proof of Proposition \ref{prop:wholeHyperplane}.
\end{proof}

\begin{prop}\label{prop:glue torsion classes}
	For each $i$, let $\cS_i \subseteq \mods\Lambda_{r_i}$ be a torsion class. If either (a) for all $i$, there exists $j_i$ so that $Y_{j_i,r_i}^i \in \cS_i$ or (b) no $\cS_i$ contains a brick of dimension vector $\overline{1}$, then there exists $v \in \RR^{n-1}$ with $\cS_i = \cS_i(v)$ for all $i$.
\end{prop}

\begin{proof}
	For each $i$, let $M_i\oplus P_i[1]$ be the support $\tau$-tilting object for which $\cS_i = \mathsf{Fac} M_i$. We observe that if there exists $j_i$ so that $Y_{j_i,r_i}^i \in \cS_i$, then $Y_{j_i,r_i+1}^i$ must be a direct summand of $M_i$. This means $P_i = 0$ and $M_i\oplus P_i[1]$ is null-nonnegative. Otherwise, no $M_i$ contains a projective direct summand and $M_i\oplus P_i[1]$ is null-nonpositive. Thus in either case, one of 
	\begin{eqnarray*}
		\iota^+(M_1\oplus P_1[1],\ldots,M_m\oplus P[1])\\
		\iota^-(M_1\oplus P_1[1],\ldots,M_m\oplus P[1])
	\end{eqnarray*}
	exists and is support regular rigid. We denote this support regular rigid object by $(M,\cP^+,\cP^-)$ and let $v \in C(M,\cP^+,\cP^-)$. Then $v(i) \in C(M_i\oplus P_i[1])$ and $\cS_i = \cS_i(v)$ by Proposition \ref{prop:defined projection} and \cite[Remark 3.28]{BST_wall}.
\end{proof}

We are now ready to prove that the association between support regular clusters and chambers is a bijection.

\begin{thm}[Theorem~\ref{thmintro:mainC}, Part 1]\label{thmC}
	Let $H$ be a tame hereditary algebra. Then the association $(M,\cP^+,\cP^-)\mapsto \fC(M,\cP^+,\cP^-)$ is a bijection between support regular clusters and chambers in the regular wall-and-chamber structure $\fD_{reg}(H)$.
\end{thm}

\begin{proof}
	Let $v \in \RR^{n-1}$ be in the complement of the union of the walls in $\fD_{reg}(H)$. Denote by $\cS = \cS(v)$ and $\cS_i = \cS_i(v)$ for each $i$. Define
	$$\fC(\cS) := \{w \in \RR^{n-1} \mid \cS(w) = \cS\} \setminus \bigcup_{M \in \Reg H} D_{reg}(M).$$
	We claim that $\fC(\cS)$ is contained in a chamber of $\fD_{reg}(H)$. Indeed, let $w_1,w_2 \in \fC(\cS)$ and let $M \in \Reg H$. If $M \in \cS$, then for all regular $M'$ with $M \rightarrow M'$ we must have $w_1\cdot \undim M' > 0$ and $w_2\cdot \undim M' > 0$. The inequalities are strict because otherwise there would exist a $w_1$- or $w_2$-regular semistable module. Thus for any $t \in [0,1]$, we have that $(tw_1 + (1-t)w_2)\cdot \undim M' > 0$. This means $M \in \cS(tw_1 + (1-t)w_2)$ and is not $(tw_1 + (1-t)w_2)$-regular semistable. If $M \notin \cS$, then there exist regular $M_1$ and $M_2$ with $M\twoheadrightarrow M_k$ and $w_k\cdot \undim M_k < 0$ for $k \in \{1,2\}$. This means $M_1,M_2 \notin \cS$. Iterating this argument, we conclude there is a common (regular) quotient $M \twoheadrightarrow N$ for which both $w_1\cdot\undim N < 0$ and $w_2 \cdot \undim N < 0$. Thus for $t \in [0,1]$, $M \notin \cS(tw_1 + (1-t)w_2)$ and $M$ is not $(tw_1+(1-t)w_2)$-regular semistable. We conclude that $\fC(\cS)$ is connected and is therefore contained in a chamber.
	
	Now for each $i$, let $M_i \oplus P_i[1]$ be the unique support $\tau$-tilting object for which $\cS_i = \mathsf{Fac}M_i$. By the proof of Proposition \ref{prop:glue torsion classes} the support regular cluster $(M,\cP^+,\cP^-) := \iota^+(M_1\oplus P_1[1],\ldots,M_m\oplus P_m[1])$ or $(M,\cP^+,\cP^-) := \iota^-(M_1\oplus P_1[1],\ldots,M_m\oplus P_m[1])$, whichever is defined, satisfies \linebreak $\fC(M,\cP^+,\cP^-) \subseteq \fC(\cS)$. This means $\fC(\cS) = \fC(M,\cP^+,\cP^-)$ and the association between support regular clusters and chambers is surjective. The injectivity follows immediately since if $w_1 \in \fC(M,\cP^+,\cP^-)$ and $w_2 \in \fC(N,\cQ^+,\cQ^-)$ with $(M,\cP^+,\cP^-) \neq (N,\cQ^+,\cQ^-)$, then $\cS(w_1) \neq \cS(w_2)$.
\end{proof}

\begin{eg}\label{eg:pics}\
\begin{enumerate}
	\item The quiver $1 \rightarrow 2 \rightarrow 3 \rightarrow 4 \leftarrow 1$ of type $\widetilde{A}_3$ has a single exceptional tube (of rank 3). Thus support regular clusters correspond precisely to support $\tau$-tilting objects for $\Lambda_{3}$. This also means the regular semi-invariant picture of $KQ$ is isomorphic to the semi-invariant picture of $\Lambda_3$. This regular semi-invariant picture is shown in Figure \ref{fig:regularPics}.
	\item The quiver $1 \rightarrow 2 \rightarrow 3 \leftarrow 4 \leftarrow 1$ of type $\widetilde{A}_3$ has two exceptional tubes (each of rank 2). The regular semi-invariant picture for $KQ$ is shown in Figure \ref{fig:regularPics2}.
	\end{enumerate}
\end{eg}

\begin{center}
{\small
\begin{figure}
\begin{tikzpicture}[scale=1.3] % Huge BBBBBBBBBBBBBB
%	\draw[help lines=1] (-5,-5) grid (5,4);
%	\draw[fill] (0,0) circle [radius=2pt];

\begin{scope}\clip (-6, -4.75) rectangle (6,4);
	\draw[very thick, color=black,dashed] (0,0) circle [radius=2.2cm]; % null circle
		\begin{scope} %begin upper right circle and semicircle
	%	\draw[fill] (1.73,1) circle [radius=2pt];
		\draw[very thick,color=black] (1.73,1) circle [radius=2.73cm];
		\clip (0,-2) rectangle (-4,4);
		\draw[very thick,color=black,dotted] (0,1) ellipse [x radius=2.8cm,y radius=2.1cm];
		\end{scope} % end upper right circle and semicircle
		\begin{scope}[rotate=120] %begin upper left circle and semicircle
	%	\draw[fill] (1.73,1) circle [radius=2pt];
		\draw[very thick,color=black] (1.73,1) circle [radius=2.73cm];
		\clip (0,-2) rectangle (-4,4);
		\draw[very thick,color=black,dotted] (0,1) ellipse [x radius=2.8cm,y radius=2.1cm];
		\end{scope} % end upper left circle and semicircle
		\begin{scope}[rotate=-120] %begin lower circle and semicircle
	%	\draw[fill] (1.73,1) circle [radius=2pt];
		\draw [very thick,color=black] (1.73,1) circle [radius=2.73cm];
		\clip (0,-2) rectangle (-4,4);
		\draw[very thick,color=black,dotted] (0,1) ellipse [x radius=2.8cm,y radius=2.1cm];
		\end{scope} % end lower circle and semicircle
%
%\draw[font=\huge,color=brown] (0,0) node{B}

\end{scope}

\node at (3.5,3) [draw,fill=white,white]{\color{black}$D_{0}(3)$};
\node at (-3.5,3) [draw,fill=white,white]{\color{black}$D_{0}(2)$};
\node at (-2,2.5) [draw,fill=white,white]{\color{black}$D_{0}(23)$};
\node at (2.25,2) [draw,fill=white,white]{\color{black}$D_{0}(341)$};
\node at (-1,-3) [draw,fill=white,white]{\color{black}$D_{0}(412)$};
\node at (0,-4.7) [draw,fill=white,white]{\color{black}$D_{0}(41)$};

\node at (0,-2.125) [draw,fill=white,white]{\color{black}$D_{0}(4123)$};
\node at (2.125,0.75) [draw,fill=white,white]{\color{black}$D_{0}(3412)$};
\node at (-1.875,1) [draw,fill=white,white]{\color{black}$D_{0}(2341)$};

\node at (0.8,2.05) [draw,fill,circle,scale=0.6,label=75:$g_0(3)$]{};
\node at (1.35,-1.7) [draw,fill,circle,scale=0.6,label=-45:$g_0(41)$]{};
\node at (-2.15,-0.35) [draw,fill,circle,scale=0.6,label=180:$g_0(2)$]{};

\node at (-0.8,2.05) [draw,fill,circle,scale=0.6,label=45:$g_0(23)$]{};
\node at (-1.35,-1.7) [draw,fill,circle,scale=0.6,label=0:$g_0(412)$]{};
\node at (2.15,-0.35) [draw,fill,circle,scale=0.6,label=45:$g_0(341)$]{};

\node at (0,-1.1) [draw,fill,circle,scale=0.6,label=0:$p(41)$]{};
\node at (0,3.1) [draw,fill,circle,scale=0.6,label=0:$-p(41)$]{};
\node at (-0.95,0.55) [draw,fill,circle,scale=0.6,label=-45:$p(2)$]{};
\node at (2.7,-1.55) [draw,fill,circle,scale=0.6,label=0:$-p(2)$]{};
\node at (0.95,0.55) [draw,fill,circle,scale=0.6,label=-135:$p(3)$]{};
\node at (-2.7,-1.55) [draw,fill,circle,scale=0.6,label=180:$-p(3)$]{};

\begin{scope}[xshift=-45mm,yshift=-43mm]
\coordinate (A) at (0,1.2);
\coordinate (B) at (0.6,.9);
\coordinate (Bm) at (0.5,.95);
\coordinate (C) at (.6,.3);
\coordinate (Cm) at (.6,.45);
\coordinate (D) at (0,0);
\coordinate (Dm) at (0,0.15);
\coordinate (Dr) at (0.1,0.05);
\draw[thick,->] (A)--(Bm);
\draw[thick,->] (B)--(Cm);
\draw[thick,->] (C)--(Dr);
\draw[thick,->] (A)--(Dm);
\foreach \x/\y in {A/1,B/2,C/3,D/4}
\draw[fill,white] (\x) circle[radius=4pt];
\foreach \x/\y in {A/1,B/2,C/3,D/4}
\draw (\x) node{$\y$};
\end{scope}

\end{tikzpicture}
\caption{The regular semi-invariant picture for the quiver $1\rightarrow 2 \rightarrow 3 \rightarrow 4 \leftarrow 1$. The union of the dashed walls $D_0(2341), D_0(3412)$, and $D_0(4123)$ is the null wall $D_0(\eta)$, where we use $D_0(-)$ as shorthand for $D_{0,g(\eta)^{\perp}}(-)$. Support regular clusters correspond to sets of three points labeling the corners of a region/chamber.}\label{fig:regularPics}
\end{figure}
}

{\small
\begin{figure}
\begin{tikzpicture}[scale=1.5] % Huge BBBBBBBBBBBBBB
%	\draw[help lines=1] (-5,-5) grid (5,4);
%	\draw[fill] (0,0) circle [radius=2pt];
	\draw[very thick, color=black,dashed] (0,0) circle [radius=1.4cm]; % null circle
%
%\draw (-2.1,-1) node[below]{$\eta_{14}$};
%\draw (-1.3,-1.2) node{$(2)$};
%\draw (1.5,.8) node{$(3)$};
		\begin{scope} %begin upper right circle and semicircle
	%	\draw[fill] (1.73,1) circle [radius=2pt];
		\draw[thick,color=black] (1,1) circle [radius=2cm];
		\end{scope} % end upper right circle and semicircle
		\begin{scope}[rotate=90] %begin upper left circle and semicircle
	%	\draw[fill] (1.73,1) circle [radius=2pt];
		\draw[thick,color=black] (1,1) circle [radius=2cm];
		\end{scope} % end upper left circle and semicircle
		\begin{scope}[rotate=180] %begin lower circle and semicircle
	%	\draw[fill] (1.73,1) circle [radius=2pt];
		\draw[thick,color=black] (1,1) circle [radius=2cm];
		\end{scope} % end lower circle and semicircle
		
		\begin{scope}[rotate=270] %begin lower circle and semicircle
	%	\draw[fill] (1.73,1) circle [radius=2pt];
		\draw[thick,color=black] (1,1) circle [radius=2cm];
		\end{scope} % end lower circle and semicircle
		
\node at (2.33,2.5) [draw,fill=white,white]{\color{black}$D_{0}(2)$};
\node at (2.33,-2.5) [draw,fill=white,white]{\color{black}$D_{0}(4)$};
\node at (-2.33,-2.5) [draw,fill=white,white]{\color{black}$D_{0}(143)$};
\node at (-2.33,2.5) [draw,fill=white,white]{\color{black}$D_{0}(123)$};

\node at (0,1.35) [draw,fill=white,white]{\color{black}$D_{0}(\eta)$};
\node at (0,-1.35) [draw,fill=white,white]{\color{black}$D_{0}(\eta)$};
\node at (1.5,0) [draw,fill=white,white]{\color{black}$D_{0}(\eta)$};
\node at (-1.5,0) [draw,fill=white,white]{\color{black}$D_{0}(\eta)$};

\node at (0,0.75) [draw,fill,circle,scale=0.6,label={90:$p(2,123)$}]{};
\node at (0,-2.75) [draw,fill,circle,scale=0.6,label={-90:$-p(2,123)$}]{};
\node at (0,-0.75) [draw,fill,circle,scale=0.6,label={-90:$p(143,4)$}]{};
\node at (0,2.75) [draw,fill,circle,scale=0.6,label={90:$-p(143,4)$}]{};
\node at (-0.75,0) [draw,fill,circle,scale=0.6]{};
\node at (0,0.2) {$p(143,123)$};
\node at (2.75,0) [draw,fill,circle,scale=0.6,label={0:$-p(143,123)$}]{};
\node at (0.75,0) [draw,fill,circle,scale=0.6]{};
\node at (0.3,-0.1) {$p(2,4)$};
\node at (-2.75,0) [draw,fill,circle,scale=0.6,label={180:$-p(2,4)$}]{};

\node at (1,1) [draw,fill,circle,scale=0.6,label={45:$g_0(2)$}]{};
\node at (1,-1) [draw,fill,circle,scale=0.6,label={-45:$g_0(4)$}]{};
\node at (-1,-1) [draw,fill,circle,scale=0.6,label={-135:$g_0(143)$}]{};
\node at (-1,1) [draw,fill,circle,scale=0.6,label=135:$g_0(123)$]{};

\begin{scope}[xshift=-45mm,yshift=-28mm]
\coordinate (A) at (.5,1);
\coordinate (B) at (0,.5);
\coordinate (Bm) at (0.1,.6);
\coordinate (C) at (1,.5);
\coordinate (Cm) at (.9,.6);
\coordinate (D) at (0.5,0);
\coordinate (Dm) at (0.4,0.1);
\coordinate (Dr) at (0.6,0.1);
\draw[thick,->] (A)--(Bm);
\draw[thick,->] (B)--(Dm);
\draw[thick,->] (C)--(Dr);
\draw[thick,->] (A)--(Cm);
\foreach \x/\y in {A/1,B/4,C/2,D/3}
\draw[fill,white] (\x) circle[radius=4pt];
\foreach \x/\y in {A/1,B/4,C/2,D/3}
\draw (\x) node{$\y$};
\end{scope}

%
%\draw[font=\huge,color=brown] (0,0) node{B};
\end{tikzpicture}
\caption{The regular semi-invariant picture for the quiver $1 \rightarrow 2 \rightarrow 3 \leftarrow 4 \leftarrow 1$. The top left half of the null wall (dashed) is $D_0(1234)$ and the bottom right half is $D_0(4123)$. Likewise, the top right half of the null wall is $D_0(2143)$ and the bottom left half is $D_0(1432)$. Support regular clusters correspond to sets of 3 or 4 points labeling the corners of a region/chamber.}
\label{fig:regularPics2}
\end{figure}
}
\end{center}

As a corollary of Theorem \ref{thmC}, we observe the following.

\begin{cor}\label{cor:chambers}
	Let $H$ be a tame hereditary algebra with exceptional tubes $\cT_1,\ldots,\cT_m$. Then the chambers of $\fD_{reg}(H)$ are in one-to-one correspondence with sets of chambers in the $\fD(\Lambda_{r_i})$ which all lie on the same side of the walls $D(\overline{1})$.
\end{cor}

\begin{proof}
	This follows from Theorem \ref{thmC} and the definitions of the maps $\rho_i, \iota^+$, and $\iota^-$.
\end{proof}

We now turn to describing the labels of the walls in terms of support regular rigid objects.

\begin{thm}\label{thm:walls}
	Let $(M,\cP^+,\cP^-)$ be support regular rigid and projectively closed.
	\begin{enumerate}
		\item If $\cP^+\cup \cP^- \neq \emptyset$, then the following are equivalent.
		\begin{enumerate}
			\item There exists a unique $i \in \{1,\ldots,m\}$ so that $\rho_i(M,\cP^+,\cP^-)$ contains $(r_i-1)$ indecomposable direct summands. All other $\rho_{i'}(M,\cP^+,\cP^-)$ are support $\tau$-tilting.
			\item The cone $C(M,\cP^+,\cP^-)$ has dimension $(n-2)$ and is included in a wall different from $D_{reg}(\eta)$ in the regular wall-and-chamber structure.
		\end{enumerate}
		Moreover, given (a) and (b), the brick labeling the wall containing $C(M,\cP^+,\cP^-)$ can be constructed as follows: Let $M_1\oplus P_1[1]$ and $M_2\oplus P_2[1]$ be the two support $\tau$-tilting objects containing $\rho_i(M,\cP^+,\cP^-)$ and suppose $\mathsf{Fac} M_1 \subset \mathsf{Fac} M_2$. Then the cokernel of the right $\add M_1$-approximation of $M_2$ is a brick $Y_{j,\ell}^i$ and $C(M,\cP^+,\cP^-) \subseteq D_{reg}(X_{j,\ell}^i)$.
		\item If $\cP^+\cup \cP^- = \emptyset$, then the following are equivalent.
		\begin{enumerate}
			\item For all $i \in \{1,\ldots,m\}$, $\rho_i(M,\cP^+,\cP^-)$ contains $(r_i-1)$ indecomposable direct summands.
			\item The cone $C(M,\cP^+,\cP^-)$ has dimension $(n-2)$ and is included in the wall $D_{reg}(\eta)$ in the regular wall-and-chamber structure.
		\end{enumerate}
		Moreover, given (a) and (b), the portion of the null wall containing $C(M,\cP^+,\cP^-)$ can be described as follows: For each tube $\cT_i$, let $M_i = \rho_i(M,\cP^+,\cP^-)$ and observe that $M_i$ has no projective direct summands. Thus there exists an indecomposable projective $Y_{j_i,r_i+1}^i$ so that $M_i\oplus Y_{j_i,r_i+1}^i[1]$ is support $\tau$-tilting by Proposition \ref{prop:complete all positive}(1) and $$C(M,\cP^+,\cP^-) = \bigcap_{i = 1}^m D_{reg}(X_{j_i,r_i}^i) \subset D_{reg}(\eta).$$
	\end{enumerate}
\end{thm}

\begin{proof}
	(1) ($a\Rightarrow b$): The dimension of $C(M,\cP^+,\cP^-)$ is $n-2$ by Corollary \ref{cor:cone dim}. Now let $M_1\oplus P_1[1]$ and $M_2\oplus P_2[1]$ be the two support $\tau$-tilting objects containing $\rho_i(M,\cP^+,\cP^-)$ so that $\mathsf{Fac}M_1\subset \mathsf{Fac}(M_2)$. By \cite[Proposition 3.17]{BST_wall}, the cokernel of the right $\add M_1$-approximation of $M_2$ is some brick $Y_{j,\ell}^i$ which is semistable for all $v_i \in C(\rho_i(M,\cP^+,\cP^-))$. Proposition \ref{prop:defined projection} thus implies that $X_{j,\ell}^i$ is $v$-regular semistable for all $v \in C(M,\cP^+,\cP^-)$. Thus $C(M,\cP^+,\cP^-) \subseteq D_{reg}(X_{j,\ell}^i)$. Moreover, since $\cP^+\cup \cP^- \neq \emptyset$, we must have that $\undim X_{j,\ell}^i \neq \eta$. This also shows that (a) implies the moreover part.
	
	($b\Rightarrow a$): Suppose $\dim C(M,\cP^+,\cP^-) = n-2$. Since $\cP^+\cup \cP^- \neq \emptyset$, Corollary \ref{cor:cone dim} implies that $$\sum_{i = 1}^m |\rho_i(M,\cP^+,\cP^-)| = n - 3 + m = \left(\sum_{i = 1}^m r_i\right) - 1.$$
	Thus there exists a unique $i$ for which $|\rho_i(M,\cP^+,\cP^-)| = r_i - 1$. For all other $i'$, we have $|\rho_{i'}(M,\cP^+,\cP^-)| = r_{i'}$ and therefore $\rho_{i'}(M,\cP^+,\cP^-)$ is support $\tau$-tilting.
	
	(2) ($a\Rightarrow b$): The dimension of $C(M,\cP^+,\cP^-)$ is $n-2$ by Corollary \ref{cor:cone dim}. Now for each tube $\cT_i$, let $Y_{j_i,r_i+1}^i$ be the unique projective so that $\rho_i(M,\cP^+,\cP^-)\oplus Y_{j_i,r_i+1}^i[1]$ is support $\tau$-tilting. Thus $j_i$ is the unique vertex (of $\Lambda_{r_i}$) on which $\rho_i(M,\cP^+,\cP^-)$ is not supported and $j_i-1$ is the unique vertex on which $\tau\rho_i(M,\cP^+,\cP^-)$ is not supported. Thus $\rho_i(M,\cP^+,\cP^-)\oplus Y_{j_i-1,r_i+1}^i$ is also support $\tau$-tilting. The cokernel of the right $\add \rho_i(M,\cP^+,\cP^-)$-approximation of $\rho_i(M,\cP^+,\cP^-)\oplus Y_{j_i-1,r_i+1}^i$ is then isomorphic to $Y_{j_i,r_i}^i$. By \cite[Proposition 3.17]{BST_wall}, this brick is $v$-semistable for all $v \in C(\rho_i(M,\cP^+,\cP^-))$. Direct computation shows that in fact $C(\rho_i(M,\cP^+,\cP^-)) = D(Y_{j_i,r_i}^i)$. Proposition \ref{prop:defined projection} thus implies that 
	$$C(M,\cP^+,\cP^-) = \bigcap_{i = 1}^m D_{reg}(X_{j_i,r_i}^i) \subset D(\eta).$$
	This also shows that (a) implies the moreover part.
	
	($b\Rightarrow a$): Suppose $\dim C(M,\cP^+,\cP^-) = n-2$. Since $\cP^+ \cup \cP^- = \emptyset$, Corollary \ref{cor:cone dim} implies that $|M| = n-2$. This means each $|\rho_i(M,\cP^+,\cP^-)|=r_i-1$ for each $i$.
\end{proof}

We now combine Theorems \ref{thmC} and \ref{thm:walls} in the spirit of \cite[Corollary 3.18]{BST_wall}.

\begin{cor}[Theorem~\ref{thmintro:mainC}, Part 2]\label{cor:walls}
	Let $(M,\cP^+,\cP^-)$ be a support regular cluster.
	\begin{enumerate}
		\item If $|\cP^+ \cup \cP^-| = 1$, then the chamber $\fC(M,\cP^+,\cP^-)$ has $n-1$ walls. $n-2$ of these walls are labeled by the bricks obtained by deleting a single direct summand from $M$ and applying Theorem \ref{thm:walls}(1). The other wall is contained in the null wall and is labeled by the bricks obtained from deleting the unique projective vector from $\cP^+\cup \cP^-$ and applying Theorem \ref{thm:walls}(2).
		\item If $|\cP^+ \cup \cP^-| \neq 1$, let $\mathcal{I}$ be the set of tubes $\cT_i$ for which $|\cT_i \cap (\cP^+ \cup \cP^-)| = 1$. Then the chamber $C(M,\cP^+,\cP^-)$ has $\left(\sum_{i = 1}^m |\rho_i(M,\cP^+,\cP^-)|\right)-|\mathcal{I}|$ walls. The walls are labeled by the bricks obtained as follows:
		\begin{enumerate}
			\item Delete any single direct summand from $M$ and apply Theorem \ref{thm:walls}(1).
			\item For any $X_{j,1}^i \in \tp(\cP^+ \cup \cP^-)$, delete all $p(\cX') \in \cP^+\cup\cP^-$ with $X_{j,1}^i \in \cX'$ and apply Theorem \ref{thm:walls}(1).
		\end{enumerate}
	\end{enumerate}
\end{cor}

\begin{proof}
	Corollary \ref{cor:cone dim} and Proposition \ref{prop:proj closed fan} imply that the walls described in the theorem together form the boundary of $C(M,\cP^+,\cP^-)$. Thus we need only show the labels are pairwise distinct. By construction, if $\cT_i \notin \mathcal{I}$, the bricks in $\cT_i$ correspond to those in $\mods\Lambda_{r_i}$ labeling the $r_i$ walls of the chamber $\fC(\rho_i(M,\cP^+,\cP^-))$. These are pairwise distinct by \cite[Corollary 3.18]{BST_wall}. Otherwise, the bricks in $\cT_i$ correspond to the bricks in $\mods\Lambda_{r_i}$ labeling $r_i-1$ of the walls of the chamber $\fC(\rho_i(M,\cP^+,\cP^-))$. Again, these are uniquely determined by \cite[Corollary 3.18]{BST_wall}.
\end{proof}

We conclude this section by briefly comparing our regular wall-and-chamber structure to other known constructions.

In \cite{RS_affine}, Reading and Stella define a fan using a compatibility condition on the almost positive Schur roots of a tame valued quiver. Interpreted in the context of representation theory, the almost positive Schur roots correspond to the indecomposable rigid represesntations, the indecomposable negative projective representations, and the representation $M_\lambda$ with $\undim M_\lambda = \eta$ and $\lambda \in K^*$ arbitrary. The equivalence relation corresponds to generic ext-orthogonality (here generic means that $M_\lambda$ is considered to correspond to a rigid representation). In this model, maximal compatible sets of almost positive Schur roots containing $\eta$ correspond to \emph{imaginary clusters}. An imaginary cluster is a maximal rigid object in the category of regular representations together with $\eta$. Such a collection contains precisely $n-2$ objects. These also appear in the work of Scherotzke \cite{scherotzke_component} under the name \emph{component clusters}. Such objects correspond precisely to support regular clusters $(M,\cP^+,\cP^-)$ with $|\cP^+\cup \cP^-| = 1$. Namely, we have the following.

\begin{cor}\label{cor:imaginary clusters}
	Let $(M,\cP^+,\cP^-)$ be a support regular cluster. Then the following are equivalent.
	\begin{enumerate}
		\item $M \oplus \eta$ is an imaginary cluster (component cluster) in the sense of Reading-Stella (Sherotzke).
		\item $|\cP^+\cup \cP^-| = 1$.
		\item $\fC(M,\cP^+,\cP^-)$ is bounded by the null wall; that is, the intersection of $C(M,\cP^+,\cP^-)$ with $D_{reg}(\eta)$ is $(n-2)$-dimensional.
	\end{enumerate}
\end{cor}

%%%%%%%%%%%%%%%%%%%%%%%%%%%%%%%%%%%%%%%%%%%%%%%%%%%%%%

\section{Generalization to the cluster-tilted case}\label{sec:cluster-tilted}
In this section, we generalize our results to cluster-tilted algebras of tame type. We will assume that $K$ is algebraically closed.

%%%%%%%%%%%%%%%%%%%%%%%%%%%%%%%%%%%%%%%%%%%%%%%%%%%%%%
\subsection{Formulas for Mutation}\label{sec:mutation formula}
We begin by discussing three mutation rules for scattering diagrams of affine (tame) type. The first is a formula of Reading \cite{reading_universal} which describes the mutation of $c$-vectors and $g$-vectors. The second is a formula of Mou \cite{mou_scattering} which describes functors related to the mutation of the algebraic scattering diagram of Bridgeland \cite{bridgeland_scattering}. The third is the mutation of (decorated) representations of quivers with potential due to Derksen-Weyman-Zelevinsky \cite{DWZ_quivers, DWZ_quivers2}.

Let $B$ be an $n\times n$ skew symmetric matrix and choose an index $1 \leq k \leq n$. There are two matrices associated to $B$, $A^+_k$ and $A^-_k$, given by
$$(A^+_k)_{ij} = \begin{cases} 1 & i = j \neq k\\-1 & i = j = k \\ \max\{B_{ij},0\} & j \neq i = k \\ 0 & i \neq j,k\end{cases} \qquad (A^-_k)_{ij} = \begin{cases} 1 & i = j \neq k\\-1 & i = j = k \\ \max\{-B_{ij},0\} & j \neq i = k \\ 0 & i \neq j,k\end{cases}.$$

We note that $A_k^+A_k^+ = \mathsf{Id}_n = A_k^-A_k^-$. The following is essentially \cite[Theorem 2.18]{BHIT_semi-invariant}.

\begin{thm}\label{readingFormula}
	Let $J(Q,W)$ be a cluster-tilted algebra. Let $D(M)$ be a wall in the standard wall-and-chamber structure $\mathfrak{D}(J(Q,W))$. Suppose $\undim M$ is a $c$-vector (see Remark \ref{rem:c-vector}).
	\begin{enumerate}
		\item Suppose there exists $w \in D(M)$ such that $w\cdot\undim S(k) > 0$ and let $v \in D(M)$ such that $v \cdot\undim S(k) \geq 0$. Then $(v^{tr} A_k^+)^{tr}$ is included in a wall $D(M')$ in the standard wall-and-chamber structure of $J(\mu_k(Q,W))$. Moreover, $\undim M' = A_k^+ \undim M$ and this is a $c$-vector.
		\item Suppose there exists $w \in D(M)$ such that $w\cdot\undim S(k) < 0$ and let $v \in D(M)$ such that $v \cdot\undim S(k) \leq 0$. Then $(v^{tr} A_k^-)^{tr}$ is included in a wall $D(M')$ in the standard wall-and-chamber structure of $J(\mu_k(Q,W))$. Moreover, $\undim M' = A_k^- \undim M$ and this is a $c$-vector.
		\item Suppose $M = S(k)$ and let $S'(k)$ be the simple module in $\mods J(\mu_k(Q,W))$ at vertex $k$. Then for all $v \in D(S(k))$, we have that $(v^{tr} A_k^+)^{tr} = (v^{tr} A_k^-)^{tr}$ is included in the wall $D(S'(k))$ in the standard wall-and-chamber structure of $J(\mu_k(Q,W))$.
	\end{enumerate}
\end{thm}

In order to prove Theorem~\ref{thmintro:mainD}, we must consider \emph{all} of the walls in the standard wall-and-chamber structure, not just the walls labeled by $c$-vectors. Thus the remainder of this section is aimed at generalizing Theorem \ref{readingFormula} to the following.

\begin{thm}\label{genReadingFormula}
	Let $J(Q,W)$ be a tame cluster-tilted algebra. Let $D(M)$ be a wall in the standard wall-and-chamber structure $\mathfrak{D}(J(Q,W))$.
	\begin{enumerate}
		\item Suppose there exists $w \in D(M)$ such that $w\cdot\undim S(k) > 0$ and let $v \in D(M)$ such that $v \cdot\undim S(k) \geq 0$. Then $(v^{tr} A_k^+)^{tr}$ is included in a wall $D(M')$ in the standard wall-and-chamber structure of $J(\mu_k(Q,W))$. Moreover, $\undim M' = A_k^+ \undim M$.
		\item Suppose there exists $w \in D(M)$ such that $w\cdot\undim S(k) < 0$ and let $v \in D(M)$ such that $v \cdot\undim S(k) \leq 0$. Then $(v^{tr} A_k^-)^{tr}$ is included in a wall $D(M')$ in the standard wall-and-chamber structure of $J(\mu_k(Q,W))$. Moreover, $\undim M' = A_k^- \undim M$.
		\item Suppose $M = S(k)$ and let $S'(k)$ be the simple module in $\mods J(\mu_k(Q,W))$ at vertex $k$. Then for all $v \in D(S(k))$, we have that $(v^{tr} A_k^+)^{tr} = (v^{tr} A_k^-)^{tr}$ is included in the wall $D(S'(k))$ in the standard wall-and-chamber structure of $J(\mu_k(Q,W))$.
	\end{enumerate}
\end{thm}

This essentially follows from \cite[Proposition 4.15]{mou_scattering}, but we include a proof here for completeness. The reason for the difference in statement between the present paper and that of Mou is that we are taking a different basis of $\RR^n$ after mutation.

We begin by recalling the following construction from \cite{DWZ_quivers}. Note that for a representation $M$ of a quiver $Q$, we associate to each arrow $\rho \in Q_1$, a linear map $M_\rho: M_{s(\rho)} \rightarrow M_{t(\rho)}$. In \cite{DWZ_quivers}, on the other hand, our $s(\rho)$ is replaced with their ``tail'' $t(\rho)$ and our $t(\rho)$ is replaced with their ``head'' $h(\rho)$.

\begin{defn}\cite[Equations~10.2 and~10.3]{DWZ_quivers}\label{def:mutated rep}
Let $M$ be a representation of the quiver with potential $(Q,W)$\footnote{The definitions in \cite{DWZ_quivers} are given for \emph{decorated} representations of $(Q,W)$, but we simplify the construction here since the representations we are considering would have trivial decorations. Thus, for the purposes of this paper, $M$ can be considered as a representation of $Q/I$ where $I$ is the ideal defined by $W$.}. We consider $M$ as a $J(Q,W)$-module. Let $k \in Q_0$ and denote
$$M_{out} := \bigoplus_{\rho \in Q_1:s(\rho)=k}M_{t(\rho)}, \qquad M_{in} := \bigoplus_{\rho\in Q_1:t(\rho) = k}M_{s(\rho)},$$
$$\beta_{M,k} := \sum_{\rho\in Q_1: s(\rho) = k}M_\rho:M_k \rightarrow M_{out},\qquad \alpha_{M,k} := \sum_{\rho\in Q_1: t(\rho) = k}M_\rho:M_{in} \rightarrow M_k.$$
\end{defn}

\begin{rem}\label{rem:muted rep multiplicity}
The multiplicity of $M_j$ as a direct summand of $M_{out}$ is precisely $\max\{B_{k,j},0\}$ and the multiplicity of $M_j$ as a direct summand of $M_{in}$ is precisely $\max\{-B_{k,j},0\}$. Moreover, the condition that $\dim\Hom(S(k), M) = 0$ is equivalent to $\beta_{M,k}$ being injective and the condition that $\dim\Hom(M, S(k)) = 0$ is equivalent to $\alpha_{M,k}$ being surjective.
\end{rem}

\begin{nota}\label{notation:sides of wall}
We denote by $\cS(k)^+(Q)$ and $\cS(k)^-(Q)$ the full subcategories $\Hom(S(k), -) = 0$ and $\Hom(-,S(k)) = 0$ of $\mods J(Q,W)$.
\end{nota}

\begin{rem}\label{rem:brick on one side}
Note in particular that if
$M \ncong S(k)$ is a brick, then $M \in \cS(k)^+(Q) \cup \cS(k)^-(Q)$. Indeed, if $M \notin \cS(k)^+(Q) \cup \cS(j)^-(Q)$, then there are nonzero morphisms $M \twoheadrightarrow S(k) \hookrightarrow M$. Since $M$ is a brick, this composition must be an automorphism, and so $M \cong S(k).$
\end{rem}

We now recall the following mutation formulas from \cite[Section~4.2]{mou_scattering}. 

\begin{thm}\label{thm:mouFormula}
	Let $J(Q,W)$ be an arbitrary cluster-tilted algebra. Then there are functors $F(Q)_k^+, F(Q)_k^-: \mods J(Q,W) \rightarrow \mods J(\mu_k(Q,W))$ with the following properties.
	\begin{enumerate}
		\item $(F(Q)_k^+, F(\mu_kQ)_k^-)$ and $(F(Q)_k^-, F(\mu_kQ)_k^+)$ are adjoint pairs.
		\item $F(Q)_k^+$ induces an equivalence of categories $\cS(k)^+(Q) \rightarrow \cS(k)^-(\mu_kQ)$ which preserves short exact sequences.
		\item $F(Q)_k^-$ induces an equivalence of categories $\cS(k)^-(Q) \rightarrow \cS(k)^+(\mu_k Q)$ which preserves short exact sequences.
		\item Let $M \in \mods J(Q,W)$. For $j \neq k$, $F(Q)_k^+ M_j = M_j = F(Q)_k^- M_j$. Moreover, $F(Q)_k^+M_k = \coker \beta_{M,k}$ and $F(Q)_k^- M_k = \ker \alpha_{M,k}$.
	\end{enumerate}
\end{thm}

We will also need the following, which is similar to \cite[Theorem 4.5(5)]{LL_maximal}.

\begin{prop}\label{prop:mouFormula} Let $J(Q,W)$ be an arbitrary cluster-tilted algebra.
	\begin{enumerate}
	\item Let $k \in Q_0$ and let $M \in \mods J(Q,W)$. If $\dim\Hom(S(k), M) = 0$, then $\undim F(Q)_k^+ M = A_k^+ \undim M$.
	\item  Let $k \in Q_0$ and let $M \in \mods J(Q,W)$. If $\dim\Hom(M, S(k)) = 0$, then $\undim F(Q)_k^- M = A_k^- \undim M$.
	\end{enumerate}
\end{prop}

\begin{proof}
	(1) Since $\Hom(S(k), M) = 0$, we have that $\beta_{M,k}$ is injective. This means $\dim\coker \beta_{M,k} = \dim M_{out} - \dim M_k$. Thus we have
		\begin{eqnarray*}
			\undim F(Q)_k^+ M &=& e_k\cdot(\dim M_{out} - \dim M_k)+ \sum_{j \neq k} e_j\cdot\dim M_j\\
				&=& e_k \cdot\left(- \dim M_k + \sum_{j \neq k}\max\{0, B_{k,j}\}\dim M_j\right) + \sum_{j \neq k} e_j\cdot\dim M_j\\
				&=& A^+_k \undim M.
		\end{eqnarray*}
		
	(2) Since $\Hom(M, S(k)) = 0$, we have that $\alpha_{M,k}$ is surjective. This means $\dim\ker\alpha_{M,k} = \dim M_{in} - \dim M_k$. Thus we have
		\begin{eqnarray*}
			\undim F(Q)_k^- M &=& e_k\cdot(\dim M_{in} - \dim M_k)+ \sum_{j \neq k} e_j\cdot\dim M_j\\
				&=& e_k \cdot\left(- \dim M_k + \sum_{j \neq k}\max\{0, -B_{k,j}\}\dim M_j\right) + \sum_{j \neq k} e_j\cdot\dim M_j\\
				&=& A^-_k \undim M.
		\end{eqnarray*}
\end{proof}

We now apply the mutation formulas of Mou and Reading to the null root and show that the result is the null root of the mutated quiver.

\begin{nota}\label{nota:homogeneous}
For a tame cluster-tilted algebra $\Lambda = \mods J(Q,W)$, we denote by $\eta(Q)$ the null root of $\Lambda$ and $M_{\lambda}(Q)$ a homogeneous $J(Q,W)$-module of dimension $\eta(Q)$ with arbitrary parameter $\lambda \in K^*$.
\end{nota}

\begin{prop} \label{prop:nullToNull}
	Let $J(Q,W)$ be a tame cluster-tilted algebra.
\begin{enumerate}
	\item Then $M_\lambda(Q)$ is a non-simple brick.
	\item Let $k \in Q_0$. Then at least one of $\dim\Hom(S(k),M_\lambda)$ and $\dim\Hom(M_\lambda, S(k))$ is 0.
	\item Let $k \in Q_0$. If $\dim\Hom(S(k),M_\lambda(Q)) = 0$, then $\eta(\mu_k Q) = A^+_k \eta(Q)$.
	\item Let $k \in Q_0$. If $\dim\Hom(M_\lambda(Q), S(k)) = 0$, then $\eta(\mu_k Q) = A^-_k \eta(Q)$.
	\end{enumerate}
\end{prop}

\begin{proof}
	Write $(Q,W) = \mu_{i_m}\circ\cdots\circ\mu_{i_1}(\Gamma)$ where $\Gamma$ is a Euclidean quiver. We will prove these results by induction on $m$. When $m = 0$, we are in the hereditary case where this result is known.
	
	Let $(Q', W') = \mu_{i_m}(Q,W)$. By the induction hypothesis, we have that $M_\lambda(Q')$ is a brick. Thus at least one of $\dim\Hom(S(i_m),M_\lambda(Q'))$ and $\dim\Hom(M_\lambda(Q,), S(i_m))$ is zero. We will assume $\dim\Hom(S(i_m), M_\lambda(Q')) = 0$ as the other case follows analogously. Now by Proposition \ref{prop:mouFormula}, we have that $\undim F(Q')_{i_m}^+ M_\lambda(Q')$ does not depend on the choice of parameter $\lambda$. Moreover, by Theorem \ref{thm:mouFormula}(2), we have that $\{F(Q')_{i_m}^+ M_\lambda(Q')\}_{\lambda \in K^*}$ is a set of pairwise non-isomorphic bricks. This is only possible if these bricks are the mouths of the homogeneous tubes in $\mods J(Q,W)$, which are not simple and of dimension $\eta(Q)$. This proves (1), (3), and (4). (2) then follows immediately from (1).
\end{proof}

Now that we have shown how the null root $\eta$ behaves under mutation, we focus on studying how the $g$-vector of $\eta$ behaves under mutation. To do so, we recall the following formula from \cite{DWZ_quivers}. Note that since $M_\lambda$ is indecomposable and not simple, we are able to state these results without mention of the associated \emph{decorations}.

Let $(Q,W)$ be a quiver with potential and $k \in Q_0$. Fix arrows $\rho:s(\rho) \rightarrow k$ and $\sigma:k \rightarrow t(\sigma)$. Following \cite[Definition~3.1]{DWZ_quivers} (with the same caveats on notational conventions as discussed before Definition~\ref{def:mutated rep}), for each simple cycle $\omega_1\cdots \omega_p$ in $Q$ with indices identified modulo $p$, we denote 
$$\partial_{[\rho\sigma]}(\omega_1\cdots\omega_p) = \begin{cases} \omega_{d+1}\omega_{d+2}\cdots \omega_{d+p-2} & \exists d: \omega_{d-1} = \sigma \text{ and } \omega_{d} = \rho\\0 & \text{otherwise}. \end{cases}$$
We then extend additively to define $\partial_{[\rho\sigma]}W$. As in \cite[Equation~10.4]{DWZ_quivers}, $\partial_{[\rho\sigma]}W$ induces a map $\gamma_{\rho,\sigma}: M_{t(\sigma)} \rightarrow M_{s(\rho)}$. Taking the sum over all possible choices of $\rho$ and $\sigma$ thus yields a map $\gamma_{M,k}: M_{out} \rightarrow M_{in}$.

\begin{thm}\label{thm:DWZFormula}
	Let $J(Q,W)$ be an arbitrary cluster-tilted algebra. For every non-simple indecomposable $M \in \mods J(Q,W)$ and every $k \in Q_0$, there is an indecomposable $\mu_k M \in \mods J\mu_k(Q,W)$ with the following properties:
	\begin{enumerate}
		\item \cite[Lemma 5.2]{DWZ_quivers2} If $g(M)\cdot \undim S(k) \geq0$, then $g(\mu_k M)^{tr} = g(M)^{tr} A_k^+$.
		\item \cite[Lemma 5.2]{DWZ_quivers2} If $g(M)\cdot \undim S(k) \leq0$, then $g(\mu_k M)^{tr} = g(M)^{tr} A_k^-$. 
		\item \cite[Prop 10.7]{DWZ_quivers} For $j \neq k$, $\mu_k M_j = M_j$. Moreover, $\mu_k M_k \cong \frac{\ker \gamma_{M,k}}{\im \beta_{M,k}} \oplus \ker \alpha_{M,k}$.
		\item \cite[Prop 6.2]{DWZ_quivers2} For all nonsimple $M'\in J(Q,W)$, $$\dim \Hom(M, M') - \dim \Hom^{[k]}(M,M') = \dim \Hom(\mu_k M, \mu_k M') - \dim \Hom^{[k]}(\mu_k M, \mu_k M'),$$ where $\Hom^{[k]}(-,-)$ is the set of morphisms which are only supported at the vertex $k$.
	\end{enumerate}
\end{thm}

We now use this result to study the $g$-vector of the null root.

\begin{prop}\label{prop:nullgToNullg}
	Let $J(Q,W)$ be a tame cluster-tilted algebra, and let $k \in Q_0$.
	\begin{enumerate}
		\item If $g(\eta(Q))\cdot\undim S(k) \geq 0$, then $g(\eta(\mu_k Q))^{tr} = g(\eta(Q))^{tr} A_k^+$.
		\item If $g(\eta(Q))\cdot\undim S(k) \leq 0$, then $g(\eta(\mu_k Q))^{tr} = g(\eta(Q))^{tr} A_k^-$.
	\end{enumerate}
\end{prop}

\begin{proof}
	Recall from Proposition \ref{prop:nullToNull}(1) that $M_\lambda(Q)$ is a non-simple brick. Thus by Theorem \ref{thm:DWZFormula}(3), we have that $\{\mu_k M_\lambda\}_{\lambda\in K^*}$ is a collection of indecomposable $J(Q,W)$-modules with the same dimension vector. Now let $\lambda \neq \lambda' \in K^*$ and assume for a contradiction that there is an isomorphism $\mu_k M_\lambda(Q) \rightarrow \mu_k M_{\lambda'}(Q)$. By Theorem \ref{thm:DWZFormula}(4), this isomorphism is supported only at the vertex $k$. Thus it must be the case that $\mu_k M_\lambda(Q) \cong S(k)$. By Theorem \ref{thm:DWZFormula}(3), this means $(M_\lambda(Q))_j = 0$ for $j \neq k$; that is, $M_\lambda(Q)$ is simple. This is a contradiction.
	
	We have thus shown that $\{\mu_k M_\lambda(Q)\}_{\lambda\in K^*}$ is a collection of pairwise hom-orthogonal indecomposable modules. Thus these modules lie in the homogeneous tubes in $\mods J(Q,W)$. In particular, any morphism $\mu_k M_\lambda(Q) \rightarrow \mu_k M_\lambda(Q)$ is supported on all vertices in the support of $\eta(\mu Q)$. As $M_\lambda(\mu Q)$ is non-simple, no such morphism is an element of $\Hom^{[k]}(\mu_k M_\lambda(Q), \mu_k M_\lambda(Q))$. By Theorem \ref{thm:DWZFormula}(4), this means $\mu_k M_\lambda(Q)$ is a brick since $$\dim\End(\mu_kM_\lambda(Q)) = \dim\End(M_\lambda(Q)) = 1.$$ Therefore $\mu_k M_\lambda(Q) = M_\lambda(\mu_kQ)$. The result then follows from Theorem \ref{thm:DWZFormula}(1,2).
\end{proof}

We have shown that the mutation formulas of Mou, Derksen-Weyman-Zelevinsky, and Reading send the null root to the null root and the $g$-vector of the null root to the $g$-vector of the null root. In order to prove Theorem \ref{genReadingFormula}, we recall the following result of Mou.

\begin{prop}\cite[Lemma 4.12, 4.13]{mou_scattering}\label{prop:wallsToWalls}
	Let $J(Q,W)$ be an arbitrary cluster-tilted algebra. Let $M \in \mods J(Q,W)$ and let $v \in \RR^n$.
	\begin{enumerate}
		\item If $M \in \cS(k)^-(Q)$, then $v \in D(M)$ if and only if $v\cdot \undim M = 0$ and $v\cdot \undim M' \leq 0$ for all $M' \subseteq M$ with $M' \in \cS(k)^-(Q)$.
		\item If $M \in \cS(k)^+(Q)$, then $v \in D(M)$ if and only if $v\cdot \undim M = 0$ and $v\cdot \undim M' \geq 0$ for all quotient modules $M'$ of $M$ with $M' \in \cS(k)^+(Q)$.
	\end{enumerate}
\end{prop}

We need one additional lemma to prove Theorem \ref{genReadingFormula}.

\begin{lem}\label{lem:oneSide}
	Let $J(Q,W)$ be an arbitrary cluster-tilted algebra. Let $D(M)$ be a wall in the standard wall-and-chamber structure $\mathfrak{D}(J(Q,W))$. If there exists $v \in D(M)$ such that $v\cdot \undim S(k) < 0$, then $M \in S(k)^-(Q)$. Likewise, if there exists $v \in D(M)$ such that $v\cdot \undim S(k) > 0$, then $M \in S(k)^+(Q)$.
\end{lem}

\begin{proof}
	Suppose $M \notin S(k)^-(Q)$. Then there exists a necessarily surjective morphism $M \rightarrow S(k)$. Thus if $v\in D(M)$, then $v\cdot \undim S(k) \geq 0$. Likewise, if $M \notin S(k)^+(Q)$, then there exists a necessarily injective morphism $S(k) \rightarrow M$. Thus if $v \in D(M)$, then $v\cdot\undim S(k) \leq 0$.
\end{proof}

\begin{proof}[Proof of Theorem \ref{genReadingFormula}]
	(1) Let $D(M)$ be a wall in the standard wall-and-chamber structure \linebreak $\mathfrak{D}(J(Q,W))$ and suppose that there exists $w \in D(M)$ so that $w \cdot \undim S(k) > 0$. Then by Lemma \ref{lem:oneSide}, $M \in S(k)^+(Q)$. Now let $v \in D(M)$ so that $v \cdot \undim S(k) \geq 0$. By definition, we have that $(v^{tr} A_k^+)(A_k^+\undim M) = 0$. Moreover, we have that $A_k^+\undim M = \undim M'$ for some brick $M' \in \cS(k)^-(\mu_k Q)$ by Theorem \ref{thm:mouFormula}(2) and Proposition \ref{prop:mouFormula}(2). Similarly, if $N' \subseteq M'$ and $N' \in \cS(k)^-(\mu_k Q)$, then the short exact sequence $N' \hookrightarrow M' \twoheadrightarrow M'/N'$ is in $\cS(k)^-(\mu_k Q)$. Thus Theorem \ref{thm:mouFormula}(2) and Proposition \ref{prop:mouFormula}(2) again imply that we can write $\undim N' = A_k^+\undim N$ for some $N \subseteq M$ in $\cS(k)^+(Q)$. In particular, this means $(v^{tr} A_k^+)\undim N' \leq 0$, so $(v^{tr} A_k^+) \in D(M')$ by Proposition \ref{prop:wallsToWalls}(1).
	
	The proof of (2) is similar and (3) is already part of Theorem \ref{readingFormula}.\end{proof}
	
%%%%%%%%%%%%%%%%%%%%%%%%%%%%%%%%%%%%%%%%%

We conclude this section with two brief corollaries that follow from comparing these mutation formulas.

\begin{cor}\label{cor:regular mutation}
	Let $J(Q,W)$ be a tame cluster-tilted algebra. If the simple module $S(k)$ is regular, then $g(\eta(\mu_kQ)) = g(\eta(Q))$.
\end{cor}

\begin{proof}
	If $S(k)$ is regular, then $g(\eta(Q))\cdot \undim S(k) = g(\eta(Q))_k = 0$. Thus we have
	$g(\eta(Q))^{tr} A_k^+ = g(\eta(Q))^{tr}$ and $g(\eta(Q))^{tr} A_k^- = g(\eta(Q))$. The result then follows immediately from Proposition~\ref{prop:nullgToNullg}.
\end{proof}

\begin{cor}\label{cor:minimal cycle}
	Let $J(Q,W)$ be a cluster-tilted algebra of type $\widetilde{A}_n$. Then $Q$ contains a unique minimal unoriented cycle. Moreover, the null root $\eta(Q)$ is supported on exactly the vertices in this minimal cycle.
\end{cor}

\begin{proof}
	The existence of the unique minimal cycle is well-known (see \cite{bastian}). As there exists a 1-parameter family of bricks supported on this cycle, their dimension vector must be the null root.
\end{proof}

%%%%%%%%%%%%%%%%%%%%%%%%%%%%%%%%%%%%%%%%%%%%%%%%%%%%%%
\subsection{Statement and proof of Theorem~\ref{thmintro:mainD}}\label{sec:thmD}

Let $J(Q,W)$ be a tame cluster-tilted algebra. We recall that we have defined $M \in \mods\Lambda$ to be regular if $\dim\Hom(M, M_\lambda) = 0 = \dim\Hom(M_\lambda, M)$, or equivalently, if $g(\eta) \in D(M)$. Moreover, recall that the walls in $\fD_{reg}(J(Q,W))=\fD_0(J(Q,W))$ are labeled by regular modules, and that each of these wall-and-chamber structures comes with an embedding into $\RR^{n}$. We denote by $\pi_Q: \RR^{n} \rightarrow \RR^{n-1}$ the projection onto $g(\eta(Q))^\perp$.

We are now ready to state and prove Theorem~\ref{thmintro:mainD}.

\begin{thm}[Theorem~\ref{thmintro:mainD}]\label{thmD}
	Let $J(Q,W)$ be a tame cluster-tilted algebra and let $k \in Q_0$. Let $(Q',W') = \mu_k(Q,W)$.
	\begin{enumerate}
		\item If $g(\eta(Q))\cdot \undim S(k) > 0$, then $\fD_{0}(J(Q',W')) = \pi_{Q'} \circ (A_k^+)^{tr} \fD_{0}(J(Q,W))$.
		\item If $g(\eta(Q))\cdot \undim S(k) < 0$, then $\fD_{0}(J(Q',W')) = \pi_{Q'} \circ (A_k^-)^{tr} \fD_{0}(J(Q,W))$.
		\item If $g(\eta(Q))\cdot \undim S(k) = 0$, then $\fD_{0}(J(Q',W')) = \pi_{Q'} \circ \psi\fD_{0}(J(Q,W))$, where
		$$\psi(g) = \begin{cases} (A_k^+)^{tr} v & v\cdot\undim S(k) \geq 0\\(A_k^-)^{tr} v & v\cdot\undim S(k) \leq 0\end{cases}$$
	\end{enumerate}
\end{thm}

\begin{proof}
Let $M \in \mods J(Q,W)$ be a regular brick, and suppose that $M \ncong S(k)$. (Note also that in cases (1) and (2) this assumption is satisfied for all regular bricks.) We claim that in each case, elements $v \in D_0(M)$ are mapped to elements of $D_0(M')$ with $\undim M' = A_k^{\pm}\undim M$ where the sign(s) depend on the sign of $g(\eta(Q))\cdot \undim S(k)$.

(1) Let $d = \undim M$. We observe that $g(\eta(Q))\cdot d = 0$ if and only if $(A_k^+)^{tr} g(\eta(Q))\cdot A_k^+ d = 0$. By Proposition \ref{prop:nullgToNullg}, this is equivalent to $g(\eta(Q'))\cdot A_k^+d = 0$.

Now let $v \in D_0(M)$, so by definition there exists $\varepsilon > 0$ such that $g(\eta(Q)) + \varepsilon v \in D(M)$. By Theorem~\ref{genReadingFormula}, this means there exists a regular brick $M' \in \mods J(Q',W')$ with $A_k^+ d = \undim M'$ for which 
$g(\eta(Q')) + \varepsilon (A_k^+)^{tr} v \in D(M')$. Now write $\varepsilon (A_k^+)^{tr} v = \varepsilon\cdot\pi_{Q'}((A_k^+)^{tr} v) + \varepsilon t\cdot g(\eta(Q'))$. Adjusting $\varepsilon$ if necessary, we can assume $ \frac{\varepsilon}{1+\varepsilon t} > 0$ and thus
$$g(\eta(Q')) + \frac{\varepsilon}{1+\varepsilon t}\pi_{Q'}((A_k^+)^{tr} v) \in D(M').$$
That is, $\pi_{Q'}((A_k^+)^{tr} v) \in D_0(M')$.

We observe that if $g \cdot \undim S(k) > 0$, then $(A_k^+)^{tr} g \cdot \undim S(k) < 0$. In this case, the new mutation matrix $(A')_k^-$ associated to $\Lambda'$ is the same as the matrix $A_k^+$ (see \cite{BHIT_semi-invariant}). Thus an analogous argument shows that $\pi_Q((A_k^+)^{tr} v) \in D_0(M)$ for any $v \in D_0(M')$. Moreover, for all $v \in \RR^{|Q_0|-1} = g(\eta(Q))^\perp$, there exists $t \in \RR$ so that
	\begin{eqnarray*}
		\pi_Q \circ (A_k^+)^{tr} \circ \pi_{Q'}\circ (A_k^+)^{tr} v &=& \pi_Q \circ (A_k^+)^{tr} \left[t\cdot g(\eta(Q')) + (A_k^+)^{tr} v\right]\\
		&=& \pi_Q\left[t \cdot g(\eta(Q)) + v\right]\\
		&=& v
	\end{eqnarray*}
This completes the proof of (1). The proof of (2) is identical to that of (1) with all of the signs reversed.

To prove (3), we recall from Proposition \ref{prop:nullgToNullg} that we have $(A_k^+)^{tr} g(\eta(Q)) = (A_k^-)^{tr} g(\eta(Q)) = g(\eta(Q'))$. Now let $v \in D_0(S(k))$, so by definition there exists $\varepsilon > 0$ such that $g(\eta) + \varepsilon v \in D(S(k))$. Theorem~\ref{readingFormula}(3) then implies that $g(\eta') + \varepsilon (A_k^+)^{tr} v = g(\eta') + \varepsilon (a_k^-)^{tr} v \in D(S'(k))$, where $S'(k)$ denotes the simple $J(Q',W')$-module at the vertex $k$. In particular, this means $\pi_Q'\circ \psi(v) \in D_0(S'(k))$. An argument similar to that in case (1) therefore implies that $\pi_{Q'}\circ \psi(D_0(S(k))) = D_0(S'(k))$.

Returning now to the case where $M$ is a regular brick different from $S(k)$, we note that, for any $v \in \RR^n$, the sign of $(g(\eta(Q)) + \varepsilon v)\cdot \undim S(k)$ is the same as the sign of $v\cdot\undim S(k)$ for all $\varepsilon > 0$. Thus the two sides of $D_0(M)$ - namely $\{v \in D_0(M): v\cdot\undim S(k) > 0\}$ and $\{v \in D_0(M): v\cdot\undim S(k) < 0\}$ - can be treated separately using the same arguments as above.
\end{proof}

\begin{cor}\label{cor:thmD}
	Let $\Gamma$ be a Euclidean quiver and let $J(Q,W)$ be a cluster-tilted algebra of type $\Gamma$. Then $\fD_{reg}(J(Q,W))$ is piecewise-linearly isomorphic to $\fD_{reg}(K\Gamma)$.
\end{cor}
	\begin{proof}
		This is a direct consequence of Theorem \ref{thmA} and Theorem \ref{thmD}.
	\end{proof}
	
Examples of the mutation of regular semi-invariant pictures described in Theorem \ref{thmD} and Corollary \ref{cor:thmD} are shown in Figures \ref{fig:mutation1} and \ref{fig:mutation2}.

{\scriptsize
\begin{center}
\begin{figure}
\begin{tikzpicture}[scale=.7] % Huge BBBBBBBBBBBBBB
%	\draw[help lines=1] (-5,-5) grid (5,4);
%	\draw[fill] (0,0) circle [radius=2pt];

\begin{scope}\clip (-6, -4.75) rectangle (6,4);

	\draw[very thick, dashed] (0,0) circle [radius=2.2cm]; % null circle
		\begin{scope} %begin upper right circle and semicircle
	%	\draw[fill] (1.73,1) circle [radius=2pt];
		\draw[very thick] (1.73,1) circle [radius=2.73cm];
		\clip (0,-2) rectangle (-4,4);
		\draw[very thick,dotted] (0,1) ellipse [x radius=2.8cm,y radius=2.1cm];
		\end{scope} % end upper right circle and semicircle
		\begin{scope}[rotate=120] %begin upper left circle and semicircle
	%	\draw[fill] (1.73,1) circle [radius=2pt];
		\draw[very thick] (1.73,1) circle [radius=2.73cm];
		\clip (0,-2) rectangle (-4,4);
		\draw[very thick,dotted] (0,1) ellipse [x radius=2.8cm,y radius=2.1cm];
		\end{scope} % end upper left circle and semicircle
		\begin{scope}[rotate=-120] %begin lower circle and semicircle
	%	\draw[fill] (1.73,1) circle [radius=2pt];
		\draw [very thick] (1.73,1) circle [radius=2.73cm];
		\clip (0,-2) rectangle (-4,4);
		\draw[very thick,dotted] (0,1) ellipse [x radius=2.8cm,y radius=2.1cm];
		\end{scope} % end lower circle and semicircle
		
	\draw (4,3) node{$3$};
	\draw (-4,3) node{$2$};
	\draw (-2.9,2) node{$23$};
	\draw (2,2) node[above]{$341$};
	\draw (0.5,-3.25) node{$412$};
	\draw (2,-4.5) node{$41$};
	%\draw (0,2.2) node[above]{$\eta_{24}$};
	%\draw (-2,2) node[right]{$\eta_{24}$};
	%\draw (2,1) node[right]{$\eta_{34}$};
	%
	%\draw[font=\huge,color=brown] (0,0) node{B};
\begin{scope}[xshift=-50mm,yshift=-40mm,scale=1.5]
\coordinate (A) at (.5,1);
\coordinate (B) at (0,.5);
\coordinate (Bm) at (0.1,.6);
\coordinate (Bp) at (0.1,.4);
\coordinate (C) at (1,.5);
\coordinate (Cm) at (.9,.6);
\coordinate (D) at (0.5,0);
\coordinate (Dm) at (0.4,0.1);
\coordinate (Dr) at (0.6,0.1);
\draw[thick,->] (A)--(Bm);
\draw[thick,<-] (Bp)--(D);
\draw[thick,->] (C)--(Dr);
\draw[thick,->] (A)--(Cm);
\foreach \x/\y in {A/1,B/4,C/2,D/3}
\draw[fill,white] (\x) circle[radius=4pt];
\foreach \x/\y in {A/1,B/4,C/2,D/3}
\draw (\x) node{$\y$};
\end{scope}

\end{scope}

\begin{scope}[shift = {(10,0)}]

\begin{scope}[xshift=-50mm,yshift=-40mm,scale=1.5]
\coordinate (A) at (.5,1);
\coordinate (Ar) at (.6,.9);
\coordinate (Ap) at (.5,.8);
\coordinate (B) at (0,.5);
\coordinate (Bm) at (0.1,.6);
\coordinate (Bp) at (0.1,.4);
\coordinate (C) at (1,.5);
\coordinate (Cb) at (.9,.4);
\coordinate (Cm) at (.9,.6);
\coordinate (D) at (0.5,0);
\coordinate (Dp) at (0.5,0.2);
\coordinate (Dm) at (0.4,0.1);
\coordinate (Dr) at (0.6,0.1);
\draw[thick,->] (A)--(Bm);
\draw[thick,->] (Ap)--(Dp);
\draw[thick,<-] (Bp)--(D);
\draw[thick,<-] (Cb)--(Dr);
\draw[thick,<-] (Ar)--(Cm);
\foreach \x/\y in {A/1,B/4,C/2,D/3}
\draw[fill,white] (\x) circle[radius=4pt];
\foreach \x/\y in {A/1,B/4,C/2,D/3}
\draw (\x) node{$\y$};
\end{scope}

%	\draw[help lines=1] (-5,-5) grid (5,4);
%	\draw[fill] (0,0) circle [radius=2pt];
\begin{scope}\clip (-6, -4.75) rectangle (6,4);

\begin{scope}[xscale=-1]

		\begin{scope} %begin upper right circle and semicircle
	%	\draw[fill] (1.73,1) circle [radius=2pt];
		\begin{scope} \clip (0, -2) rectangle (5, 4);
		\draw[very thick] (1.73,1) circle [radius=2.73cm];
		\end{scope}\begin{scope}
		\clip (0, -2) rectangle (-4, 4);
		\draw[very thick] (1.73,1) circle [radius=2.73cm];
		\end{scope}\begin{scope}
		\clip (0,-2) rectangle (-4,4);
		\draw[very thick] (0,1) ellipse [x radius=2.8cm,y radius=2.1cm];
		\end{scope} \end{scope}% end upper right circle and semicircle
%%
		%begin upper left circle and semicircle
		\begin{scope} \clip (0, -2) rectangle (-5, 4);
		\draw[very thick,dotted] (-1.73,1) circle [radius=2.73cm];
		\end{scope}
		\begin{scope}
		\clip (0, -2) rectangle (4, 4);
		\draw[very thick] (-1.73,1) circle [radius=2.73cm];
		\end{scope}
		
		\begin{scope}[rotate=120]
	%	\draw[fill] (1.73,1) circle [radius=2pt];
		\begin{scope}
		\clip (0,-2) rectangle (-4,4);
		\end{scope}\end{scope}% end upper left circle and semicircle
		\begin{scope}[rotate=-120] %begin lower circle and semicircle
	%	\draw[fill] (1.73,1) circle [radius=2pt];
				\begin{scope} \clip (0, -2) rectangle (5, 4);
		\draw[very thick] (1.73,1) circle [radius=2.73cm];
		\end{scope}\begin{scope}
		\clip (0, -2) rectangle (-4, 4);
		\draw[very thick,dotted] (1.73,1) circle [radius=2.73cm];
		\end{scope}\begin{scope}
		\clip (0,-2) rectangle (-4,4);
		\draw[very thick] (0,1) ellipse [x radius=2.8cm,y radius=2.1cm];
		\end{scope} \end{scope}% end lower circle and semicircle
%
%\draw[font=\huge,color=brown] (0,0) node{B};
\end{scope}

\begin{scope}[rotate=-120]
\draw[very thick,dashed] (0,1) ellipse [x radius=2.8cm,y radius=2.1cm];
\end{scope}
\begin{scope}\clip (3, -0.35) rectangle (-0.8, 3);
\draw[very thick,dotted] (0,0) circle [radius = 2.2cm];
\end{scope}

\draw (4,3) node{$3$};
\draw (-4,3) node{$2$};
\draw (-2.9,2) node{$412$};
\draw (2,2.5) node[right]{$23$};
\draw (2,-4.5) node{$41$};
%\draw (0,2.2) node[above]{$\eta_{24}$};
%\draw (-2,2) node[right]{$\eta_{24}$};
%\draw (2,1) node[right]{$\eta_{34}$};
\draw (2,1.4) node{$2341$};
\draw (0, 2.1) node[above]{$3412$};
\end{scope}

\end{scope}

\end{tikzpicture}
\caption{The regular semi-invariant pictures for the quivers $1\rightarrow 2 \rightarrow 3 \rightarrow 4 \leftarrow 1$ (left) and $1 \rightarrow 3 \rightarrow 4 \leftarrow 1 \leftarrow 2 \leftarrow 3$ (right).  Each wall is labeled with the support (both vertices and arrows) of its brick except the null wall, which is dashed. For example, 3412 in the picture on the right is shorthand for the brick which has dimension 1 at each vertex and is nonzero only on the arrows $3 \rightarrow 4 \leftarrow 1 \leftarrow 2$.} These pictures are piecewise-linearly isomorphic since the quivers are related by a mutation at the vertex 2. The shading of each (half-)wall in the second picture matches its preimage under the mutation formula.\label{fig:mutation1}
\end{figure}
\end{center}
}

{\scriptsize
\begin{center}
\begin{figure}
\begin{tikzpicture}[scale=.7] % Huge BBBBBBBBBBBBBB
%	\draw[help lines=1] (-5,-5) grid (5,4);
%	\draw[fill] (0,0) circle [radius=2pt];

\begin{scope}\clip (-6, -4.75) rectangle (6,4);

\begin{scope}[xshift=-50mm,yshift=-43mm,scale=1.5]
\coordinate (A) at (.5,1);
\coordinate (Ab) at (.4,.9);
\coordinate (Ar) at (.6,.9);
\coordinate (Ap) at (.5,.8);
\coordinate (B) at (0,.5);
\coordinate (Bm) at (0.1,.6);
\coordinate (Bp) at (0.1,.4);
\coordinate (C) at (1,.5);
\coordinate (Cb) at (.9,.4);
\coordinate (Cm) at (.9,.6);
\coordinate (D) at (0.5,0);
\coordinate (Dp) at (0.5,0.2);
\coordinate (Dm) at (0.4,0.1);
\coordinate (Dr) at (0.6,0.1);
\draw[thick,<-] (Ab)--(Bm);
\draw[thick,->] (Ap)--(Dp);
\draw[thick,->] (Bp)--(Dm);
\draw[thick,<-] (Cb)--(Dr);
\draw[thick,<-] (Ar)--(Cm);
\foreach \x/\y in {A/4,B/3,C/2,D/1}
\draw[fill,white] (\x) circle[radius=4pt];
\foreach \x/\y in {A/4,B/3,C/2,D/1}
\draw (\x) node{$\y$};
\end{scope}

\begin{scope}[xscale=-1]

		\begin{scope} %begin upper right circle and semicircle
	%	\draw[fill] (1.73,1) circle [radius=2pt];
		\begin{scope} \clip (0, -2) rectangle (5, 4);
		\draw[very thick] (1.73,1) circle [radius=2.73cm];
		\end{scope}\begin{scope}
		\clip (0, -2) rectangle (-4, 4);
		\draw[very thick] (1.73,1) circle [radius=2.73cm];
		\end{scope}\begin{scope}
		\clip (0,-2) rectangle (-4,4);
		\draw[very thick,dotted] (0,1) ellipse [x radius=2.8cm,y radius=2.1cm];
		\end{scope} \end{scope}% end upper right circle and semicircle
%%
		%begin upper left circle and semicircle
		\begin{scope} \clip (0, -2) rectangle (-5, 4);
		\draw[very thick] (-1.73,1) circle [radius=2.73cm];
		\end{scope}
		\begin{scope}
		\clip (0, -2) rectangle (4, 4);
		\draw[very thick] (-1.73,1) circle [radius=2.73cm];
		\end{scope}
		\begin{scope}[rotate=120]
	%	\draw[fill] (1.73,1) circle [radius=2pt];
		\begin{scope}
		\clip (0,-2) rectangle (-4,4);
		\end{scope}\end{scope}% end upper left circle and semicircle
		\begin{scope}[rotate=-120] %begin lower circle and semicircle
	%	\draw[fill] (1.73,1) circle [radius=2pt];
		\begin{scope} \clip (0, -2) rectangle (5, 4);
		\draw[very thick] (1.73,1) circle [radius=2.73cm];
		\end{scope}\begin{scope}
		\clip (0, -2) rectangle (-4, 4);
		\draw[very thick] (1.73,1) circle [radius=2.73cm];
		\end{scope}\begin{scope}
		\clip (0,-2) rectangle (-4,4);
		\draw[very thick,dotted] (0,1) ellipse [x radius=2.8cm,y radius=2.1cm];
		\end{scope} \end{scope}% end lower circle and semicircle
%
%\draw[font=\huge,color=brown] (0,0) node{B};

\end{scope}
\begin{scope}[rotate=-120]
\draw[very thick,dashed] (0,1) ellipse [x radius=2.8cm,y radius=2.1cm];
\end{scope}
	
\begin{scope}\clip (3, -0.35) rectangle (-0.8, 3);
\draw[very thick,dotted] (0,0) circle [radius = 2.2cm];
\end{scope}

\draw (4,3) node{$31$};
\draw (-4,3) node{$2$};
\draw (-2.9,2) node{$24$};
\draw (2,2.5) node[right]{$312$};
\draw (2,-4.5) node{$4$};
\draw (2, 1.4) node{$4312$};
\draw (0, 2.1) node[above]{$2431$};

\end{scope}
\begin{scope}[shift = {(10,0)}]

\begin{scope}[xshift=-50mm,yshift=-43mm,scale=1.5]
\coordinate (A) at (.5,1);
\coordinate (Ar) at (.6,.9);
\coordinate (Ap) at (.5,.8);
\coordinate (B) at (0,.5);
\coordinate (B1) at (0.05,.5);
\coordinate (B2) at (.02,.32);
\coordinate (Bm) at (0.1,.6);
\coordinate (Bp) at (0.1,.4);
\coordinate (C) at (1,.5);
\coordinate (Cb) at (.9,.4);
\coordinate (Cm) at (.9,.6);
\coordinate (D) at (0.5,0);
\coordinate (D1) at (0.4,0.15);
\coordinate (D2) at (0.3,0.05);
\coordinate (Dp) at (0.5,0.2);
\coordinate (Dm) at (0.4,0.1);
\coordinate (Dr) at (0.6,0.1);
\draw[thick,->] (A)--(Bm);
\draw[thick,<-] (Ap)--(Dp);
\draw[thick,->] (B1)--(D1);
\draw[thick,->] (B2)--(D2);
%\draw[thick,<-] (Cb)--(Dr);
\draw[thick,->] (Ar)--(Cm);
\foreach \x/\y in {A/4,B/3,C/2,D/1}
\draw[fill,white] (\x) circle[radius=4pt];
\foreach \x/\y in {A/4,B/3,C/2,D/1}
\draw (\x) node{$\y$};
\end{scope}

%	\draw[help lines=1] (-5,-5) grid (5,4);
%	\draw[fill] (0,0) circle [radius=2pt];
\begin{scope}\clip (-6, -4.75) rectangle (6,4);

\begin{scope}[xscale=-1]

		\begin{scope} %begin upper right circle and semicircle
	%	\draw[fill] (1.73,1) circle [radius=2pt];
		\begin{scope}\clip (5, -3) rectangle (2, 4);
		\draw[very thick] (1.73,1) circle [radius=2.73cm];
		\end{scope}\begin{scope}\clip (2, 4) rectangle (-1.5, 0.55);
		\draw[very thick] (1.73,1) circle [radius=2.73cm];
		\end{scope}\begin{scope}
		\clip (-1, 0.55) rectangle (2.7, -3);
		\draw[very thick,dotted] (1.73,1) circle [radius=2.73cm];
		\end{scope}\end{scope}% end upper right circle and semicircle
%%
		%begin upper left circle and semicircle
		\begin{scope} \clip (0, -2) rectangle (-5, 4);
		\draw[very thick,dashed] (-1.73,1) circle [radius=2.73cm];
		\end{scope}
		\begin{scope}
		\clip (0, -2) rectangle (4, 4);
		\draw[very thick,dashed] (-1.73,1) circle [radius=2.73cm];
		\end{scope}
		\begin{scope}[rotate=120]
	%	\draw[fill] (1.73,1) circle [radius=2pt];
		\begin{scope}
		\clip (0,-2) rectangle (-4,4);
		%\draw[very thick,color=orange] (0,1) ellipse [x radius=2.8cm,y radius=2.1cm];
		\end{scope}\end{scope}% end upper left circle and semicircle
		\begin{scope}[rotate=-120] %begin lower circle and semicircle
	%	\draw[fill] (1.73,1) circle [radius=2pt];
		\begin{scope} \clip (0, -2) rectangle (5, 4);
		\draw[very thick] (1.73,1) circle [radius=2.73cm];
		\end{scope}\begin{scope}
		\clip (0, -2) rectangle (-4, 4);
		\draw[very thick] (1.73,1) circle [radius=2.73cm];
		\end{scope}\end{scope}% end lower circle and semicircle
%
%\draw[font=\huge,color=brown] (0,0) node{B};

\end{scope}
\begin{scope}[rotate=-120]
\draw[very thick] (0,1) ellipse [x radius=2.8cm,y radius=2.1cm];
\end{scope}

\begin{scope}[rotate=120]\clip (0,-2) rectangle (-4,4);
\draw[very thick] (0,1) ellipse [x radius=2.8cm,y radius=2.1cm];
\end{scope}

\begin{scope}\clip (3, -0.35) rectangle (0.8, 3);
\draw[very thick,dotted] (0,0) circle [radius = 2.2cm];
\end{scope}

\begin{scope}\clip (3, -0.35) rectangle (-1.4, -3);
\draw[very thick,dotted] (0,0) circle [radius = 2.2cm];
\end{scope}

\begin{scope}\clip (0, -3.5)rectangle(2.5, 0);
\draw[very thick,dotted](0,0) ellipse [x radius = 2.2cm, y radius = 3.12cm];
\end{scope}

%\draw (4,3) node{$31$};
\draw (-4,3) node{$2$};
\draw (-2,-2.8) node{$42$};
\draw (-1.7,-0.5) node{$431$};
\draw (2,-4.5) node{$4$};
%\draw (0,2.2) node[above]{$\eta_{24}$};
%\draw (-2,2) node[right]{$\eta_{24}$};
%\draw (2,1) node[right]{$\eta_{34}$};
\draw (3.1, 0.2) node[right]{$314$};
\draw (2.1,1.3) node{$3142$};
\draw (0, -1.5) node[below]{$2431$};
\draw (3.75, -3.5) node{$43142$};

\draw[->] (3.2,-3.2)--(1.75,-2.25);

\end{scope}
\end{scope}
\end{tikzpicture}
\caption{The regular semi-invariant pictures for the quivers $4\rightarrow 1 \leftarrow 3 \rightarrow 4 \leftarrow 2 \leftarrow 1$ (left) and $4\rightarrow 3 \rightrightarrows 1 \rightarrow 4 \rightarrow 2$ (right). As in Figure~\ref{fig:mutation1}, each wall is labeled with the support (both vertices and arrows)} of its brick except the null wall, which is dashed. These pictures are piecewise-linearly isomorphic since the quivers are related by a mutation at the vertex 4. The shading of each (half-)wall in the second picture matches its preimage under the mutation formula. Note also that the left picture is the same as the right picture of Figure \ref{fig:mutation1} since the quivers are related by mutation at vertex 3 (vertex 4 in Figure \ref{fig:mutation1}) and $S(3)$ is not a regular module.\label{fig:mutation2}
\end{figure}
\end{center}
}

%%%%%%%%%%%%%%%%%%%%%%%%%%%%%%%%%%%%%%%%%%%%%%%%%%%%%%

\section{Discussion and future work}
It is unclear how the definition of support regular rigid and the results of Section \ref{sec:regular rigid} extend to the cluster-tilted case. If we consider, for example, the quiver $1 \rightarrow 3 \rightarrow 4 \leftarrow 1 \leftarrow 2 \leftarrow 3$ and its regular semi-invariant picture (Figure \ref{fig:mutation1}, right). The vertices lying in the corners of the regions/chambers now come in five types, only the first three of which appear in the hereditary case:
\begin{enumerate}
	\item $g$-vectors of regular modules, projected to $g(\eta)^\perp$: for example, the intersection of $D_0(23)$ with the null wall is the ray through $g_0(3)$.
	\item Projective vectors: for example, the intersection of $D_0(2)$ and $D_0(41)$ on the positive side of $D_0(3)$ is the ray through $p(3)$.
	\item Negative projective vectors: for example, the intersection of $D_0(2)$ and $D_0(41)$ on the negative side of $D_0(3)$ is the ray through $-p(3)$.
	\item Negative $g$-vectors of modules which are both regular and projective: for example, the intersection of $D_0(41)$ and $D_0(3)$ on the negative side of $D_0(2)$ is the ray through $(0,-1,0,0) = -g_0(214)$. The module $214$ is regular and equal to $P(2)$.
	\item $g$-vectors of non-regular modules which are sent to regular modules by $\tau$: for example, the intersection of $D_0(41), D_0(23)$, and $D_0(2341)$ is the ray through $g_0(34)$. The module $34$ is not regular since $g(\eta)\cdot \undim(34) = -1$. However, $\tau(34) \cong 2$, which is regular.
\end{enumerate}

The mutation formulas imply that the mutation of a support $\tau$-rigid object is again support $\tau$-rigid. Thus the appropriate generalization of support regular rigid must contain these five cases with a notion of compatibility preserved under the mutation. It remains an open problem to characterize such a notion.

%%%%%%%%%%%%%%%%%%%%%%%%%%%%%%%%%%%%%%%%%%%%%%%%%%%%%%

\bibliographystyle{amsalpha}
{\footnotesize
\bibliography{ClusterNullReferences}}

\end{document}